\documentclass{amsart}
\usepackage{amsmath,amsfonts,amssymb,amsthm,amsbsy}
\usepackage[all]{xy}
\usepackage{hyperref}
\numberwithin{equation}{section}
\newtheorem{theorem}[equation]{Theorem}
\newtheorem{proposition}[equation]{Proposition}
\newtheorem{definition}[equation]{Definition}
\newtheorem{conjectures}[equation]{Conjectures}
\newtheorem{conjecture}[equation]{Conjecture}

\newtheorem{lemma}[equation]{Lemma}
\newtheorem{corollary}[equation]{Corollary}
\newtheorem{notation}[equation]{Notation}

\theoremstyle{definition}

\newtheorem{example}[equation]{Example}

\makeatletter
\renewcommand{\part}{
\@startsection {part}{0}\z@ {\linespacing \@plus \linespacing }
{.5\linespacing }{\Large \bfseries \centering }}
\makeatother
\numberwithin{equation}{section}
\newcommand\cB{{\mathcal B}}
\newcommand\cC{{\mathcal C}}
\newcommand\cD{{\mathcal D}}
\newcommand\cE{{\mathcal E}}
\newcommand\cH{{\mathcal H}}
\newcommand\cI{{\mathcal I}}
\newcommand\cJ{{\mathcal J}}
\newcommand\cL{{\mathcal L}}
\newcommand\cO{{\mathcal O}}
\newcommand\cP{{\mathcal P}}
\newcommand\cS{{\mathcal S}}
\newcommand\bB{{\mathbf B}}
\newcommand\bG{{\mathbf G}}
\newcommand\bI{{\mathbf I}}
\newcommand\bJ{{\mathbf J}}
\newcommand\bK{{\mathbf K}}
\newcommand\bL{{\mathbf L}}
\newcommand\bP{{\mathbf P}}
\newcommand\bQ{{\mathbf Q}}
\newcommand\bS{{\mathbf S}}
\newcommand\bT{{\mathbf T}}
\newcommand\bU{{\mathbf U}}
\newcommand\bV{{\mathbf V}}
\newcommand\bW{{\mathbf W}}
\newcommand\bX{{\mathbf X}}
\newcommand\ba{{\mathbf a}}
\newcommand\bb{{\mathbf b}}
\newcommand\bc{{\mathbf c}}
\newcommand\bh{{\mathbf h}}
\newcommand\bs{{\mathbf s}}
\newcommand\bt{{\mathbf t}}
\newcommand\bu{{\mathbf u}}
\newcommand\bv{{\mathbf v}}
\newcommand\bw{{\mathbf w}}
\newcommand\bx{{\mathbf x}}
\newcommand\by{{\mathbf y}}
\newcommand\bz{{\mathbf z}}
\newcommand\us{{\underline s}}
\newcommand\uw{{\underline w}}
\newcommand\uB{{\underline B}}
\newcommand\BC{{\mathbb C}}
\newcommand\BF{{\mathbb F}}
\newcommand\BN{{\mathbb N}}
\newcommand\BQ{{\mathbb Q}}
\newcommand\BR{{\mathbb R}}
\newcommand\BZ{{\mathbb Z}}
\newcommand\Sgot{{\mathfrak S}}
\newcommand\inv{^{-1}}
\newcommand\bpi{{\boldsymbol\pi}}
\newcommand\bchi{{\boldsymbol\chi}}
\newcommand\lexp[2]{\kern\scriptspace\vphantom{#2}^{#1}\kern-\scriptspace#2}
\newcommand\Qlbar{{\overline\BQ_\ell}}
\newcommand\Zlbar{{\overline\BZ_\ell}}
\newcommand\conjd[2]{#1\xrightarrow{#2}\ud}
\DeclareMathOperator\GL{\mathrm{GL}}
\DeclareMathOperator\Obj{\mathrm{Obj}}
\DeclareMathOperator\ad{\mathrm{ad}}
\DeclareMathOperator\Conj{\mathrm{Conj}}
\DeclareMathOperator\cyc{\mathrm{cyc}}
\DeclareMathOperator\Id{\mathrm{Id}}
\DeclareMathOperator\End{\mathrm{End}}
\DeclareMathOperator\Hom{\mathrm{Hom}}
\DeclareMathOperator\Irr{\mathrm{Irr}}
\DeclareMathOperator\Ker{\mathrm{Ker}}
\DeclareMathOperator\lcm{\mathrm{lcm}}
\DeclareMathOperator\lS{\mathrm{lg}_\cS}
\DeclareMathOperator\Res{\mathrm{Res}}
\DeclareMathOperator\Ind{\mathrm{Ind}}
\DeclareMathOperator\Sh{\mathrm{Sh}}
\DeclareMathOperator\St{\mathrm{St}}
\DeclareMathOperator\bSt{\mathbf{St}}
\DeclareMathOperator\Trace{\mathrm{Trace}}
\DeclareMathOperator\supp{\mathrm{supp}}
\DeclareMathOperator\target{\mathrm target}
\DeclareMathOperator\source{\mathrm source}
\newcommand\LwF{{\bL_I^{t(\bw\phi)}}}
\newcommand\Fpbar{{\overline\BF_p}}
\newcommand\genby[1]{\langle #1\rangle}
\newcommand\Isom[1]{#1^\times}
\newcommand\rightad[2]{#2^{#1}}
\newcommand\CCCi{\Isom\cC}
\newcommand\eqir{=^\times}
\newcommand\FConj{F\text{-}\Conj}   
\newcommand\Fcyc{F\text{-}\cyc}   
\newcommand\DI{\cD^\cI}
\newcommand\nnode[1]{{\kern -0.6pt\mathop\bigcirc\limits_{#1}\kern -1pt}}
\newcommand\edge{{\vrule width10pt height3pt depth-2pt}}
\newcommand\vertbar[2]{\rlap{\kern4pt\vrule width1pt height17.3pt depth-7.3pt}
 \rlap{\raise19.4pt\hbox{$\kern -0.4pt\bigcirc\scriptstyle#2$}}
                 \nnode{#1}}
\newcommand\overbar[1]{\kern -1.5pt\mathop\edge\limits^{#1}\kern -2pt}
\author[F.~Digne]{Fran\c cois Digne}
\address{LAMFA, CNRS UMR 7352, Universit\'e de Picardie-Jules Verne}
\email{digne@u-picardie.fr}
\author[J.~Michel]{Jean Michel}
\address{IMJ, CNRS UMR 7586, Universit\'e Paris VII}
\email{jmichel@math.jussieu.fr}
\thanks{This work was partially supported by the ``Agence Nationale pour la Recherche''
project ``Th\'eories de Garside'' (number ANR-08-BLAN-0269-03)}
\newcommand\f[1]{{f\inv(Y_{\ge #1})}}
\newcommand\dash{{\hbox{---}}}
\newcommand\ud{\hbox{-}}
\newcommand\scal[3]{\langle#1,#2\rangle_{#3}}
\begin{document}
\title{Parabolic Deligne-Lusztig varieties.}
\begin{abstract}
Motivated  by the Brou\'e  conjecture on blocks  with abelian defect groups
for   finite  reductive  groups,  we  study  ``parabolic''  Deligne-Lusztig
varieties  and construct on those which  occur in the Brou\'e conjecture an
action of a braid monoid, whose action on their $\ell$-adic cohomology will
conjecturally  factor  through  a  cyclotomic  Hecke  algebra.  In  order to
construct  this action, we need to enlarge the set of varieties we consider
to  varieties attached  to a  ``ribbon category'';  this category has a {\em
Garside family}, which plays an important role in our constructions, so we
devote the  first  part  of  our  paper  to  the  necessary  background on
categories with Garside families.
\end{abstract}
\maketitle
\section{Introduction}
In this paper, we study ``parabolic'' Deligne-Lusztig varieties, one of the
main motivations being the Brou\'e conjecture on blocks with abelian defect
groups for finite reductive groups.

Let  $\bG$  be  a  connected  reductive  algebraic  group over an algebraic
closure $\Fpbar$ of the prime field $\BF_p$ of characteristic $p$. Let $F$ be
an  isogeny  on  $\bG$  such  that  some  power  $F^\delta$  is a Frobenius
endomorphism   attached  to  a  split   structure  over  the  finite  field
$\BF_{q^\delta}$;
this defines a positive real number $q$ such that $q^\delta$ is an
integral  power of  $p$. When  $\bG$ is  quasi-simple, any isogeny $F$ such
that the group of fixed points $\bG^F$ is finite is of the above form; such
a  group $\bG^F$ is called a ``finite reductive group'' or a ``finite group
of Lie type''.

Let  $\bL$ be an $F$-stable Levi subgroup of a (non necessarily $F$-stable)
parabolic  subgroup  $\bP$  of  $\bG$.  Then,  for  $\ell$  a  prime number
different  from $p$, Lusztig has  constructed a ``cohomological induction''
$R_\bL^\bG$   which  associates  with   any  $\Qlbar\bL^F$-module  a  virtual
$\Qlbar\bG^F$-module.  We study the particular case $R_\bL^\bG(\Id)$, which
is given by the alternating sum of the $\ell$-adic cohomology groups of the
variety  $$\bX_\bP=\{g\bP\in\bG/\bP\mid g\bP\cap F(g\bP)\ne\emptyset\}$$ on
which $\bG^F$ acts on the left. We will construct a monoid of endomorphisms
$M$  of  $\bX_\bP$  related  to  the  braid group, which conjecturally will
induce in some cases an action of a cyclotomic Hecke algebra on the cohomology of
$\bX_\bP$.  To construct  $M$ we  need to  enlarge the  set of varieties we
consider,  to  include  varieties  attached  to  morphisms  in  a  ``ribbon
category''  ---  the  ``parabolic Deligne-Lusztig varieties''
of this paper; $M$ corresponds to  the endomorphisms in the ``conjugacy
category'' of this ribbon category of the object attached to $\bX_\bP$.

The relationship with Brou\'e's conjecture for the principal block comes as
follows: assume, for some prime number $\ell\ne p$, that a Sylow $\ell$-subgroup
$S$  of $\bG^F$ is abelian. Then Brou\'e's conjecture \cite{Br1} predicts in this special
case  an equivalence of  derived categories between  the principal block of
$\Zlbar\bG^F$  and that of $\Zlbar  N_{\bG^F}(S)$. Now $\bL:=C_\bG(S)$ is a
Levi  subgroup of a  (non $F$-stable unless  $\ell|q-1$) parabolic subgroup
$\bP$; restricting  to unipotent  characters and
discarding  an eventual torsion  by changing coefficients  from $\Zlbar$ to
$\Qlbar$,  this translates after refinement (see \cite{BM})
into  conjectures about  the cohomology of $\bX_\bP$ (see
\ref{conjecture}); these conjectures  predict that the
image  in the  cohomology of  our monoid  $M$ is a cyclotomic Hecke algebra.

The  main feature of  the ribbon categories  we consider is  that they have
{\em  Garside  families}.  This  concept  has  appeared  in  recent work to
understand  the ordinary and dual monoids  attached to the braid groups; in
the  first part of this paper, we recall its basic properties and go as far
as  computing the centralizers  of ``periodic elements'',  which is what we
need in the applications. The reader who wants to avoid the general theory
of Garside families can try to read only Section \ref{application} where
we spell out the results in the case of Artin monoids.

In the second part, we first define the parabolic Deligne-Lusztig varieties
which  are  the  aim  of  our  study,  and  then  go  on to establish their
properties. We extend to this setting in particular all the material in
\cite{BM} and \cite{BR2}.

We  thank C\'edric Bonnaf\'e and Rapha\"el  Rouquier for discussions and an
initial  input which started  this work, and  Olivier Dudas for a careful
reading and many suggestions for improvement.

After this paper was written, we received a preprint from Xuhua He and Sian
Nie  (see \cite{he-nie}) where, amidst other interesting results, they also
prove Theorem \ref{bonne racine}.
\part*{I. Garside families}
This  part collects some prerequisites on  categories with Garside families.
It is mostly self-contained apart from the next section where the proofs are
omitted; we refer for them to \cite{Article} or the book \cite{Livre} in
preparation.
\section{Basic results on Garside families}\label{basic results}
Given a category $\cC$, we write $f\in\cC$ to say that $f$ is a morphism of
$\cC$,  and we write $\cC(x,y)$ (resp.\ $\cC(x,\ud)$, resp.\ $\cC(\ud,y)$)
for  the set of  morphisms from $x\in\Obj\cC$ to
$y\in\Obj\cC$ (resp.\ the set of morphisms with source $x$, resp.\ the set
of morphisms with target $y$).
We write $fg$ for the composition of $f\in\cC(x,y)$ and
$g\in\cC(y,z)$, and $\cC(x)$ for $\cC(x,x)$.
By $\cS\subset\cC$ we mean that $\cS$ is a set of morphisms in $\cC$.

Recall that a category  is  cancellative if each one of the
relations $hf=hg$ or $fh=gh$ implies $f=g$;
equivalently every morphism is a monomorphism and an
epimorphism.  We  say   that  $f$   left-divides  $g$,  or equivalently that
$g$ is a right-multiple of $f$, written
$f\preccurlyeq g$, if there exists $h$ such that $g=fh$; in this situation 
since the category is cancellative $h$ is uniquely defined by $g$ and $f$
and we write $h=f\inv g$.
Similarly  we say that  $f$ right-divides  $g$, or that $g$ is a left-multiple
of $h$ and write  $g\succcurlyeq f$ if there exists $h$ such that $g=hf$.

We denote by $\Isom\cC$ the set of invertible morphisms of $\cC$, and write
$f\eqir g$ if there exists $h\in\CCCi$ such that $fh=g$ (or equivalently
there exists $h\in\CCCi$ such that $f=gh$). 
\begin{definition}\label{famgar}In a cancellative category $\cC$
a Garside family is  a subset $\cS\subset\cC$ such that;
\begin{enumerate}
\item $\cS$ together with $\Isom\cC$   generates  $\cC$, and
$\Isom\cC\cS\subset \cS\Isom\cC\cup\CCCi$.
\item  For every  product $fg$  with $f,g\in\cS$,  we can write $fg=f_1g_1$
with $f_1, g_1\in\cS$ such  that if for $k\in\cC$    and $h\in\cS$ we have
$h\preccurlyeq  kfg$ then
$h\preccurlyeq kf_1$.
\end{enumerate}
\end{definition}
If  item (ii) of the above definition holds we say that the 2-term sequence
$(f_1,g_1)$ is an $\cS$-normal decomposition of $fg$. We extend this notion
first  to the case where $f,g\in\cS\Isom\cC\cup\CCCi$ by requiring the same
condition   but  with  $f_1,g_1\in\cS\Isom\cC\cup\CCCi$;   we  extend  then
$\cS$-normal    decompositions   to   longer   lengths   by   saying   that
$(x_1,\dots,x_n)$ is an $\cS$-normal decomposition of $x=x_1\dots x_n$ if
for each $i$ the sequence $(x_i,x_{i+1})$ is an $\cS$-normal decomposition.
We finally extend it to elements $x\in\cS\Isom\cC\cup\CCCi$ by saying that
$(x)$ is an $\cS$-normal decomposition.

In  a cancellative category with a  Garside family every element $x$ admits
an $\cS$-normal decomposition. We will just say ``normal decomposition'' if
$\cS$  is clear from the context. A normal decomposition $(x_1,\dots,x_n)$
is  {\em strict} if no entry is  invertible and all entries excepted possibly
$x_n$  are in $\cS$. In a cancellative  category with a Garside family every
non-invertible element admits a strict $\cS$-normal decomposition.

Normal decompositions are unique up to invertible elements, precisely
\begin{lemma}[{\cite[2.11]{Article}}] \label{deformation}
If   $(x_1,\dots,x_n)$   and   $(x'_1,\dots,x'_{n'})$  with $n\le n'$
are   two  normal decompositions  of  $x$  then  for  any $i\le n$ we have 
$x_1\dotsm x_i\eqir x'_1\dotsm x'_i$ and for $i>n$ we have $x'_i\in\CCCi$.
\end{lemma}

\subsection*{Head functions}

\begin{definition} 
Let  $\cC$ be a cancellative category  and let $\cS\subset\cC$. Then we say
that  a function $\cC-\CCCi\xrightarrow H\cS$  is an $\cS$-head function if
for  any $h\in\cS$, we have $ h\preccurlyeq g \Leftrightarrow h\preccurlyeq
H(g)$.
\end{definition}

We say that a subset $\cS\subset\cC$ is closed under right-divisor if
$f\succcurlyeq g$ with $f\in\cS$ implies $g\in\cS$.
We have the following criterion to be Garside:
\begin{proposition}[{see \cite[3.10 and 3.34]{Article}}]\label{critereGarside}
Assume  that  $\cC$  is  a  cancellative  category and that $\cS\subset\cC$
together  with  $\Isom\cC$  generates  $\cC$.  
Consider  the  following  property  for  an  $\cS$-head function: $$\forall
f\in\cC,\forall g\in\cC-\CCCi, H(fg)\eqir H(fH(g)).\leqno(\cH)$$ Then $\cS$
is  Garside if  there exists  an $\cS$-head  function satisfying  ($\cH$) or
there  exists an $\cS$-head function and $\cS\CCCi\cup\CCCi$ is closed under
right-divisor.  Conversely if $\cS$ is  Garside then $\cS\CCCi\cup\CCCi$ is
closed under right-divisor and any $\cS$-head function satisfies ($\cH$).
\end{proposition}
An   $\cS$-head  function  $H$   computes  the  first   term  of  a  normal
decomposition  in  the  sense  that  if  $(x_1,\dots,  x_n)$  is  a normal
decomposition   of  $x\in\cC-\CCCi$  then   $H(x)\eqir  x_1$.  Further  any
$x\in\cC-\CCCi$  has a strict normal decomposition $(x_1,\dots, x_n)$ with
$H(x)=x_1$.

Let $\cC$ be a cancellative category with a Garside family $\cS$.
For  $f\in\cC$  we  define  $\lS(f)$  to  be the minimum number $k$ of
morphisms  $s_1,\dots,s_k\in\cS$  such  that  $s_1\dotsm  s_k\eqir f$, thus
$\lS(f)=0$  if  $f\in\CCCi$;  if  $f\notin\CCCi$ then $\lS(f)$ is
also  the number of terms in a strict normal decomposition of $f$. 

The  following shows that $\cS$ ``determines''  $\cC$ up to invertible elements; we
say  that a subset  $\cC_1$ of $\cC$  is closed under  right-quotient if an
equality $f=gh$ with $f,g\in\cC_1$ implies $h\in\cC_1$.
\begin{lemma}[{\cite[VII 2.13]{Livre}}]\label{Garside in right-quotient closed subcategory}
Let $\cC_1$  be a subcategory of
$\cC$    closed   under   right-quotient   which   contains   $\cS$.   Then
$\cC=\cC_1\CCCi\cup\CCCi$ and $\cS$ is a Garside family in $\cC_1$.
\end{lemma}

\subsection*{Categories with automorphism}
Most  categories we  want to  consider will  have no non-trivial invertible
element,  which simplifies Definition  \ref{famgar}(i) to ``$\cS$ generates
$\cC$''.  The  only  source  of  invertible  elements  will be the following
construction.

An  automorphism of a category $\cC$ is  a functor $F:\cC\to \cC$ which has
an  inverse.  Given  such an  automorphism we define

\begin{definition}\label{semi-direct}
The semi-direct product category $\cC\rtimes\genby F$ is the category whose
objects  are the objects of  $\cC$ and whose morphisms are
the  pairs  $(g,F^i)$,  which  will  be  denoted  by $gF^i$, where $g\in\cC$
and $i$ is an integer. The source of $gF^i$ is $\source(g)$ and the
target of $gF^i$ is $F^{-i}(\target(g))$.
The  composition  rule  is  given  by  $gF^i\cdot hF^j=gF^i(h)F^{i+j}$ when
$\source(h)=\target(gF^i)$.
\end{definition} 
Note that we do {\em not} identify $(g,F^i)$ and $(g,F^j)$ even
when $F^{i-j}$ is the identity functor --- it will be convenient in our
semi-direct products to have the cyclic group generated by $F$ to be infinite
even though $F$ acts via a finite order automorphism.

The  conventions  on  $F$  are  such  that the composition rule is natural.
However,  they  imply  that  the  morphism  $(\Id,F)$  of the semi-direct product
category  represents the functor $F\inv$: it  is a morphism from the object
$F(A)$ to the object $A$ and we have the commutative diagram:
$$\xymatrix{F(A)\ar[r]^{F(f)}\ar[d]^F&F(B)\ar[d]^F\\A\ar[r]^f&B}$$

$\cC$ embeds in $\cC\rtimes\genby F$ by identifying $g$ and $(g,F^0)$.
\begin{lemma}[{\cite[VIII 1.34 (ii)]{Livre}}]\label{garside in CsF}
If  $\cS$ is a Garside  family in the cancellative category $\cC$, 
and  $F$ an automorphism of $\cC$ preserving $\cS$,  then $\cS$ is also a
Garside family in $\cC\rtimes\genby F$.
\end{lemma}

If $(f_1,\dots  f_k)$ is an $\cS$-normal
decomposition  of $f\in \cC$ then  $(f_1,\dots,f_kF^i)$ is an $\cS$-normal
decomposition  of $fF^i\in\cC\rtimes\genby  F$. 
Note  that if  $\cC$ has no
non-trivial invertible element, then the only invertible elements   in
$\cC\rtimes\genby  F$ are  $\{F^i\}_{i\in\BZ}$. In  general, if $a,b\in
\cC$ then $aF^i\preccurlyeq b F^j$ if and only if $a\preccurlyeq b$.

We have the following property

\begin{proposition}[{\cite[VII 4.4]{Livre}}] \label{fixed points} 
Assume that the cancellative category $\cC$ has a Garside family $\cS$ and has 
no non-trivial invertible morphisms. Let $F$ be an automorphism of
$\cC$ preserving $\cS$. Then the subcategory of fixed objects and 
morphisms $\cC^F$ has a
Garside family which consists of the fixed points $\cS^F$.
\end{proposition}

\subsection*{Gcds and lcms, Noetherianity}
We  call {\em right-lcm} of a family $\cC_1\subset\cC$ a right-multiple $f$
of  all morphisms in $\cC_1$ such  that for any other common right-multiple
$f'$  we have $f\preccurlyeq f'$; this corresponds to the categorical notion
of a pullback. Similarly  a {\em left-gcd} of the family
$\cC_1$  is  a  common  left-divisor  $f$  such  that  for any other common
left-divisor  $f'$ we have $f'\preccurlyeq f$; it corresponds to the notion of
a pushout. Left-lcms and right-gcds are
defined in the same way exchanging left and right.

The  existence of left-gcds and right-lcms  are related when the cancellative category
$\cC$ is right-Noetherian, which means that there is no infinite sequence 
$f_0\succcurlyeq f_1\succcurlyeq
\dots\succcurlyeq  f_n \succcurlyeq \cdots$  where $f_{i+1}$ is a {\em
proper}  right-divisor of  $f_i$, that is we do not have $f_i\eqir f_{i+1}$.
It means equivalently since  $\cC$ is cancellative 
that there is no infinite sequence $g_0\preccurlyeq g_1\preccurlyeq\dots
\preccurlyeq  g_n \preccurlyeq \dots\preccurlyeq g$  where $g_i$ is a
proper  left-divisor of  $g_{i+1}$. The equivalence is obtained by 
$f_i=g_i\inv g$ and $g=f_0$. In a right-Noetherian category any element is
right-divisible by an \emph{atom}, which is an element which cannot be written as the
product of two non-invertible elements. If the category is Noetherian (that
is, both left and right-Noetherian) we have:

\begin{proposition}[{\cite[II 2.64]{Livre}}]
A cancellative and Noetherian category is generated by its atoms and its invertible
elements.
\end{proposition}

We say that $\cC$ admits  conditional right-lcms if, whenever $f$ and $g$ have a
common right-multiple,  they have a right-lcm.
We then have:
\begin{proposition}[{\cite[II 2.41]{Livre}}]\label{lcm=>gcd} 
If $\cC$ is cancellative, right-Noetherian and admits conditional right-lcms,
then any family of morphisms of $\cC$ with the same source has a left-gcd.
\end{proposition}
If $\cC$ admits conditional right-lcms we say that a subset $X\subset\cC$
is  closed (resp.\ weakly closed)
under right-lcm if whenever two elements of $X$ have a right-lcm in
$\cC$ this lcm is in $X$ (resp.\ in $X\CCCi$). If further $X$ is closed under
right-quotient an lcm in $\cC$ which is in $X$ is also an lcm in $X$.
The following is proved in \cite[Proposition 3.25]{Article} (where there is a
Noetherianity assumption not used in the direct part of the proof).

\begin{lemma}\label{garside family lcm-closed}
If $\cS$ is a Garside family in a category which admits conditional right-lcms
then $\cS\CCCi$ is closed under right-lcm.
\end{lemma}

Here is a general situation when a Garside family of a subcategory can be
determined.
\begin{lemma}[{\cite[VII 1.10]{Livre}}]\label{Garside subfamily}
Let  $\cS$ be a Garside family  in $\cC$ assumed cancellative, 
right-Noetherian and having
conditional  right-lcms.  Let  $\cS_1\subset  \cS$  be  a  subfamily  such  that
$\cS_1\CCCi\cup\CCCi$ is as a subset of $\cS\CCCi\cup\CCCi$ closed under right-lcm and
right-quotient; then $\cS_1$ is a Garside family in the subcategory $\cC_1$
generated  by $\cS_1\CCCi$. Moreover $\cC_1$  is a subcategory closed under
right-quotient.
\end{lemma}

\begin{lemma}[{\cite[VII 1.18]{Livre}}]\label{alphaI}
Let  $M$ be a cancellative right-Noetherian monoid which admits conditional
right-lcms and let $M'$ be a submonoid of $M$ closed under right-quotient
and weakly closed under right-lcm.
Then any $u\in M$ has a unique (up to right-multiplication by
$M^{\prime\times}$) maximal left-divisor in $M'$.
\end{lemma}

\subsection*{Garside maps}
An  important special case  is when a Garside family
$\cS$  is attached to  a Garside map. A
Garside map is a map $\Obj\cC\xrightarrow\Delta\cC$ where
$\Delta(x)\in\cC(x,\ud)$ such that the map $x\mapsto\target(\Delta(x))$ is
injective and such that $\cS\CCCi\cup\CCCi$ is both the set of elements that
left-divide some $\Delta(x)$ and the set of elements that right-divide
some $\Delta(x)$.

This definition of a Garside map agrees with \cite[V 2.30]{Livre}
if we take in account that, using the notation of loc.\ cit., the fact that $\cS^\#$ is
the set of left- and right-divisors of $\Delta$ implies that the Garside family
$\cS$ is bounded.

A Garside map allows to  define a  functor $\Phi$,  first on  objects by taking for
$\Phi(x)$  the  target  of  $\Delta(x)$,  then  on morphisms, first on morphisms
$s\in\cS$  by, if $s\in\cC(x,\ud)$ defining $s'$ by $ss'=\Delta$ (we omit
the  source of $\Delta$ if it is clear from the context) and then $\Phi(s)$
by   $s'\Phi(s)=\Delta$.   We   then   extend   $\Phi$  by  using  normal
decompositions;  it can  be shown  that this  is well-defined and defines a
functor such that for any $f\in\cC$ we have $f\Delta=\Delta\Phi(f)$. It can
also be shown that the cancellativity of $\cC$ implies that $\Phi$ is
an automorphism.

The automorphism $\Phi$ is a typical automorphism of $\cC$ preserving 
$\cS$ that we  will call  the \emph{Garside automorphism}\index{Garside automorphism}.

If $\cS$ is attached to a Garside map, we then have the
following properties:

\begin{proposition}\label{second domino}
\begin{enumerate}
\item  If $f\preccurlyeq g$ then $\lS(f)\le\lS(g)$.
\item  Assume $f,g,h\in\cS$ and $(f,g)$ is $\cS$-normal; then
$\lS(fgh)\le 2$ implies $gh\in\cS\CCCi$.
\item   For  $f\in   \cC(x,\ud)$,  the   first  term   of  an  $\cS$-normal
decomposition of $x$ is a left-gcd of $f$ and $\Delta(x)$.
\end{enumerate}
\end{proposition}
\begin{proof}
(i) is \cite[V 2.39 (v)]{Livre}, (iii) is \cite[V 1.14]{Livre}.
(ii) is \cite[IV 1.38]{Livre} using \cite[2.15]{Livre} which says, with the
notation as in loc.\ cit., that $\cS^\#$ est left-comultiple-closed.
\end{proof}

We  will write $\Delta^p$ for the map which associates with an object $x$ the
morphism   $\Delta(x)\Delta(\Phi(x))\dotsm\Delta(\Phi^{p-1}(x))$.  For  any
$f\in\cC(x,\ud)$ there exists $p$ such that $f\preccurlyeq\Delta^p(x)$.

\begin{proposition}[{\cite[III 1.37 and V 2.14]{Livre}}]\label{power of Delta}
If $\cS$ is a Garside family attached to a Garside map $\Delta$
then  for any positive integer $p$, $\Delta^p$ is a Garside map and
$\{f_1f_2\cdots f_p\mid f_i\in\cS\}$ is a Garside family attached to $\Delta^p$.
\end{proposition}

\section{The conjugacy category}

The  context for  this section  is a  cancellative category $\cC$.

\begin{definition}
Given   a  category   $\cC$,  we   define  the  \index{conjugacy  category}
\emph{conjugacy  category} $\Conj \cC$  of $\cC$ as  the category whose objects
are  the  endomorphisms  of  $\cC$  and  where,  for  $w\in\cC(A)$ and
$w'\in\cC(B)$   we  set   $\Conj  \cC(w,w')=\{x\in\cC(A,B)\mid
xw'=wx\}$.  We  say  that  $x$  \emph{conjugates}  $w$  to  $w'$  and  call
\index{centralizer}  \emph{centralizer} of $w$ the set $\Conj \cC(w)$.
The  composition of  morphisms in  $\Conj\cC$ is  given by the composition in
$\cC$, which is compatible with the defining relation for $\Conj\cC$.
\end{definition}
Note that it is the formula for $\Conj\cC(w,w')$ that forces the objects
of $\Conj \cC$ to be endomorphisms of $\cC$.

Since  $\cC$  is  cancellative,  the  data  $x$ and $w$ determine $w'$
(resp.\ $x$ and $w'$ determine $w$). This
allows  us to write $\rightad xw$ for $w'$ (resp.\ $\lexp x w'$ for $w$);
this  illustrates that our category $\Conj\cC$ is a right-conjugacy category;
we call left-conjugacy category the opposed category.

A   proper  notation  for  an  element  of  $\Conj\cC(w,\ud)$  is  a  triple
$w\xrightarrow  x \rightad x w$ (that we will abbreviate often to $\conjd x
w$),  since $x$ by itself does not specify its source; but we will use just $x$
when  the context makes clear which source $w$ is meant (or which target is
meant). The forgetful functor which   sends  $w\in\Obj(\Conj   \cC)$  to   $\source(w)$  and
$\conjd w x$ to $x$ is  faithful, though not injective on objects; it
allows us to  identify $\Conj\cC(w,\ud)$ with the
subset  $\{x\in  \cC(\source(w),\ud)  \mid  x\preccurlyeq  wx\}$; similarly
we may identify $\Conj\cC(\ud,w)$ with the subset $\{x\in\cC(\ud,\source(w))\mid
xw\succcurlyeq x\}$.

It  follows that the  category $\Conj \cC$  inherits automatically from $\cC$
properties  such as cancellativity  or Noetherianity. The  forgetful functor maps
$\Isom{(\Conj\cC)}$  surjectively  to  $\CCCi$,  so  in particular the subset
$\Conj\cC(w,\ud)$    of    $\cC(\source(w),\ud)$    is   closed   under
multiplication  by $\CCCi$.  In the  proofs and  statements which follow we
identify  $\Conj\cC$ with a subset of  $\cC$ and $\Isom{(\Conj\cC)}$ to $\CCCi$; for the statements obtained about
$\Conj\cC$ to make sense, the reader has to check that the sources and target
of morphisms viewed as morphisms in $\Conj\cC$ make sense.

\begin{lemma}\label{FAd C is lcm}
\begin{enumerate}
\item  The subset $\Conj\cC$ of $\cC$ is closed under right-quotient.
\item  The subset $\Conj\cC(w,\ud)$ of  $\cC(\source(w),\ud)$ is closed under
right-lcm. In particular if $\cC$ admits conditional
right-lcms then so does $\Conj \cC$.
\end{enumerate}
Similarly  $\Conj\cC(\ud,w)$ is  a subset  of $\cC(\ud,\source(w))$ closed
under left-lcm and left-quotient.
\end{lemma}
\begin{proof}
We  show  (i).  If  $y=xz$ with 
$y\in \Conj\cC(w,w')$, $x\in\Conj\cC(w,\ud)$ and $z\in\cC(\ud,\source(w'))$
we have $x\preccurlyeq wx$ and $yw'=wy$. By cancellation, let us
define  $w''$ by  $xw''=wx$.
Now   since  $y=xz$   the  equality   $yw'=wy$  gives
$xzw'=wxz=xw''z$ which gives by cancellation that $zw'=w''z$ showing
that $z\in\Conj\cC(\ud,w')$.

We  now  show (ii). If $x,y\in  \Conj\cC(w,\ud)$  then
$x\preccurlyeq  wx$ and  $y\preccurlyeq wy$.  Suppose now  that $x$ and $y$
have  a right-lcm $z$ in $\cC$.  Then $x\preccurlyeq wz$ and $y\preccurlyeq
wz$   from  which  it  follows  that  $z\preccurlyeq  wz$,  that  is  $z\in
\Conj\cC(w,\ud)$,  thus $z$ is the image by the forgetful functor of a 
right-lcm of $x$ and $y$ in $\Conj \cC$.

The proof of the second part is just a mirror symmetry of the above proof.
\end{proof}
\begin{proposition}\label{FAd Garside}
Assume  that $\cS$ is a Garside family  in $\cC$; then $\Conj\cC\cap\cS$ is a
Garside  family in $\Conj\cC$ and $\cS$-normal decompositions of an element of
$\Conj\cC$ are $\Conj\cC\cap\cS$-normal decompositions.
\end{proposition}
\begin{proof}
We    will   use   Proposition   \ref{critereGarside}   by   showing   that
$(\Conj\cC\cap\cS)\cup\CCCi$  generates  $\Conj\cC$  and  exhibiting  a
$\Conj\cC\cap\cS$-head function
$H:\Conj\cC-\CCCi\to     \Conj\cC\cap\cS$    satisfying ($\cH$).

Let  $H$  be  a  $\cS$-head  function  in  $\cC$.  We  first  show that the
restriction  of $H$ to $\Conj\cC$ takes its values in $\Conj\cC\cap\cS$. Indeed
if   $x\preccurlyeq   wx$   then   $H(x)\preccurlyeq   H(wx)\eqir  H(wH(x))
\preccurlyeq wH(x)$ where the middle $\eqir$ is by ($\cH$). 

We  now  deduce  by  induction  on  $\lS$  that  $(\Conj\cC\cap\cS)\cup\CCCi$
generates   $\Conj\cC$.  The induction starts with elements of length 0 which
are exactly the elements of $\CCCi$.
Assume now  that $x\in \Conj\cC$ is such
that  $\lS(x)=n>0$ and define $x'$ by $x=H(x)x'$; since $H(x)$ can be taken as 
the first term of a strict normal decomposition we have $\lS(x')=n-1$.
Since we proved $H(x)\in \Conj\cC$,  we deduce  by Lemma  
\ref{FAd C  is lcm}(i)  that $x'\in \Conj\cC$, whence the result by induction. 

It  is  straightforward that  the  restriction  of $H$ to $\Conj\cC-\CCCi$ is
still a head function satisfying ($\cH$), 
which proves that $\Conj\cC\cap\cS$  is a Garside family. The
assertion about normal decompositions follows. \end{proof}

\subsection*{Simultaneous conjugacy}
A  straightforward generalization of the conjugacy category is the ``simultaneous
conjugacy  category'',  where  objects   are  families  of  morphisms
$w_1,\dots,w_n$   with  same  source  and  target,  and  morphisms  verify
$x\preccurlyeq  w_i x$ for all $i$.  Most statements have a straightforward
generalization to this case.

\subsection*{$F$-conjugacy}
We  want to consider ``twisted conjugation'' by an automorphism, which will
be  useful  for  applications  to  Deligne-Lusztig  varieties, but also for
internal  applications, with  the automorphism  being the  one induced by a
Garside map. 

\begin{definition}\label{F-centralizer}
Let $F$ be an automorphism of finite order of the category $\cC$.
We  define the $F$-conjugacy category of  $\cC$, denoted by $\FConj \cC$, as the
category  whose  objects  are  the  morphisms in some $\cC(A,F(A))$ and
where,   for  $w\in\cC(A,F(A))$   and  $w'\in\cC(B,F(B))$   we  set
$\FConj  \cC(w,w')=\{x\in \cC\mid xw'=wF(x)\}$. We  say that $x$ \emph{
$F$-conjugates}  $w$  to  $w'$  and  we  call  \emph{$F$-centralizer}  of a
morphism $w$ of $\cC$ the set $\FConj \cC(w)$.
\end{definition}

Note  that $F$-conjugacy  specializes to  conjugacy when  $F=\Id$; again,
it is the formula for $\FConj\cC(w,w')$ which forces the objects of
$\FConj\cC$ to lie in some $\cC(A,F(A))$.

The notion of $F$-conjugacy turns out to be a particular form of conjugacy in
the semi-direct product category $\cC\rtimes\genby F$; 
this is the same as the relation between
twisted conjugacy classes in a group and conjugacy classes in cosets.

Consider the application which sends 
$w\in\cC(A,F(A))\subset\Obj(\FConj\cC)$   to  $wF\in(\cC\rtimes\genby
F)(A)\subset\Obj(\Conj(\cC\rtimes\genby    F))$.   Since   $x(w'F)=(wF)x$   is
equivalent  to $xw'=wF(x)$, this extends to a functor $\iota$ from $\FConj\cC$ to
$\Conj(\cC\rtimes\genby  F)$. This functor is clearly an isomorphism onto its
image.

The image $\iota(\Obj(\FConj\cC))$ is  the subset of  $\cC\rtimes\genby F$ which
consists  of  endomorphisms  which lie  in  $\cC  F$;  and  $\iota(\FConj  \cC)$
identifies via the forgetful functor with the subset 
$\Conj(\cC\rtimes\genby F)\cap\cC$ of $\cC\rtimes\genby F$.

Remark that, since
in $\Conj(\cC\rtimes\genby F)$  there is no  morphism between $gF^i$ and
$g'F^j$ when $i\ne j$, the full subcategory that we will
denote by $\Conj(\cC F)$ of $\Conj(\cC\rtimes\genby F)$ whose objects are in $\cC
F$  is a union  of connected components  of $\Conj(\cC\rtimes\genby F)$; thus
many  properties will transfer automatically from $\Conj(\cC\rtimes\genby F)$
to $\Conj(\cC F)$.

In  particular, if $\cC$  has a Garside  family $\cS$ and  $F$ is a Garside
automorphism, then $\cS$ is still a Garside family for $\cC\rtimes\genby F$
by  \ref {garside  in CsF},  and by  Proposition \ref{FAd  Garside} and the
above remark gives rise to a Garside  family $\cS\cap\Conj(\cC F)$ of $\Conj(\cC F)$.
The image $\iota(\FConj\cC)$ is the subcategory of $\Conj(\cC F)$ consisting 
(via the forgetful functor) of
the   morphisms  in  $\cC$,   thus  satisfies  the   assumptions  of  Lemma
\ref{Garside  in  right-quotient  closed  subcategory}:  it is closed under
right-quotient, because in a relation $fg=h$ if $f$ and $h$ do not involve
$F$  the  same  must  be  true  for  $g$,  and  contains the Garside family
$\cS\cap\Conj(\cC F)$ of $\Conj(\cC F)$.

This   will  allow  to  generally   translate  statements  about  conjugacy
categories  to  statements  about  $F$-conjugacy  categories.  For example,
$\iota\inv(\cS\cap\Conj(\cC F))$ is a Garside family for $\FConj \cC$;
this last family is just $\FConj\cC\cap\cS$ when identifying $\FConj\cC$ with a
subset of morphisms of $\cC$ by the forgetful functor.

The  assumption that  $F$ acts  through an  automorphism of finite order is
used  as follows: since $(xF)^x=Fx=(xF)^{F\inv}$ and  the action of $F$ has
finite  order, two morphisms in $\cC  F$ are conjugate in $\cC\rtimes\genby
F$ if and only if they are conjugate by a morphism of $\cC$.

\subsection*{The cyclic conjugacy category}
A  restricted  form  of  conjugation  called  ``cyclic  conjugacy'' will be
important  in applications. In  particular, it turns  out (a particular case
of Proposition \ref{Ad=Cyc}) that two periodic
braids are conjugate if and only if they are cyclically conjugate.
The context for this subsection is again a cancellative category $\cC$.

\begin{definition}\label{cyclicFconjugacy}
We define the cyclic conjugacy category $\cyc\cC$ of $\cC$ as the subcategory
of $\Conj \cC$ generated by $\cS'=\cup_w\{x\in\Conj\cC(w,\ud)\mid x\preccurlyeq w\}$.
\end{definition}
That  is, $\cyc\cC$ has the  same objects as $\Conj  \cC$ but contains only the
products  of elementary conjugations of the form $w=xy\xrightarrow x yx$.
Note that since $\cC$ is cancellative
$\cup_w\{x\in\Conj\cC(w,w')\mid x\preccurlyeq w\}=
\{x\in\Conj\cC(\ud,w')\mid  w'\succcurlyeq x\}$ so cyclic conjugacy
``from  the left'' and ``from the right'' are  the same. To be
more precise, the functor which is the identity on objects, and when $w=xy$
and  $w'=yx$, sends $x\in\cyc\cC(w,w')$ to $y\in\cyc\cC(w',w)$,
is an isomorphism between $\cyc\cC$ and its opposed category.

\begin{proposition}\label{CFC}  Assume $\cC$ is right-Noetherian and admits
conditional  right-lcms; if $\cS$ is a  Garside  family  in  $\cC$  then 
$\cS'\cap\cS$ is a Garside family in $\cyc\cC$.
\end{proposition}
\begin{proof}
Set $\cS_1=\cS'\cap\cS$.
We   first  observe  that  $\cS_1\CCCi\cup\CCCi$   generates  $\cyc\cC$.  Indeed  if
$x\preccurlyeq  w$ and  we choose  a decomposition  $x=s_1\dotsm s_n$  as a
product  of  morphisms  in  $\cS\CCCi\cup\CCCi$  it  is  clear that each $s_i$ is in
$\cyc\cC$, so is in $\cS_1\CCCi\cup\CCCi$.

The proposition then results from Lemma \ref{Garside subfamily}, which
applies to $\cyc\cC$ since
$\cS_1\CCCi\cup\CCCi$  is closed under right-divisor  and right-lcm; this is obvious
for  right-divisor and for right-lcm results from the facts that
$\cS\CCCi\cup\CCCi$
is  closed  under  right-lcm  by Lemma \ref{garside family lcm-closed}
and  that a right-lcm of two divisors of $w$ is a divisor of $w$.
\end{proof}
We see by Lemma \ref{Garside subfamily} that $\cyc\cC$ is closed under
right-quotient in $\Conj\cC$.

We  now prove that $\cS'$ --- which does not depend on the existence of a Garside 
family $\cS$ in $\cC$ --- is a Garside  family attached to a Garside
map;  $\cS'$ is usually larger than the Garside family $\cS'\cap\cS$
of  Proposition \ref{CFC}, since it contains  all left-divisors of $w$ even
if $w$ is not in $\cS$.

\begin{proposition}\label{Fcyc Garside}
Assume  $\cC$ is right-Noetherian and admits
conditional right-lcms; then $\cS'$
is  a  Garside  family  in $\cyc\cC$ attached to the
Garside  map  $\Delta$  such  that  $\Delta(w)=  w\in\cyc\cC(w)$;  the
corresponding Garside automorphism $\Phi$ is the identity functor.
\end{proposition}
\begin{proof}
The set $\cS'$ generates $\cyc\cC$ by definition of $\cyc\cC$. It is closed
under  right-divisors since $xy\preccurlyeq w$ implies $x\preccurlyeq w$ so
that  $\rightad xw$ is defined and $y\preccurlyeq \rightad xw$; since $\cC$
is right-Noetherian and admits conditional right-lcms, any two morphisms of
$\cC$  with same source have a gcd by Proposition \ref{lcm=>gcd}. We define
a  function $H:\cyc \cC-\CCCi\to \cS'$ by  letting $H(x)$ be an arbitrarily
chosen  left-gcd  of  $x$  and  $w$  if $x\in\cyc\cC(w,\ud)$. It is readily
checked  that $H$  is an  $\cS'$-head function.  We conclude by Proposition
\ref{critereGarside} that $\cS'$ is a Garside family for $\cyc\cC$. The set
$\cS'(w,\ud)$  is  the  set  of  left-divisors  of $w=\Delta(w)$; similarly
$\cS'(\ud,w)$ is the set of right-divisors of $w=\Delta(w)$. Hence $\Delta$
is  a Garside map in $\cyc\cC$. The  equation $x\rightad xw=w x$ shows that
$\Phi$ is the identity.
\end{proof}

We say that a subset $X\subset\cC$ is  closed under left-gcd if whenever two
elements of $X$ have a left-gcd in $\cC$ this gcd is in $X$.
\begin{proposition}\label{CFC  stable  gcd}  Assume  $\cC$ is right-Noetherian
and admits conditional right-lcms; then the subcategory $\cyc\cC$ of $\Conj
\cC$ is closed under left-gcd.
\end{proposition}
\begin{proof}
Let  $(x_1,\dots, x_n)$  and $(y_1,\dots,  y_m)$ be $\cS'$-normal 
decompositions respectively of $x\in \cyc\cC(w,\ud)$ and
$y\in\cyc\cC(w,\ud)$.

We first prove that if $\gcd(x_1,y_1)\in\CCCi$ then $\gcd(x,y)\in\CCCi$ (here 
we consider left-gcds in $\Conj \cC$). We proceed by induction on $\inf\{m,n\}$.
We write $\Delta$ for $\Delta(w)$ when there is no ambiguity on the source $w$. 
Since $x_n$ and $y_m$ divide $\Delta$, we get that $\gcd(x,y)$ divides 
\begin{multline*}
\gcd(x_1\dotsm x_{n-1}\Delta,y_1\dotsm y_{m-1}\Delta)\eqir 
\gcd(\Delta x_1\dotsm x_{n-1},
\Delta y_1\dotsm y_{m-1})\\
\hfill\eqir\Delta\gcd(x_1\dotsm x_{n-1},y_1\dotsm y_{m-1})\eqir\Delta=w,
\end{multline*}
where the first equality uses that $\Phi$ is the identity and
the third results from the induction hypothesis.
So we get that $\gcd(x,y)$ divides $w$ thus is in $\cS'$; 
by the property of normal decompositions it thus divides
$x_1$ and $y_1$, thus is in $\CCCi$.

We now prove the proposition. If $\gcd(x_1,y_1)\in\CCCi$ then
$\gcd(x,y)\in\CCCi$ thus is in $\cyc\cC$ and we are done. Otherwise let $d_1$
be  a  gcd  of  $x_1$  and  $y_1$  and  let $x^{(1)},y^{(1)}$ be defined by
$x=d_1x^{(1)}$,  $y=d_1y^{(1)}$. Similarly let $d_2$ be  a gcd of the first
terms  of a normal decomposition of $x^{(1)}$, $y^{(1)}$ and let $x^{(2)}$,
$y^{(2)}$  be the remainders,  etc\dots Since $\cC$  is right-Noetherian the
sequence $d_1, d_1d_2, \dots$ of increasing divisors of $x$ must stabilize
at  some stage $k$, which means that the corresponding remainders $x^{(k)}$
and $y^{(k)}$ have first terms of their normal decomposition coprime, so by
the  first  part  are  themselves  coprime.  Thus $\gcd(x,y)\eqir d_1\dotsm
d_k\in \cyc\cC$.
\end{proof}

We now give a quite general context where cyclic conjugacy coincides with
conjugacy.
\begin{proposition}\label{Ad=Cyc}
Let $\cC$ be a right-Noetherian category with a Garside map
$\Delta$, and let $x$ be an endomorphism of $\cC$ such that for
$n$ large enough we have $\Delta\preccurlyeq x^n$. Then 
we have $\cyc\cC(x,\ud)=\Conj\cC(x,\ud)$.
\end{proposition}
\begin{proof}
We  first show  that the  property $\exists  n,\;\Delta\preccurlyeq x^n$ is
stable by conjugacy. Indeed, if $u\in\Conj\cC(x,\ud)$ then there exists $k$
such    that   $u\preccurlyeq\Delta^k$.   Since   $\Delta^{k+1}\preccurlyeq
x^{n(k+1)}$,    we    have    $u\inv\Delta^k\cdot\Delta\preccurlyeq   u\inv
x^{n(k+1)}$. If $\Phi$ is the Garside automorphism attached to $\Delta$, we
have     $u\inv\Delta^k\cdot\Delta=\Delta\cdot\Phi(u\inv\Delta^k)$     thus
$\Delta\preccurlyeq u\inv x^{n(k+1)}$. We deduce that
$(x^u)^{n(k+1)}=(u\inv   x\cdot  u)^{n(k+1)}=u\inv  x^{n(k+1)}\cdot  u$  is
divisible by $\Delta$.

We  prove then  by Noetherian  induction on  $f$ that $f\in\Conj\cC(x,\ud)$
implies $f\in\cyc\cC(x,\ud)$. This is true if $f$ is invertible. Otherwise,
write  $f=u_1f_1$ with $u_1=\gcd(f,x)$; then $u_1\in\cyc\cC(x,x^{u_1})$. If
we   can   prove   that   if  $f\in\Conj\cC(x,\ud)$,  $f\notin\CCCi$,  then
$\gcd(f,x)\notin\CCCi$,  we will be  done by Noetherian  induction since we
can  write similarly $f_1=u_2f_2,\dots$  and the sequence $u_1,u_2,\dots$
has to exhaust $f$.

Since as observed any $u\in\Conj\cC(x,\ud)$ divides some power of $x$
($x^{nk}$ if $u\preccurlyeq\Delta^k$) it is enough to show that 
if $u\in\Conj\cC(x,\ud)$, $u\notin\CCCi$ and $u\preccurlyeq x^n$, 
then  $\gcd(u,x)\notin\CCCi$. We do this by induction on $n$.
From $u\in\Conj\cC(x,\ud)$ we have $u\preccurlyeq xu$, and from 
$u\preccurlyeq x^n$ we deduce $u\preccurlyeq x\gcd(u,x^{n-1})$. If
$\gcd(u,x^{n-1})\in\CCCi$ then  $u\preccurlyeq x$ and we are done:
$\gcd(x,u)=u$. Otherwise let $u_1=\gcd(u,x^{n-1})$. We have
$u_1\preccurlyeq xu_1$, $u_1\notin\CCCi$ and $u_1\preccurlyeq x^{n-1}$
thus we are done by induction.
\end{proof}
\subsection*{The $F$-cyclic conjugacy}
Let  $F$ be a  finite order automorphism  of the category  $\cC$. We define
$\Fcyc \cC$ as the subcategory of $\FConj \cC$ generated by
$\cup_w\{x\in\FConj\cC(w,\ud)\mid  x\preccurlyeq  w\}$,  or  equivalently, since
$\cC$  is  cancellative, by $\cup_{w'}\{x\in\Conj\cC(\ud,w')\mid
w'\succcurlyeq F(x)\}$. By the functor $\iota$, the morphisms in
$\Fcyc\cC(w,w')$ identify with the morphisms in
$\cyc(\cC\rtimes\genby  F)(wF,w'F)$ which  lie in  $\cC$. To simplify
notation, we will denote by $\cyc\cC(wF,w'F)$ this last set of morphisms.
If $\cC$ is right-Noetherian and admits conditional right-lcms, then so does
$\cC\rtimes\genby  F$. If $\cS$  is a Garside family  in $\cC$ and $F$ is
an automorphism preserving $\cS$, and we
translate Proposition \ref{CFC} to the image of $\iota$ and then to $\Fcyc\cC$,
we get that $\cup_w\{x\in\FConj\cC(w,\ud)\mid x\preccurlyeq w\text{ and }
x\in \cS\}$ is a Garside family in $\Fcyc \cC$.

Similarly Proposition \ref{Fcyc Garside} says that the set
$\cup_w\{x\in\FConj\cC(w,\ud)\mid  x\preccurlyeq w\}$ is a Garside family
in  $\Fcyc  \cC$  attached  to  the  Garside map $\Delta$ which sends the
object  $w$ to  the morphism  $w\in\Fcyc\cC(w, F(w))$; the associated
Garside automorphism is the functor $F$. 

Finally Proposition \ref{CFC stable gcd} says that under the 
assumptions of Proposition \ref{Fcyc Garside} the subcategory $\Fcyc \cC$ of $\FConj \cC$
is closed under left-gcd.

\subsection*{Periodic elements}
\begin{definition}
Let  $\cC$  be  a  cancellative category  with a  Garside family $\cS$
attached to Garside map $\Delta$;  then an  endomorphism $f$  of $\cC$ is
said  to  be  \emph{$(d,p)$-periodic}  if  $f^d\in\Delta^p  \CCCi$ for some
positive integers $d,p$.
\end{definition}

Note  that if $f$ is $(d,p)$-periodic it is also $(nd,np)$-periodic for any
non-zero integer $n$; 
conversely, up to cyclic conjugacy, a $(nd,np)$-periodic element
is $(d,p)$-periodic, see \ref{Bestvina}.  
We call $d/p$ the \index{period}\emph{period} of $f$.
If the Garside automorphism $\Phi$ given by $\Delta$ is  of finite  order, 
then  a conjugate of a periodic
element is periodic of the same period, though the minimal pair $(d,p)$ may
change. 
Our  interest in periodic elements
comes mainly from the fact that one can describe their centralizers,
which is related to the fact that by
Proposition \ref{Ad=Cyc}  two periodic morphisms are conjugate if and
only  if they are cyclically conjugate.

We deal first with the case $p=2$, where 
we show by elementary computations the following:

\begin{lemma}\label{69g}
Let  $f$ be  a $(d,2)$-periodic  element of  $\cC$ and  let $e=\lfloor\frac
d2\rfloor$.  Then $f$ is cyclically conjugate to a $(d,2)$-periodic element
$g$ such that $g^e\in \cS\CCCi$.

Further,  if $g$ is a $(d,2)$-periodic element such that $g^e\in \cS\CCCi$,
then
\begin{itemize}
\item if $d$  is even  $g$ is  $(d/2,1)$-periodic.
\item if $d$ is odd and we
define  $h\in  \cS\CCCi$  by  $g^eh=\Delta$  and  $\varepsilon\in\CCCi$  by
$g^d=\Delta^2\varepsilon$ then $g=h\Phi(h)\varepsilon$.
\end{itemize}
\end{lemma}
\begin{proof}

We  will prove  by increasing  induction on  $i$ that  for $i\le d/2$ there
exists $v\in\cyc\cC$ such that $(\rightad vf)^i\in \cS\CCCi\cup\CCCi$ and $(\rightad
vf)^d\in\Delta^2 \CCCi$. We start the induction with $i=0$ where the result
holds trivially with $v=1$.

We  consider now the  general step: assuming  the result for  $i$ such that
$i+1\le  d/2$, we will prove it for $i+1$. We thus have a $v$ for step $i$,
thus  replacing $f$  by $\rightad vf$ we  may assume that $f^i\in
\cS\CCCi\cup\CCCi$ and $f^d\in\Delta^2\CCCi$; 
we will conclude by finding $v\in \cS$
such  that  $v\preccurlyeq  f$, $(\rightad  vf)^{i+1}\in \cS\CCCi$ and
$(\rightad  vf)^d\in\Delta^2 \CCCi$. If $f^{i+1}\preccurlyeq\Delta$ we have
the  desired result with  $v=1$. We may  thus assume that $\lS(f^{i+1})\ge
2$.  Since $f^{i+1}\preccurlyeq\Delta^2$ we have actually $\lS(f^{i+1})=2$
by Proposition  \ref{second  domino}(i);  since $f^i$ is in $\cS\CCCi$
and divides $f^{i+1}$, a normal decomposition of $f^{i+1}$ can be written
$(f^iv,w)$ with $f^iv, w\in\cS\CCCi$. 
As  $f^ivw\cdot  f^iv\preccurlyeq  f^ivw\cdot  f^ivw=f^{2(i+1)}\preccurlyeq
f^d\eqir\Delta^2$, we still have $2=\lS(f^iv\cdot w\cdot
f^iv)=\lS(f^iv\cdot  w)$.  By  Proposition  \ref{second domino}(ii) we thus
have $w\cdot f^iv\in \cS\CCCi$. Then $\cS\CCCi\owns w
\cdot f^iv=w(vw)^iv=(\rightad vf)^{i+1}$ and $v\preccurlyeq f$.

So  $v$ will  do if $(\rightad vf)^d\in \Delta^2\CCCi$. Write
$f^d=\Delta^2\varepsilon$  with  $\varepsilon\in\CCCi$;  then $f$
commutes   with  $\Delta^2\varepsilon$,   thus  $f^{i+1}$   also,  which can
be written $\Phi^2(f^{i+1})\varepsilon=\varepsilon     f^{i+1}$     or   
equivalently $\Phi^2(f^iv)\Phi^2(w)\varepsilon=\varepsilon f^iv w$. 
Now since $\Phi$ preserves normal decompositions
$(\Phi^2(f^iv),\Phi^2(w)\varepsilon)$  is a normal
decomposition  thus comparing with $(f^iv,w)$  by
Lemma \ref{deformation}    there    exists    $\varepsilon'\in\CCCi$   such   that
$\Phi^2(f^iv)\varepsilon'= \varepsilon f^iv$. Thus
$f^i\Delta^2\Phi^2(v)\varepsilon'=\Delta^2\Phi^2(f^iv)\varepsilon'=
\Delta^2\varepsilon   f^iv=f^i\Delta^2\varepsilon  v$,  the  last  equality
using again that $f$  commutes with  $\Delta^2\varepsilon$. 
Canceling  $f^i\Delta^2$ we
get   $\Phi^2(v)\varepsilon'=\varepsilon  v$, whence $v(\rightad
vf)^d=f^d       v=\Delta^2\varepsilon      v=\Delta^2\Phi^2(v)\varepsilon'=
v\Delta^2\varepsilon'$ whence the result by canceling $v$.

We  prove now  the second  part. Since  $g^e\in \cS\CCCi$  the element $h$ 
defined by $g^eh=\Delta$ is in $\cS\CCCi\cup\CCCi$. Defining $\varepsilon\in\CCCi$ 
by $g^d=\Delta^2\varepsilon$ we get
$g^eh\Delta\varepsilon=\Delta^2\varepsilon=g^d$,   whence  by  cancellation
$h\Delta\varepsilon=g^eg^a$  with $a=1$ if  $d$ is odd  and $a=0$ if $d$ is
even.   Using  $h\Delta\varepsilon=\Delta\Phi(h)\varepsilon=
g^eh\Phi(h)\varepsilon$ and canceling $g^e$ we get $h\Phi(h)\varepsilon=g^a$.

If  $d$ is odd we get the statement of  the lemma, and if $d$ is even we get
$h\Phi(h)\in\CCCi$, so $h\in\CCCi$, so $g^e\in\Delta\CCCi$.
\end{proof}

We  will  need  at  one  stage  the  following  more general statement (see
\cite[VIII,  3.33]{Livre}) whose proof uses an interpretation by Bestvina
of normal decompositions as geodesics.
\begin{theorem}\label{Bestvina}
Let $f_1$ be a $(d_1,k_1)$-periodic element of $\cC$; let
$d=d_1/\gcd(d_1,k_1)$  and $k=k_1/\gcd(d_1,k_1)$; then  $f_1$ is cyclically
conjugate  to  a  $(d,k)$-periodic  element  $f$.  Further,  write an equality
$dk'=1+kd'$  in positive  integers. Then  $f$ is  cyclically conjugate to a
$(d,k)$-periodic element $g$ such that $g^{d'}\preccurlyeq \Delta^{k'}$. If
we    then   define   $g_1\in\cC$    by   $g^{d'}g_1=\Delta^{k'}$   then
$(g_1\Phi^{k'})^d\eqir\Delta$ and $(g_1\Phi^{k'})^k\eqir g$ in
$\cC\rtimes\genby\Phi$.
\end{theorem}
\subsection*{$F$-periodic elements}
Let  us apply Lemma \ref{69g} to the case of a semi-direct product category
$\cC\rtimes\genby  F$ where $\cC$ is a cancellative category with a Garside
family  $\cS$ attached to a Garside map $\Delta$ and $F$ is an automorphism
of  finite order of $\cC$  preserving $\cS$; then $\cS$  is still a Garside
family  of  $\cC\rtimes\genby  F$.  We  assume  further  that  $\cC$ has no
non-trivial   invertible  elements.   Then  a   morphism  $yF\in\cC  F$  is
$(d,p)$-periodic    if   and   only   if   $\target(y)=F(\source(y))$   and
$(yF)^d=\Delta^p F^d$.

From Lemma \ref{69g} we can deduce:
\begin{corollary}\label{70g}
Let $yF\in\cC F$ be $(d,2)$-periodic and let $e=\lfloor\frac d2\rfloor$ and
$\Lambda=\Phi F^{-e}$. Then
\begin{enumerate}
\item  If $d$  is even,  there exists an $(e,1)$-periodic element
$xF\in\cC F$ cyclically conjugate to
$yF$. The centralizer of  $xF$ in $\cC$ identifies with
$\cyc\cC(xF)$.  Further, we may compute this centralizer in the category of
fixed  points $(\cyc\cC)^\Lambda$ since the  morphisms in $\cyc\cC(xF)$ are
$\Lambda$-stable.
\item  If $d$ is odd, there exists a $(d,2)$-periodic element
$xF\in\cC F$ cyclically conjugate to $yF$
such that $(xF)^e\preccurlyeq\Delta F^e$. The element
$s$  defined  by  $(xF)^es=\Delta  F^e$  is  such  that,  in  the  category
$\cC\rtimes\genby\Lambda$,    we    have    $x\Lambda^2=(s\Lambda)^2$   and
$(s\Lambda)^d=\Delta   \Lambda^d$.  The   centralizer  of   $xF$  in  $\cC$
identifies   with  the $F^d\Phi^{-2}$-fixed points of $\cyc\cC(s\Lambda)$. 
\end{enumerate}
\end{corollary}
Note that \ref{fixed points} describes Garside families for the
fixed point categories mentioned above.
\begin{proof}
Lemma   \ref{69g}   shows   that   $yF$   is   cyclically  conjugate  to  a
$(d,2)$-periodic element $xF$ such that $(xF)^e\in\cS F^e$.

If $d$ is even Lemma \ref{69g} says  that  $xF$  is  $(e,1)$-periodic,  and
Proposition   \ref{Ad=Cyc}   says   that   the   centralizer   of  $xF$  is
$\cyc\cC(xF)$. The elements of this centralizer, commuting to $xF$, commute
to $(xF)^e=\Delta F^e$ thus are $\Phi\inv F^e$-stable.

If   $d$  is  odd  Lemma  \ref{69g}  says  that  if  $(xF)^eh=\Delta$  then
$xF=h\Phi(h)F^d$. Since $h=sF^{-e}$ we get
$x=sF^{-e}\Phi(sF^{-e})F^{d-1}=s\Lambda(s)$.    This   can   be   rewritten
$x\Lambda^2=(s\Lambda)^2$.  Now  since  $\Delta  F^e  s\inv=(xF)^e$  we get
$(\Delta   F^e   s\inv)^d=\Delta^{2e}F^{de}$   which   gives  $(\Lambda\inv
s\inv)^d\Delta^d=\Delta^{2e}\Lambda^{-d}$ and finally
$(s\Lambda)^d=\Delta\Lambda^d$.  The elements of  $\Conj\cC(xF)$ commute to
$(xF)^e=\Delta F^e s\inv$ thus commute to $s\Lambda$ thus
$\Conj\cC(xF)\subset\Conj\cC(s\Lambda)$.   Note   that   the   elements  of
$\Conj\cC(xF)$   commute   to   $(xF)^d$   thus  to  $F^d\Phi^{-2}$.  Using
$x\Lambda^2=(s\Lambda)^2$ we get
$\Conj\cC(s\Lambda)\subset\Conj\cC(x\Lambda^2)$; but
$x\Lambda^2=xF(F^d\Phi^{-2})\inv$ so
$\Conj\cC(x\Lambda^2)^{F^d\Phi^{-2}}\subset\Conj\cC(xF)$, whence the result
using that by Proposition \ref{Ad=Cyc} we have
$\Conj\cC(s\Lambda)=\cyc\cC(s\Lambda)$.
\end{proof}
We will apply \ref{Bestvina} in the following particular form
\begin{corollary}\label{F-Bestvina}
Assume  that  $F$  is  of  finite  order  and that $\Phi=\Id$. 
Then any periodic element of $\cC F$ is conjugate to a $(d,k)$-periodic
element $yF\in\cC  F$  where  $k$  is  prime  to  $d$.  
Further for any choice of positive integers $d'$ and $k'$
with   $dk'=1+kd'$, the element  $yF$   is   cyclically  conjugate  to  a
$(d,k)$-periodic    element    $xF$    satisfying    $(xF)^{d'}\preccurlyeq
\Delta^{k'}$.   If  we   then  define   $x_1\in\cC$  by  $(xF)^{d'}  x_1
F^{-d'}=\Delta^{k'}$  then  $(x_1  F^{-d'})^d=\Delta  F^{-dd'}$  and  $(x_1
F^{-d'})^k=(x F) F^{-k'd}$.
\end{corollary}
We have a partial converse:
\begin{lemma}\label{x1 periodic implies x periodic}
Assume that $F$ is of finite order and that $\Phi=\Id$. Let $d,k,d',k'$
be positive integers such that $dk'=1+kd'$ with $d'$ prime to the order of $F$.
If $x_1\in\cC$ satisfies $(x_1F^{-d'})^d=\Delta F^{-dd'}$ then the element 
$xF\in\cC F$ defined by $(x_1F^{-d'})^k=(xF)F^{-k'd}$  
satisfies $(xF)^d=\Delta^kF^d$.
\end{lemma}
\begin{proof}
The element $x_1F^{-d'}$ is $F^{-dd'}$-stable since $\Phi=\Id$
and $(x_1  F^{-d'})^d=\Delta  F^{-dd'}$.
Since $d'$ is prime to the order of $F$ an element $F^{-dd'}$-stable is
$F^d$-stable. Thus, raising the equality $(x_1F^{-d'})^k=(xF)F^{-k'd}$ to the 
$d$-th power
we get $(xF)^d=(x_1F^{-d'})^{dk}F^{k'd^2}=(\Delta
F^{-dd'})^kF^{k'd^2}=\Delta^kF^d$.
\end{proof}
The following lemma shows that we can always choose $d'$ satisfying the assumption
of lemma \ref{x1 periodic implies x periodic}.
\begin{lemma}\label{pas Dirichlet}Given $k$ and $d$ coprime natural integers,
and an integer $\delta$,
there exists natural integers $d',k'$ such that $dk'=1+kd'$ with $d'$ prime to
$\delta$.
\end{lemma} 
\begin{proof}  $k'$ and $d'$ exist since $k$ and $d$ are coprime; we  may  change  $d'$  by  any  multiple  of $d$. Thus it is
sufficient  to show that given coprime integers  $d$ and $d'$,
we may choose $a$ such that $d'+ad$ is prime to any given $\delta$. Let
$p_1,\ldots,p_n$  be the prime  factors of $\delta$.  We have to choose $a$
such  that $d'+ad$ is nonzero mod each $p_i$. If $p_i|d$ this is automatic.
If  $p_i$ is prime to
$d$ we have to avoid $a\equiv-d'/d \pmod{p_i}$;  by the
Chinese remainder theorem
we can choose $a$ to avoid this finite set of congruences.
\end{proof}

\section{An example: ribbon categories}\label{ribbon cat}

An  example of a category with a Garside family is  a {\em Garside monoid},
which is just the case where $\cC$ has one object. In this case we will say
Garside element instead of Garside map. 

\begin{example}\label{artin monoids}
A classical example is given by the
{\em Artin monoid} $(B^+,\bS)$ associated with a Coxeter system $(W,S)$.
If the presentation of $W$ is
$$W=\langle S\mid s^2=1, 
\underbrace{sts\dotsm}_{m_{s,t}}=
  \underbrace{tst\dotsm}_{m_{s,t}}\text{ for $s,t\in S$}\rangle$$
then $B^+$ is defined by the presentation
$B^+=\langle \bS\mid
\underbrace{\bs\bt\bs\dotsm}_{m_{s,t}}=
  \underbrace{\bt\bs\bt\dotsm}_{m_{s,t}}\text{ for $\bs,\bt\in \bS$}\rangle$
where  $\bS$ is a copy of $S$; the  group with the same presentation is the
Artin  group  $B$.  There  is  an  obvious  quotient  $B^+\to  W$ since the
relations  of $B^+$ hold in $W$. Matsumoto's lemma stating that two reduced
expressions  for  an  element  of  $W$  can  be related by using only braid
relations implies that there is a well-defined section $W\mapsto\bW$ of the
quotient  $B^+\to W$ which maps a reduced expression $s_1\dotsm s_n$ to the
product  $\bs_1\dotsm\bs_n\in  B^+$.  The  monoid  $B^+$  is  cancellative,
Noetherian,  admits conditional left-lcms and  right-lcms; the set $\bS$ is
the  set of  atoms of  $B^+$ and  $\bW$ is  a Garside  family in $B^+$ (for
details, see \cite[6.27]{Article}). The Garside family $\bW$ is attached to
a  Garside element if and only if $W$ is finite. In this case we call $B^+$
{\em  spherical}. The Garside element  is the lift to  $\bW$ of the longest
element  $w_0$ of $W$; it will be  written $\bw_0$ or $\Delta$ depending on
the context.

Finally, an automorphism $\phi$ of $(W,S)$ (that is, an automorphism of $W$
which  preserves $S$) extends  naturally to an  automorphism of $(B^+,\bS)$
given by $\bs\mapsto\phi(\bs)$ which preserves the Garside family $\bW$.
\end{example}

\begin{example}
Another  example, attached to the  same Artin braid group  $B$ as the above
example,  is the  {\em dual  braid monoid}  introduced by David Bessis (see
\cite{bessis1}),  whose  construction  can  be  extended  to well-generated
finite complex reflection groups.
\end{example}
The constructions of this section apply to the study, in the semi-direct
product of an Artin monoid
$(B^+,\bS)$ by an automorphism stabilizing $\bS$, of the
conjugates and normalizer of a ``parabolic'' submonoid --- the
submonoid generated by a subset of the atoms $\bS$.
The ``ribbon category'' that we consider occurs, when the automorphism is the
identity, in the work
of  Paris \cite{paris} and Godelle \cite{godelle} on this topic. In Section
\ref{section  8} we will attach  parabolic Deligne-Lusztig varieties to
elements  of the ribbon  category and endomorphisms  of these varieties to
elements of the conjugacy category of this ribbon category.

The next proposition gives a list of properties that spherical Artin monoids
satisfy; the rest of the section describes ribbons in
an arbitrary monoid satisfying the same properties, which includes the
case of the dual braid monoid; this is a motivation for giving the results
in a more general context. Before stating this
proposition, we need a definition.
\begin{definition}
We  say that  a set  $\bI$ of  atoms of  a cancellative  monoid $M$ is {\em
parabolic}  if the  submonoid $M_\bI$  of $M$  generated by $\bI$ is closed
under right-quotient and weakly closed under right-lcm.
\end{definition}
Note that a monoid generated by a set $\bI$ of atoms has no  non-trivial
invertible elements, since such an element would be a product of atoms and an
atom is not invertible. Similarly, since an atom cannot be a product of
several atoms, we see that $\bI$ is the whole set of atoms of the monoid.
\begin{proposition}\label{artin}
Let  $M=B^+\rtimes\genby\phi$ be the semi-direct product of a spherical  Artin
monoid by a  diagram automorphism (see \ref{artin monoids}); then 
\begin{enumerate}
\item \label{artin 0}
$M$ is cancellative, right-Noetherian and admits conditional right-lcms.
\item  \label{artin(i)} There exists a finite set $\bS\subset M$ which is a
transversal  of  the  $\eqir$-classes  of  atoms  in $M$, and together with
$M^\times$ generates $M$.
\item\label{artin(ii)} Any conjugate in $M$ of an element of $\bS$ is in $\bS$.
\item\label{artin(iv)a} $M$ has a Garside family $\cS$ attached to a Garside
element $\Delta$.
\item\label{artin(iii)}  For  any  parabolic  subset  $\bI$  of  $\bS$, the
maximal divisor $\Delta_\bI$ of $\Delta$ given by Lemma \ref{alphaI} (which
is  unique  since  $M_\bI^\times=\{1\}$)  is  a Garside element in $M_\bI$,
and $\cS\cap M_\bI$ is a Garside family attached to $\Delta_\bI$.
\item\label{artin(v)}  For  any  parabolic  subset  $\bI\subset\bS$ and any
$\bs\in\bS-\bI$   there  exists   a  parabolic   subset  $\bJ$   such  that
$\Delta_\bJ$ is the right-lcm of $\bs$ and $\Delta_\bI$.
\end{enumerate}
\end{proposition}
\begin{proof} 
Let  us prove (i).  The monoid $M$  is cancellative since  it embeds in the
semi-direct  product of the  Artin group by  $\phi$. Similarly it inherits
from  $B^+$ Noetherianity and the Garside  family $\bW$, which implies that
it admits conditional right-lcms.

We  prove (ii). Take for $\bS$ the set of atoms of $M$.
An invertible element must  have length 0, hence the powers
of  $\phi$ are the only invertible elements.  The atoms are the elements of
length 1 that is the elements of $\bS\genby\phi$, thus $\bS$ is indeed a
transversal of the atoms.

For  (iii), we have to check that if  we have $\bs f=f\bt$ with $\bs\in\bS$ and
$f$  and $\bt$ in $M$ then $\bt\in\bS$.  Taking lengths we see that the length
of  $\bt$  is  1  so  that  $\bt=\bs'\phi^k$  for  some  integer  $k$ and some
$\bs'\in\bS$. Looking then at the powers of $\phi$ on both sides we get $k=0$.

For  (iv),  take $\Delta=\bw_0$.
We  have  seen  in  Example \ref{artin monoids} that (using the
notation  of loc.\ cit.) the  lift $\bw_0$ to $\bW$  of the longest element
$w_0$  of $W$  is a  Garside element  in $B^+$.  Hence $\Delta=\bw_0$  is a
Garside element in $M$ by Lemma \ref{garside in CsF}. We take $\cS=\bW$; it is
a Garside family attached to $\Delta$.

For  (v) we  notice first that  $M_\bI$, 
being generated  by atoms,  has no non-trivial invertible elements.

Before proving the rest, let us state the following (the fact that this
fails in dual braid monoids is a motivation for defining parabolic
subsets).
\begin{lemma}\label{all parabolic} Any subset of $\bS$ is parabolic.
\end{lemma}
\begin{proof}
We  show  that  $M_\bI$ is closed under
right-quotient.  Since both  sides of  each defining  relation for an Artin
monoid  involve the  same elements  of $\bS$,  two equivalent  words for an
element  $v\in M$ involve the  same subset of the  generating set $\bS$; we
call  this subset the  {\em support} of  $v$. Hence if  $xy=z$ with $x,z\in
M_\bI$  then the power of $\phi$ in $y$ is  $0$ and the support of $y$ is a
subset  of that of $z$, thus a subset  of $\bI$, thus $y$ is in $M_\bI$. 

We  now show  that $M_\bI$  is weakly  closed under right-lcms. Keeping the
notations of \ref{artin monoids}, $B^+$ is associated to the Coxeter system
$(W,S)$. Since $M_\bI$ is a spherical Artin monoid associated with
the  Coxeter subgroup $W_I$ of  $W$ generated by the  image in $W$ of $\bI$
(see  for  example  \cite[3.1]{paris})  two  elements  of  $M_\bI$  have  a
right-lcm  in $M_\bI$.  This right-lcm  is left-divisible  by any  of their
right-lcms in $M$, so has to be equal to one of these lcms since $M_\bI$ is
obviously stable by left-divisor.
\end{proof}

Since  by $M_\bI$  is a  spherical Artin  monoid it  has a  Garside element
$\bw_\bI$,  the lift  of the  longest element  of $W_I$.  The corresponding
Garside  family is $\bW_\bI=\bW\cap M_\bI$, that  is the set of divisors in
$M_\bI$   of  $\Delta$  which   by  definition  of   $\Delta_\bI$  are  the
left-divisors  of $\Delta_\bI$. We get that $\bw_\bI$ and $\Delta_\bI$ have
the same set of left-divisors, so are equal since $M_\bI^\times=\{1\}$.

We  finally show  (vi). We  take $\bJ=\bI\cup\{\bs\}$.  The following lemma
applied  with $\bS=\bI$  (resp.\ $\bS=\bJ$) gives  that $\Delta_\bI$  is a
right-lcm of $\bI$ (resp.\ $\Delta_\bJ$ is a right-lcm of $\bJ$).
We thus get the result by associativity of the lcm.
\begin{lemma}
The Garside element $\Delta=\bw_0$ of $B^+$ is the right-lcm of $\bS$.
\end{lemma}
\begin{proof}
By \cite[6.27]{Article} a common multiple of $\bS$
in $\bW$ corresponds to an element $w\in W$ such that $l(sw)<l(w)$ for all
$s\in S$. It is well known that only $w_0$ satisfies this, so $\Delta=\bw_0$
is the only element of $\bW$ multiple of all the atoms.
\end{proof}
\end{proof}

\subsection*{The category $\Conj(M,\cI)$.}
Until  the  end  of  Section  \ref{ribbon  cat},  we fix a monoid $M$ and a
transversal  $\bS$ of its  set of atoms;  we assume that  $M$ has a Garside
family  $\cS$ associated with a Garside  element $\Delta$ so that these data
satisfy properties (i) to (vi) of Proposition \ref{artin}.

The reader only interested in internal  applications to this paper can 
assume that we are in the case
$M=B^+\rtimes\genby\phi$, the semi-direct product of a spherical 
Artin monoid with a diagram automorphism (with $\bS$ the usual atoms and
Garside family  $\cS=\bW$).
Our results apply also to the case of dual Artin monoids, but
this will not be used in this paper.

We  fix also the conjugacy  class $\cI$ under $M$  of a subset of $\bS$. By
property  \ref{artin(ii)} any  element of  $\cI$ is  a subset  of $\bS$. We
assume all elements of this class are parabolic subsets (which is automatic
in the ordinary Artin monoid case where all subsets are parabolic).

Let $\Conj(M,\cI)$ be the connected component of the simultaneous conjugacy
category  of $M$  whose objects  are the  elements of  $\cI$. A morphism in
$\Conj(M,\cI)$ with source $\bI\in\cI$ is given by $\bb\in M$ such that for
each  $\bs\in \bI$ we have $\bs^\bb\in M$, which by property \ref{artin(ii)}
implies $\bs^\bb\in\bS$. We denote such a morphism in
$\Conj(M,\cI)(\bI,\ud)$  by  $\conjd\bI\bb$,  or  if  we  want  to
specify   the  target   we  denote   it  by  $\bI\xrightarrow\bb\bJ$  where
$\bJ=\{\bs^\bb   \mid  \bs\in\bI\}$,   and  in   this  situation  we  write
$\bJ=\bI^\bb$.

By Proposition \ref{FAd Garside} the set
$\{\conjd\bI\bb\mid\bb\in\cS\}\cap\Conj(M,\cI)$ is
a Garside family in $\Conj(M,\cI)$.

\subsection*{The ribbon category}

For $\bb\in M$ we denote by $\alpha_\bI(\bb)$ the maximal left-divisor of $\bb$ 
in $M_\bI$ given by Lemma \ref{alphaI}, which is unique since $M_\bI^\times=\{1\}$.
We denote by $\omega_\bI(\bb)$ the element defined by 
$\bb=\alpha_\bI(\bb)\omega_\bI(\bb)$.
We say that $\bb\in M$ is {\em $\bI$-reduced} if it is left-divisible by
no element of $\bI$, or equivalently if $\alpha_\bI(\bb)=1$.

\begin{definition}\label{defribbon}
We   define  the   ribbon  category   $M(\cI)$  as   the  subcategory  of
$\Conj(M,\cI)$  obtained by restricting the morphisms to the
$\conjd\bI\bb$ such that $\bb$ is $\bI$-reduced.
\end{definition}

This  makes  sense  since  the  above  class  of  morphisms  is  stable  by
composition  by  (ii)  in  the  next  proposition;  assertion  (i)  of  the  next
proposition  is a motivation  for restricting to  such morphisms by showing
that we ``lose nothing'' in doing so.

\begin{proposition}\label{ribbon}
\begin{enumerate}
\item\label{ribbon(i)}    Given   $\bI\in\cI$    and   $\bb\in    M$   then
$\conjd\bI\bb\in\Conj(M,\cI)$  if  and  only if $\bI^{\alpha_\bI(\bb)}=\bI$
and $\conjd\bI{\omega_{\bI}(\bb)}\in M(\cI)$.
\item\label{ribbon(ii)} If $\bI\xrightarrow\bb\bJ\in M(\cI)$ then
for any $\bb'\in M$ we have $\alpha_\bJ(\bb')=\alpha_\bI(\bb\bb')^\bb$.
In particular if  $(\bI\xrightarrow\bb\bJ)\in M(\cI)$ and
$(\bJ\xrightarrow{\bb'}\bK)\in\Conj(M,\cI)$ then
$(\bI\xrightarrow{\bb\bb'}\bK)\in M(\cI)$ if and only if
$(\bJ\xrightarrow{\bb'}\bK)\in M(\cI)$.
\item\label{alphalcm}   
If two morphisms in $M(\cI)$ admit a right-lcm in
$\Conj(M,\cI)$, then this lcm is in $M(\cI)$.
\end{enumerate}
\end{proposition}
Note   that   if   $\conjd\bI\bc$   is   the  right-lcm  of  two  morphisms
$\conjd\bI\bb$  and $\conjd\bI{\bb'}$ as in (iii)  then by Lemma \ref{FAd C
is lcm} $\bc$ is the right-lcm in $M$ of $\bb$ and $\bb'$.
\begin{proof}
Let us prove (i). 
We prove that if $\bs\in\bI$ and
$\bs^\bb\in  M$ then $\bs^{\alpha_\bI(\bb)}\in\bI$. This  will prove (i) in
one  direction ---we use  that $\bI$ is  finite, see \ref{artin(i)}, so that
$\bI^{\alpha_\bI(\bb)}\subset\bI$ implies $\bI^{\alpha_\bI(\bb)}=\bI$. The
converse is obvious.

By property \ref{artin(ii)} we  have $\bs\bb=\bb\bt$ for some
$\bt\in \bS$. If $\bs\preccurlyeq\bb$ we
write $\bb=\bs^k\bb'$ for some $k$ and $\bb'$ such that $\bs$ does not left-divide $\bb'$.
We have $\bs\bb'=\bb'\bt$ and
$\alpha_\bI(\bb)=\bs^k\alpha_\bI(\bb')$  and  we  are reduced to the case where
$\bs$  does not left-divide $\bb$. Then any right-lcm of $\bs$ and $\alpha_\bI(\bb)$
left-divides $\bs\bb=\bb\bt$ and there is such a right-lcm in $M_\bI$ since
$M_\bI$  is weakly closed under right-lcm (\ref{artin(iii)}). We write this lcm
$\bs\bv=\alpha_\bI(\bb)\bu$, with $\bv$ and $\bu$  in  $M_\bI$ since
$M_\bI$ is closed under right-quotient (\ref{artin(iii)}) and
$\bv,\bu\ne 1$ since $\bs\not\preccurlyeq\bb$.
Since $\bs\bv\preccurlyeq\bs\bb$
we get that $\bv$ left-divides $\bb$, so left-divides $\alpha_\bI(\bb)$,
thus $\alpha_\bI(\bb)=\bv\ba$ for some $\ba\in M_\bI$.  
We get  $\bs\bv=\alpha_\bI(\bb)\bu=\bv\ba\bu$.
By property \ref{artin(ii)}  we have $\ba\bu\in\bS$, thus
$\bu$ is an atom which is in $M_\bI$, hence $\bu\in\bI$ and $\ba=1$ since
$\bS$ is a transversal for $\eqir$. We get
$\bs^{\alpha_\bI(\bb)}=\bs^\bv=\bu\in\bI$, which gives the result.

Let us prove  (ii).  For  $\bs\in  \bI$  let $\bs'=\bs^\bb\in \bJ$.
Since $\bI\xrightarrow{\bb}\bJ\in M(\cI)$ we have $\bs\not\preccurlyeq\bb$.
Then $\bb\bs'=\bs\bb$ is  a  common multiple of $\bs$ and
$\bb$  which has  to be an lcm since $\bs'$ is an  atom. So for
$\bs\in\bI$ we have $\bs\preccurlyeq \bb\bb'$ if and only if
$\bb\bs'\preccurlyeq\bb\bb'$,  that is, $\bs^\bb\preccurlyeq\bb'$ whence
the result. 

To prove (iii) we show first the statement that if for $\bb,\bc\in  M$
we have $\bb\preccurlyeq\bc$ and $\conjd\bI\bb\in M(\cI)$, then
$\bb\preccurlyeq\omega_\bI(\bc)$.
We write $\bc=\bb\bb'$ and $\bJ=\bI^\bb$.
By (ii) we have $\alpha_\bI(\bc)^\bb=\alpha_\bJ(\bb')$, whence
$\alpha_\bI(\bc)\bb=\bb\alpha_\bJ(\bb')\preccurlyeq
\bb\bb'=\bc=\alpha_\bI(\bc)\omega_\bI(\bc)$. Left-canceling
$\alpha_\bI(\bc)$ we get $\bb\preccurlyeq\omega_\bI(\bc)$.

Now (iii) is a particular case of the above statement since
if $\bc$ is the right-lcm of $\bb$ and $\bb'$ where $\conjd\bI\bb$
and $\conjd\bI{\bb'}$ are in $M(\cI)$,
we  get that  $\omega_\bI(\bc)$ is  a common right-multiple of $\bb$ and $\bb'$,
thus   $\bc\preccurlyeq\omega_\bI(\bc)$,   which implies $\alpha_\bI(\bc)=1$.
\end{proof}

Note that by Proposition \ref{ribbon}(i) a morphism in $M(\cI)$ with source
$\bI$  corresponds by the forgetful functor to an element  $\bb\in M$ such that $\alpha_\bI(\bb)=1$ and
such  that  for  each  $\bs\in\bI$  we  have  $\bs^\bb\in  M$. We will thus
sometimes just denote by $\bb$ such a morphism in $M(\cI)$ when the context
makes its source clear.

The next proposition shows that $(\cS\cap M(\cI))\cup M^\times$ generates $M(\cI)$. 
Note any element of $M^\times$ gives rise to an element of $M(\cI)$.

\begin{proposition}  \label{normal in MI}
All the terms of a normal decomposition 
in $\Conj(M,\cI)$ of a morphism of $M(\cI)$ are in $M(\cI)$.
\end{proposition}
\begin{proof} Let $\conjd\bI\bb\in M(\cI)$ and let
$\bb=\bb_1\dotsm\bb_k$ be a normal decomposition in $M$,
which gives a normal decomposition of $\conjd\bI\bb$ in
$\Conj(M,\cI)$ by Proposition \ref{FAd Garside}. 
We proceed by induction on $k$. We have
$\alpha_{\bI}(\bb_1)\preccurlyeq\alpha_{\bI}(\bb)=1$ thus
$\alpha_{\bI}(\bb_1)=1$ and $\bI\xrightarrow{\bb_1}\bI^{\bb_1}\in M(\cI)$.
This is the first step of the induction. Now, by
\ref{ribbon(ii)} we get $\conjd{\bI^{\bb_1}}{\bb_2\dotsm\bb_k}\in M(\cI)$
which concludes by induction.
\end{proof}

\begin{corollary}\label{Garside C(I)}
The set $\cS\cap M(\cI)=\{\conjd\bI\bw\in \Conj(M,\cI)\mid\bw\in\cS\text{ and } 
\alpha_\bI(\bw)=1\}$ is a Garside family in $M(\cI)$.
\end{corollary}
\begin{proof}
By \ref{ribbon(ii)} and \ref{alphalcm} the  subcategory  $M(\cI)$  of
$\Conj(M,\cI)$   is  closed under  right-quotient  and   right-lcm, hence 
the subfamily $\cS\cap M(\cI)$ is closed under right-quotient and right-lcm in
$\cS\cap\Conj(M,\cI)$.
Thus Lemma \ref{Garside subfamily} gives the result since $(\cS\cap M(\cI))\cup
M^\times$ generates $M(\cI)$ by Proposition \ref{normal in MI}.
\end{proof}

Our aim now is Proposition \ref{atoms C(I)} which gives
a description  of the atoms  of $M(\cI)$, and
a convenient criterion to decide whether  $\bb\in M$ gives rise to an
element  of $M(\cI)$.

For $\bI\subset\bS$ let $\Phi_\bI$ be the Garside automorphism of $M_\bI$
associated with the Garside element $\Delta_\bI$ (see \ref{artin(iv)a}). 
Since $\bI$ is finite (see \ref{artin(i)}) and is the whole set of atoms of
$M_\bI$, we have $\Phi_\bI(\bI)=\bI$.

We denote by $\Phi$ the Garside automorphism of $M$ associated to $\Delta$.
Since $\Phi$ is an automorphism which preserves $\cS$, for $\bI\subset\bS$,
it sends the Garside family $\cS\cap M_\bI$ to the Garside family
$\cS\cap M_{\Phi(\bI)}$ thus $\Phi(\Delta_\bI)=\Delta_{\Phi(\bI)}$.
\begin{proposition}\label{Garside map in M(cI)}
$M(\cI)$ has a Garside map  defined by the collection
of  morphisms  $\bI\xrightarrow{\Delta_\bI\inv\Delta}\Phi(\bI)$ for
$\bI\in\cI$. 
\end{proposition}
\begin{proof} We have $\Phi_\bI(\bI)=\bI$ and
$\omega_\bI(\Delta)=\Delta_\bI\inv\Delta$. Thus by Proposition \ref{ribbon(i)}
$\bI\xrightarrow{\Delta_\bI\inv\Delta}\Phi(\bI)\in \cS\cap M(\cI)$. 
We need two lemmas.
\begin{lemma}\label{left-divide Delta(I)}
Any morphism $\conjd\bI\bb\in M(\cI)\cap\cS$ left-divides
$\bI\xrightarrow{\Delta_\bI\inv\Delta}\Phi(\bI)$.
\end{lemma}
\begin{proof}
The divisibility we seek is equivalent to
$\Delta_\bI\bb$  left-dividing $\Delta$. Since $\Delta_\bI$
and $\bb$ left-divide $\Delta$, a right-lcm $\delta$ of these elements
divides $\Delta$. We claim that $\delta\eqir \Delta_\bI\bb$ which will show the
lemma. Since $\bI^\bb\subset\bS$ we have $\Delta_\bI^\bb\in M$ thus
$\delta\preccurlyeq\bb\Delta_\bI^\bb=\Delta_\bI\bb$.
Notice that $\alpha_\bI(\delta)=\Delta_\bI$ since $\Delta_\bI\preccurlyeq\delta$
and $\alpha_\bI(\delta)\preccurlyeq\alpha_\bI(\Delta)=\Delta_\bI$.
Now write $\delta=\bb\bx$; by Proposition~\ref{ribbon(ii)}
we have $\alpha_{\bI^\bb}(\bx)=\alpha_\bI(\delta)^\bb=\Delta_\bI^\bb$.
Thus $\Delta_\bI^\bb\preccurlyeq\bx$ thus
$\bb\Delta_\bI^\bb\preccurlyeq\bb\bx=\delta$, whence our claim.
\end{proof}
\begin{lemma}\label{Delta_I^b=Delta_J}
If $\bI\xrightarrow\bb\bJ$ is in $M(\cI)$ we have $\Delta_\bJ=\Delta_\bI^\bb$;
conjugation by $\bb$ induces an isomorphism of Garside monoids
$M_\bI\xrightarrow\sim M_\bJ$ which preserves normal forms.
\end{lemma}
\begin{proof}It is sufficient to prove the lemma for elements of the
generating set $(\cS\cap M(\cI))\cup\Isom M$. So we assume $\bb\in\cS\cup\Isom M$.
If $\bb\in\cS$,
in the proof of Lemma \ref{left-divide Delta(I)} we have 
$\Delta_\bI^\bb\preccurlyeq\bx$ where $\bx$ is a right-divisor
hence a left-divisor of $\Delta$, thus $\Delta_\bI^\bb\preccurlyeq\Delta$.
This is also clearly true if $\bb\in\Isom M$. Since $\Delta_\bI^\bb\in M_\bJ$
we get $\Delta_\bI^\bb\preccurlyeq\Delta_\bJ$. We show by contradiction
that this divisibility cannot be strict. By Lemma \ref{left-divide Delta(I)}
we can write
$\Delta_\bI\inv\Delta=\bb\bb'$; then by \ref{ribbon(ii)} we have
$\bJ\xrightarrow{\bb'}\Phi(\bI)\in M(\cI)$ and by the same argument as
above $\Delta_\bJ^{\bb'}\preccurlyeq\Delta_{\Phi(\bI)}$. Now
$\bb'$ induces by conjugation a morphism $M_\bJ\to M_{\Phi(\bI)}$ so we can
transport the strict divisibility $\Delta_\bI^\bb\prec\Delta_\bJ$ to
$\Delta_\bI^{\bb\bb'}\prec\Delta_\bJ^{\bb'}$. Composing we get
$\Phi(\Delta_\bI)=\Delta_\bI^{\bb\bb'}\prec\Delta_\bJ^{\bb'}\preccurlyeq
\Delta_{\Phi(\bI)}=\Phi(\Delta_\bI)$, a contradiction.

The second part of the statement follows from the first since
the first term of a  normal form of an element $\bx$ in a monoid with a 
Garside element $\Delta$ is a left-gcd of $\bx$ and $\Delta$ (see
Proposition \ref{second domino}(iii)),
and the conjugation by $\bb$ preserves gcds since it is an isomorphism.
\end{proof}
We now show the proposition. We know by Lemma \ref{left-divide Delta(I)}
that any $\bI\xrightarrow\bb\bJ$ in $\cS\cap M(\cI)$ left-divides
$\bI\xrightarrow{\Delta_\bI\inv\Delta}\Phi(\bI)$. 
It remains to show that such a morphism right-divides
$\Delta_{\Phi\inv(\bJ)}\inv\Delta$,  which is  equivalent to $\bb\Delta_\bJ$
right-dividing   $\Delta$  since  
$\Phi(\Delta_{\Phi\inv(\bJ)})=\Delta_\bJ$.  This in  turn is equivalent to
$\bb\Delta_\bJ$  left-dividing $\Delta$ since $\Delta$ is a Garside element.
The  result is then  a consequence of the fact
that  $\Delta_\bI\bb$ divides $\Delta$ as we have seen in Lemma
\ref{left-divide Delta(I)} and of the equality $\bb\Delta_\bJ=\Delta_\bI\bb$ 
which is given by Lemma \ref{Delta_I^b=Delta_J}.
\end{proof}

\begin{proposition}  Let $\bI\in\cI$ and let $\bJ$ be a parabolic subset of $\bS$ such that
$M_\bI\subsetneq M_\bJ$.
Then $\Delta_\bI\preccurlyeq\Delta_\bJ$ (see \ref{artin(v)}) and
$\bI\xrightarrow{v(\bJ,\bI)}\Phi_\bJ(\bI)$, where
$v(\bJ,\bI)=\Delta_\bI\inv\Delta_\bJ$,
is   a   morphism   in $M(\cI)$.
\end{proposition}
\begin{proof}
As noted after Proposition \ref{ribbon} we have to show that
$\alpha_\bI(v(\bJ,\bI))=1$
and that any $\bt\in\bI$ is conjugate by
$v(\bJ,\bI)$ to an element of $M$.
Since $\Delta_\bI\inv\Delta_\bJ$ left-divides $\Delta_\bI\inv\Delta$,
and $\alpha_\bI(\Delta_\bI\inv\Delta)=1$, by definition of $\Delta_\bI$,
we get the first property. The second is clear since
by definition $v(\bJ,\bI)$ conjugates $\bt$ to $\Phi_\bJ(\Phi_\bI\inv(\bt))$.
\end{proof}
(i)  of the next proposition is due  to Paris \cite[5.6]{paris} in the case
of Artin monoids.
\begin{proposition}\label{atoms C(I)}
\begin{enumerate}
\item Let $\bI\in\cI$ and
$\bb\in M$ such that $\alpha_\bI(\bb)=1$ and such that there exists
$p>0$ such that $(\Delta_\bI^p)^\bb\in M$. Then $\conjd\bI\bb\in M(\cI)$.
\item The atoms of $M(\cI)$ are the $v(\bJ,\bI)$ not strictly
divisible by another $v(\bJ',\bI)$ for $\bI\in\cI$.
\end{enumerate}
\end{proposition}
\begin{proof}  
Since $M$ is right-Noetherian, for (i) it suffices to prove
that  under our  assumption $\bb$  is either invertible or
left-divisible by  some non-invertible $\bv\in M$ giving rise to an element
of $M(\cI)$;  indeed if $\bb=\bv\bb'$  where $\bI\xrightarrow \bv\bI'\in M(\cI)$
then  by \ref{ribbon(ii)} we have $\alpha_{\bI'}(\bb')=1$ and since $\bI^\bv=\bI'$
we  have  $(\Delta_{\bI'}^p)^{\bb'}\in M$ by Lemma \ref{Delta_I^b=Delta_J},
so  by Noetherian induction we have $\conjd{\bI'}{\bb'}\in M(\cI)$, whence
$\conjd\bI\bb\in M(\cI)$. We will prove that $\bb$ is left-divisible by $v(\bJ,\bI)$
for some parabolic $\bJ\supsetneq\bI$ which will imply (i).
We proceed by decreasing induction on  $p$. We show that if for $i>0$
we have $\bs\preccurlyeq\Delta_\bI^i \bb$ for some atom $\bs$ not in $M_\bI$,
$v(\bJ,\bI)\preccurlyeq  \Delta_\bI^{i-1}\bb$ where $\bJ$ is as prescribed 
in \ref{artin(v)} from $\bI$ and $\bs$.
Indeed, the  right-lcm  of $\bs$ and
$\Delta_\bI$ is $\Delta_\bJ$ by property \ref{artin(v)}  thus  from
$\bs\preccurlyeq\Delta_\bI^i  \bb$ and  $\Delta_\bI\preccurlyeq\Delta_\bI^i \bb$ we
deduce $\Delta_\bJ\preccurlyeq \Delta_\bI^i \bb$. Since
$\Delta_\bJ=\Delta_\bI v(\bJ,\bI)$ we get as claimed
$v(\bJ,\bI)\preccurlyeq \Delta_\bI^{i-1}\bb$.
The induction starts at $i=p$ by taking  for $\bs$ any atom left-dividing
$\bb$, thus not in $M_\bI$ since $\alpha_\bI(\bb)=1$. Such an atom
satisfies $\bs\preccurlyeq \bb\preccurlyeq \Delta_\bI^p \bb$ since the assumption on
$\bb$ can be written $\bb\preccurlyeq\Delta_\bI^p\bb$.
Since any atom $\bt$ such that
$\bt\preccurlyeq  v(\bJ,\bI)$ is not in $M_\bI$ the induction can go on while
$i-1>0$.

We get (ii) from the proof of (i): any  element
$\bb\in M(\cI)$  satisfies the assumption  of (i) for  $p=1$ 
and $\bI$ equal to the source of $\bb$; whence the result since in
the proof of (i) we have seen that $\bb$ is a product of elements of the form
$v(\bJ,\bK)$.
\end{proof}

Though  in the current paper we need only finite Coxeter groups,
we note  that the  above description  of the  atoms also extends
to the  case of Artin  monoids which are
associated with infinite Coxeter groups ---and thus do not have a Garside element. 
Proposition \ref{godelle} below can be
extracted from the proof of Theorem 0.5 in \cite{godelle}.

In the case of an Artin monoid $(B^+,\bS)$ the Garside family
of Corollary \ref{Garside C(I)} in $B^+(\cI)$ is
$\bW\cap B^+(\cI)=\{\bI\xrightarrow\bw\bJ \in
\Conj B^+(\cI)\mid\bw\in\bW\text{  and  }  \alpha_\bI(\bw)=1\}$. 
For $\bI\subset\bS$ and $\bs\in\bS$ we denote by $\bI(\bs)$ the connected
component of $\bs$ in the Coxeter diagram of $\bI\cup\{\bs\}$, that is the
vertices of the connected component of $\bs$ in the graph with vertices
$\bI\cup\{\bs\}$ and an edge between $\bs'$ and $\bs''$ whenever $\bs'$ and 
$\bs''$ do not commute.

When $\bI$ is spherical, the subgroup $W_I$ generated by the image
$I$ of $\bI$ in $W$
is finite even though $W$ is not, in which case
we denote by $\bw_\bI$ the lift in $\bW$ of the
longest element of $W_I$. With these notations, we have
\begin{proposition}\label{godelle} The atoms of $B^+(\cI)$ are the morphisms
$\bI\xrightarrow{v(\bs,\bI)}\lexp{v(\bs,\bI)}\bI$
where $\bI$ is in $\cI$ and $\bs\in \bS-\bI$ is such that $\bI(\bs)$
is spherical,
and where $v(\bs,\bI)=\bw_{\bI(\bs)} \bw_{\bI(\bs)-\{\bs\}}$.
\end{proposition}
\section{Application to Artin groups}\label{application}
We will spell out how the above results can be stated in two particular cases.
We try to recall enough notation so this section can be read independently
of the previous ones.

\subsection*{Artin monoids with automorphism}
We  first look at the case of a  spherical Artin monoid $B^+$ attached to a
Coxeter  system $(W,S)$  with a  diagram automorphism  $\phi$, see 
\ref{artin monoids}. The category
$\cC$  we will take is the  monoid $B^+\rtimes\genby\phi$; it has a Garside
element   $\bw_0$  and  an  attached  Garside  family  $\bW$.  The  Garside
automorphism $\Phi$ is given by $\bb\mapsto\bb^{\bw_0}$; it  is trivial if
$\bW_0$ is central and has order 2 otherwise. 
We set $\bpi=\bw_0^2$, a central element in $B^+$.
An element $\bb\phi\in B^+\rtimes \genby\phi$ is $(d,p)$-periodic 
if $(\bb\phi)^d=\bw_0^p\phi^d$, 
which can
be written $\bb\lexp\phi\bb\lexp{\phi^2}\bb\dotsm=\bw_0^p$.
\begin{theorem}  If $\phi=\Id$, two periodic elements of
$B^+$ of same period are cyclically conjugate.
\end{theorem}
\begin{proof}  This results from the work of David Bessis on the dual braid
monoid. Two periodic elements of same period in $B^+$ are also periodic and
have  equal periods in the dual monoid,  since the Garside element $\bw_0$
of $B^+$ is a power of
the  Garside  element  of  the  dual monoid. By \cite[11.21]{bessis1}, such
elements  are conjugate in the dual monoid,  so are conjugate in $B$, hence
are  conjugate in  $B^+$; indeed  if $\bb'=\bh\inv\bb\bh$ with $\bb,\bb'\in
B^+$ and $\bh\in B$, then there exists $i>0$ such that $\bh\bpi^i\in B^+$
and  since $\bpi$ is central $\bh\bpi^i$ still conjugates $\bb$ to $\bb'$. By
Proposition   \ref{Ad=Cyc}  conjugate  periodic   elements  are  cyclically
conjugate.
\end{proof}
We  conjecture that the  same result holds  in the case $\phi\ne\Id$.

Taking in account that $\Phi^2=\Id$, statement \ref{70g} gives:
\begin{proposition}\label{6970 for B}
Let   $\bb'\phi\in B^+\phi$  be    $(d,2)$-periodic,  that  is
$(\bb'\phi)^d=\bpi\phi^d$,  and let $e=\lfloor\frac  d2\rfloor$. Then there
exists  $\bb\phi\in B^+\phi$  cyclically conjugate  to $\bb'\phi$ such that
$\bb^e\in\bW$, and
\begin{itemize}
\item
If   $d$   is   even   then   $(\bb\phi)^e=\bw_0\phi^e$.  The  centralizer
$C_{B^+}(\bb\phi)$  identifies  with  $\cyc  B^+(\bb\phi)$,  and  even  more
specifically   to  the  endomorphisms  of   $\bb\phi$  in the category of
conjugacy  by $\bw_0\phi^e$-stable divisors.
\item
If $d$ is odd there exists $\bv\in\bW$ such that
$(\bb\phi)^e\bv=\bw_0\phi^e$   and   $\bb=\bv\phi^{-e}(\bv^{\bw_0})$.   The
centralizer   $C_{B^+}(\bb\phi)$   identifies   with   the  endomorphisms  of
$\bv\bw_0\phi^{-e}$   in  the  category  of  conjugacy  by  $\phi^d$-stable
divisors.
\end{itemize}
\end{proposition}
Part of the above proposition is already in \cite[6.8]{BM}.
The equation $(\bb\phi)^d=\bpi\phi^d$ for $(d,2)$-periodic elements
made the authors of \cite{BM} call such elements $d$-th $\phi$-roots of $\bpi$.
\subsection*{Ribbons in Artin monoids}\label{artin ribbons}

We  keep in this subsection a spherical Artin monoid $B^+$ attached to $(W,S)$
with  a  diagram  automorphism  $\phi$  and  consider  the  ribbon category
$B^+\rtimes\genby\phi(\cI)$  defined by a conjugacy  class $\cI$ of subsets
of $\bS$. 

A subset $I\subset S$ and the corresponding subset
$\bI\subset\bS$ determine:
\begin{itemize} 
\item A {\em standard parabolic} subgroup $W_I$ generated by $I$; we
denote by $w_I$ its longest element (with this notation $w_0=w_S$). 
In every coset $W_I w$ there is a unique
shortest element called $I$-reduced.
\item A {\em parabolic submonoid} $B^+_\bI$ generated by
$\bI$; it has the Garside family $\bW_\bI:=\bW\cap B^+_\bI$ and the associated
Garside element is the lift $\bw_\bI$ of $w_I$; we set $\bpi_\bI=\bw_\bI^2$.
By Lemma \ref{alphaI} every element $\bb\in B^+$ 
has a unique longest divisor $\alpha_\bI(\bb)$ in $B^+_\bI$; an element
such that $\alpha_\bI(\bb)=1$ is called $\bI$-reduced.
\end{itemize}

The  ribbon  category  $B^+(\cI)$  is  the  category  whose objects are the
elements  of $\cI$  and a  morphism $\bI\xrightarrow  \bb\bJ$ is  given by
an $\bI$-reduced element $\bb\in  B^+$ such that $\bI^\bb=\bJ$; since $\bJ$
is  determined by  $\bI$ and  $\bb$ we  denote also by $\conjd\bI\bb$ this
morphism.  Proposition \ref{ribbon} shows that this definition makes sense,
that  is if we have a composition $\bI\xrightarrow\bb\bJ\xrightarrow\bc\bK$
in   $B^+(\cI)$,  then  $\alpha_\bI(\bb\bc)=1$.  

By  Corollary \ref{Garside C(I)} and Proposition \ref{Garside map in M(cI)}
$B^+(\cI)$ has a Garside family $\cS$ consisting of the morphisms
$\conjd\bI\bw$ where $\bw\in\bW$ and a Garside map
$\Delta_\cI(\bI)=\bI\xrightarrow{\bw_\bI\inv\bw_0}\bI^{\bw_0}$. These
properties include the following:
\begin{lemma}\label{bw generate B+(I)}
\begin{enumerate}
\item
$\cS$  generates  $B^+(\cI)$;  specifically, if $\bI\xrightarrow\bb\bJ\in
B^+(\cI)$  and $(\bw_1,\dots,\bw_k)$  is the  $\bW$-strict normal 
decomposition  of $\bb$, there
exist  subsets $\bI_i$ with $\bI_1=\bI$,  $\bI_{k+1}=\bJ$ such that for all
$i$ we have $\bI_{i+1}=\bI_i^{\bw_i}$; thus
$\bI\xrightarrow{\bw_1}\bI_2\to\dots\to  \bI_k\xrightarrow{\bw_k}\bJ$ is a
decomposition  of $\bI\xrightarrow\bb\bJ$  in $B^+(\cI)$  as a product of
elements of $\cS$.
\item
The relations
$(\bI\xrightarrow{\bw_1}\bJ\xrightarrow{\bw_2}\bK)=(\bI\xrightarrow{\bw}\bK)$ 
when $\bw=\bw_1\bw_2\in\bW$ form a presentation of $B^+(\cI)$.
\end{enumerate}
\end{lemma}
In our case strict normal decompositions are unique. They can be defined as
follows:  for $\bb\in B^+$, let $\alpha(\bb)$  be the left-gcd of $\bb$ and
$\bw_0$;  the  restriction  of  $\alpha$  to  $B^+-\{1\}$  is  a $\bW$-head
function,  thus  $\bw_1:=\alpha(\bb)$  is  the  first  term  of  the normal
decomposition  of  $\bb$,  and  the  other  terms  are defined similarly by
induction, setting $\bw_2=\alpha(\bw_1\inv \bb)$, etc$\dots$

For generating the category $B^+\rtimes\genby\phi(\cI)$ we need additionally 
the invertible
morphisms $\bI\xrightarrow\phi\bI^\phi$. The family $\cS$ is still a
Garside family for this category, with the same Garside map $\Delta_\cI$.
When  $\cI=\{\emptyset\}$, $B^+(\cI)$ reduces to the Artin-Tits monoid $B^+$
and $B^+\rtimes\genby\phi(\cI)$ reduces to $B^+\rtimes\genby\phi$, thus
the results in this subsection generalize those of the previous subsection.

We  will  be interested in
$(d,2)$-periodic   elements  in  $B^+\phi(\cI)$.  Such  an  element  is  an
endomorphism   of  the   form  $\bI\xrightarrow{\bb\phi}\bI$   or  via  the
correspondence   between   conjugacy   in   the  semi-direct  category  and
$\phi$-conjugacy, a morphism $\bI\xrightarrow\bb\lexp\phi\bI$ in $B^+(\cI)$
where $\lexp{\bb\phi}\bI=\bI$. Since
$\Delta_\cI(\bI)\Delta_\cI(\bI^{\bw_0})=\bI\xrightarrow{\bpi/\bpi_\bI}\bI$
the condition for this morphism to be $(d,2)$-periodic is
$(\bb\phi)^d=\bpi/\bpi_\bI\phi^d$.

By the forgetful functor $(\conjd\bI{\bb\phi})\mapsto \bb\phi$ 
the morphisms in $B^+\phi(\cI)(\bI,\ud)$ identify with the elements 
$\bb\phi\in B^+\phi$ such that
$\lexp{\bb\phi}\bI\subset\bS$ and $\alpha_\bI(\bb)=1$. We will thus
sometimes write $\bb\phi\in B^+\phi(\cI)(\bI,\ud)$ to 
mean $\conjd\bI{\bb\phi}\in B^+\phi(\cI)(\bI,\ud)$.

Taking into account the above,  and that the 
Garside automorphism associated to $\Delta_\cI$ is 
$\Phi(\bI\xrightarrow\bv\bI^\bv)=
\bI^{\bw_0}\xrightarrow{\bv^{\bw_0}}\bI^{\bv\bw_0}$,
the generalization of Proposition \ref{6970 for B} is
\begin{proposition}\label{periodic ribbons}
Let   $\bb'\phi\in B^+\phi$   be  such that $(\bb'\phi)^d=\bpi/\bpi_\bJ\phi^d$
for  some $\phi^d$-stable $\bJ\in\cI$,  and let $e=\lfloor\frac d2\rfloor$.
Then $\bb'\phi$ defines an endomorphism of $\bJ$ in $B^+\phi(\cI)$, that is
$\lexp{\bb'\phi}\bJ=\bJ$  and  $\alpha_\bJ(\bb')=1$.  This  endomorphism is
$(d,2)$-periodic  and  there  exists  a  $\phi^d$-stable  $\bI\in\cI$ and
$\bI\xrightarrow{\bb\phi}\bI\in  B^+\phi(\cI)(\bI)$ cyclically conjugate to
$\bJ\xrightarrow{\bb'\phi}\bJ\in B^+\phi(\cI)(\bJ)$ such that
$(\bb\phi)^d=\bpi/\bpi_\bI\phi^d$, $(\bb\phi)^e\in\bW\phi^e$, and
\begin{itemize}
\item
If   $d$   is   even   then   $(\bb\phi)^e=\bw_\bI\inv\bw_0\phi^e$.  
The  centralizer $\Conj B^+(\cI)(\bI\xrightarrow{\bb\phi}\bI)$  identifies  with  
$(\cyc  B^+(\cI)(\bI\xrightarrow{\bb\phi}\bI))^{\bw_0\phi^e}$.
\item
If $d$ is odd there exists $\cI\xrightarrow\bv\bI^{\bw_0\phi^e}\in\bW\cap B^+(\cI)$ 
such that $(\bb\phi)^e\bv=\bw_\bI\inv\bw_0\phi^e$   and   
$\bb=\bv\phi^{-e}(\bv^{\bw_0})$.   The
centralizer   $\Conj B^+(\cI)(\bI\xrightarrow{\bb\phi}\bI)$   identifies   with
$(\cyc B^+(\cI)(\bI\xrightarrow{\bv\Phi\phi^{-e}}\bI))^{\phi^d}$.
\end{itemize}
\end{proposition}
\begin{proof}
We need to prove that $(\bb'\phi)^d=(\bpi_\bJ)\inv\bpi\phi^d$   implies
$\alpha_\bJ(\bb')=1$ and that $\lexp{\bb'\phi}\bJ=\bJ$.
The condition $\alpha_\bJ(\bb')=1$ follows from
$\alpha_\bJ(\bb')\preccurlyeq\alpha_\bJ((\bb'\phi)^d)$ and from the fact that
$(\bpi_\bJ)\inv\bpi$ defines a morphism in $B^+(\cI)$ as we have seen above.
By Proposition \ref{atoms C(I)}(i) $\bb'\phi$ defines a morphism
$\bJ\xrightarrow{\bb'\phi}\bK$ in $B^+(\cI)$. Hence $\bb'\phi$ conjugates
$\bpi_\bJ$ to $\bpi_\bK$ by Lemma \ref{Delta_I^b=Delta_J}.
Since $\bb'\phi$ centralizes $\bpi/\bpi_\bJ\phi^d$ and $\bpi$ is central,
it thus centralizes $\bpi_\bJ\phi^d$, hence it centralizes
$\bpi_\bJ^\delta$, where $\delta$ is the order of
$\phi$ and we get $\bpi_\bJ^\delta=\bpi_\bK^\delta$. However
the support (see the proof of Lemma \ref{all parabolic}) of $\bpi_\bJ^\delta$ is $\bJ$ and that of $\bpi_\bK^\delta$
is $\bK$, thus $\bJ=\bK$ and $\bb'\phi$ stabilizes $\bJ$.

The other assertions of the proposition are straightforward translations
of Corollary \ref{70g}.
\end{proof}
We  note that  any element  which conjugates  a $(d,2)$-periodic element in
$B^+\phi$  to another is $\phi^d$-stable. Indeed such an element conjugates
some  $\bpi/\bpi_\bJ\phi^d$ to  some $\bpi/\bpi_\bI\phi^d$;  if $\delta$ is
the   order  of  $\phi$   since  $\bpi$  is   central  it  thus  conjugates
$\bpi_\bJ^\delta$  to  $\bpi_\bI^\delta$  thus  by  the same reasoning as 
the end of the proof above it conjugates
$\bI$ to $\bJ$, which finally implies that it commutes with $\phi^d$.

We now state \ref{F-Bestvina} in the case of ribbons.
\begin{corollary}\label{ribbon-Bestvina}
Let $\bb'\phi\in B^+\phi$ be such that
$(\bb'\phi)^d=(\bpi/\bpi_\bJ)^k\phi^d$     for     some     $\phi^d$-stable
$\bJ\in\cI$.  Then $\bb'\phi$  defines a  $(d,2k)$-periodic endomorphism of
$\bJ$  in $B^+\phi(\cI)$, and up to cyclic conjugacy in $B^+\phi(\cI)$, we
may  assume $k$ prime to  $d$. Then, for any  choice of integers $d',k'$
with $dk'=1+kd'$ there exists a $\phi^d$-stable
$\bI\in\cI$ and $\bI\xrightarrow{\bb\phi}\bI\in
B^+\phi(\cI)(\bI)$   cyclically   conjugate   to
$\bJ\xrightarrow{\bb'\phi}\bJ$   such  that
$(\bb\phi)^d=(\bpi/\bpi_\bI)^k\phi^d$    and    $(\bb\phi)^{d'}\preccurlyeq
(\bpi/\bpi_\bI)^{k'}$,   and   if   we   define   $\bb_1\in   B^+(\cI)$  by
$(\bb\phi)^{d'}   \bb_1  \phi^{-d'}=   (\bpi/\bpi_\bI)^{k'}$  then  $(\bb_1
\phi^{-d'})^d=\bpi/\bpi_\bI  \phi^{-dd'}$  and  $(\bb_1  \phi^{-d'})^k=(\bb
\phi)\phi^{-k'd}$.
\end{corollary}
\begin{proof}
As in the beginning of the proof \ref{periodic ribbons} we deduce from the
equality $(\bb'\phi)^d=(\bpi/\bpi_\bJ)^k\phi^d$ that $\bb'\phi$ defines
an element of $B^+\phi(\cI)(\bJ)$. The only other observation needed
is that we apply \ref{F-Bestvina} for the Garside structure corresponding to
the Garside map  $\Delta(\bJ)=\bJ\xrightarrow{\bpi/\bpi_\bJ}\bJ$, 
the square of the previously introduced Garside map $\Delta_\cI$
---this is allowed by \ref{power of Delta}. For this Garside map the
corresponding functor $\Phi$ is the identity, as required by \ref{F-Bestvina}.
\end{proof}
\begin{corollary}\label{power}
As in corollary \ref{ribbon-Bestvina}  let  $\bb'\phi\in  B^+\phi$ be such that
$(\bb'\phi)^d=(\bpi/\bpi_\bJ)^k\phi^d$     for     some     $\phi^d$-stable
$\bJ\in\cI$. Then $\bI\xrightarrow{\bb'\phi}\bJ$
is   cyclically   conjugate   in   $B^+\phi(\cI)$  to  a  $(d,2k)$-periodic
endomorphism $\bI\xrightarrow{\bb\phi}\bI$ such that
$(\bb\phi)^{\lfloor\frac              d{2k}\rfloor}\in\bW\phi^{\lfloor\frac
d{2k}\rfloor}$.
\end{corollary}
\begin{proof}
By \ref{ribbon-Bestvina} we may first assume that $k$ is prime to $d$.
We then use \ref{ribbon-Bestvina} to get $\bb_1\phi^{-d'}\in  B^+\phi^{-d'}$
satisfying
the assumption of \ref{periodic ribbons} with $\phi$ replaced by
$\phi^{-d'}$.  By \ref{periodic ribbons} we may find a cyclic conjugate
$\bb'_1\phi^{-d'}$ of $\bb_1\phi^{-d'}$ such that 
$(\bb'_1\phi^{-d'})^{\lfloor \frac d2\rfloor} \in\bW
\phi^{-d'\lfloor \frac d2\rfloor}$. If this cyclic conjugation conjugates
$\bb\phi=(\bb_1\phi^{-d'})^k\phi^{k'd}$ to
$(\bb'_1\phi^{-d'})^k\phi^{k'd}$ we are done since $
k\lfloor\frac   d{2k}\rfloor\le \lfloor\frac d2\rfloor$. Note that the cyclic
conjugacy in \ref{periodic ribbons} conjugates $\bJ$ to $\bI$ and 
$\bpi/\bpi_\bJ\phi^d$ to $\bpi/\bpi_\bI\phi^d$, so is $\phi^d$-stable
($\phi^{-dd'}$-stable in our application). If we had that any
$\phi^{dd'}$-stable element is $\phi^d$-stable we would be done since the
conjugation would then commute with $\phi^{k'd}$. Thus we finish using Lemma
\ref{pas Dirichlet} which shows that we may choose $d'$ prime to the order of
$\phi$.
\end{proof}
For $\bb\in B^+$, let $\alpha(\bb)=\gcd(\bb,\bw_0)$. It is a $\bW$-head
function in $B^+$ thus by Proposition \ref{normal in MI} and Corollary
\ref{Garside C(I)} $(\conjd\bI\bb)\mapsto(\conjd\bI{\alpha(\bb)})$ is a
$\cS$-head function.
\begin{lemma}\label{alpha(vw)}
For $\conjd\bI\bb\in B^+(\cI)$ and $\bv\in B^+_\bI$ we have
$\alpha(\bv\bb)=\alpha(\bv)\alpha(\bb)$.
\end{lemma}
\begin{proof}
Lemma \ref{bw generate B+(I)} implies that $\alpha(\bb)$ defines an
element of $B^+(\cI)(\bI,\ud)$ so that $\bv^{\alpha(\bb)}\in B^+$.
We have
$\alpha(\bv\bb)=\alpha(\bv\alpha(\bb))=\alpha(\alpha(\bb)\bv^{\alpha(\bb)})=
\alpha(\alpha(\bb)\alpha(\bv^{\alpha(\bb)}))$,
the first and last equalities by property ($\cH$) of Proposition
\ref{critereGarside}.
By Lemma \ref{Delta_I^b=Delta_J} we have
$\alpha(\bv^{\alpha(\bb)})=\alpha(\bv)^{\alpha(\bb)}$, 
so that $\alpha(\bv\bb)=\alpha
(\alpha(\bb)\alpha(\bv)^{\alpha(\bb)})=\alpha(\alpha(\bv)\alpha(\bb))$.
Since $\alpha(\bb)$ is $\bI$-reduced we have 
$\alpha(\bv)\alpha(\bb)\in\bW$, hence $\alpha(\alpha(\bv)\alpha(\bb))
=\alpha(\bv)\alpha(\bb)$.
\end{proof}
The following proposition shows a compatibility of morphisms in $B^+(\cI)$ 
with a ``parabolic'' situation.
\begin{proposition}\label{normal form of vw}
Fix  $\bI\in\cI$,  and  for  $\bJ\subset\bI$,  let  $\cJ$  be  the  set  of
$B^+_\bI$-conjugates of $\bJ$. Let $(\bI\xrightarrow{\bb}\bI')\in B^+(\cI)$
and  let $(\bJ\xrightarrow{\bv}\bJ')\in B^+_\bI(\cJ)$.
Let  $(\bu_1,\dots,\bu_k)$  be  the  strict normal decomposition of
$\bv\bb$  and let $(\bw_1,\bw_2,\dots,\bw_k)$  be a normal decomposition
of  $\bb$ (we have added some 1's at the end of the strict normal
decomposition so the decompositions have same length);
then for   each   $i$  there exists $\bv_i\in B^+$ such that $\bu_i=\bv_i\bw_i$
and $(\bv_1,\lexp{\bw_1}\bv_2,\lexp{\bw_1\bw_2}\bv_3,\dots)$   is  a  normal
decomposition of $\bv$.
\end{proposition}
\begin{proof}
We  proceed  by  induction  on  $k$.  By  Lemma  \ref{alpha(vw)},  we  have
$\bu_1=\alpha(\bv)\alpha(\bb)=\alpha(\bv)\bw_1$. Hence $\bv_1=
\alpha(\bv)$ is a solution. Cancelling $\bv_1$ we get
$\bu_2\dotsm\bu_k=\omega(\bv)^{\alpha(\bb)}\omega(\bb)$.    The   induction
hypothesis  applied to $\omega(\bv)^{\alpha(\bb)}$, which defines an
element of $B^+_{\bI^{\alpha(\bb)}}(\cJ^{\alpha(\bb)})$, and to
$\omega(\bb)$ which defines an element of $B^+(\cI)$ gives the result.
\end{proof}
\subsubsection*{The category $\DI$}\label{D+(I)}
The  category $\cyc B^+\phi(\cI)$ will play  an important role in our work:
it  will be interpreted as a  category of morphisms between Deligne-Lusztig
varieties.  For  this  reason  we  will  abbreviate its name to $\DI$; when
$\cI=\{\emptyset\}$ it reduces to the category $\cD^+$ of \cite[5.1]{DMR}.

The  objects  of  $\DI$  are endomorphisms $\bI\xrightarrow{\bw\phi}\bI$ in
$B^+\phi(\cI)$  and the morphisms are generated by the ``simple'' morphisms
that  we will  denote by  $\ad \bv$,  defined for $\bv\preccurlyeq\bw$ such
that  $\bI^\bv\subset\bS$;  such  a  morphism goes    from
$\bI\xrightarrow{\bw\phi}\bI$ to $\bJ\xrightarrow{\bv\inv\bw \phi\bv}\bJ$
where $\bJ=\bI^\bv$.

By Proposition \ref{CFC} the category $\DI$ has a Garside family consisting
of  the simple  morphisms. In  particular defining  relations for $\DI$ are
given  by the  equalities $\ad  \bv_1\dotsm \ad  \bv_k=\ad \bv'_1\dotsm \ad
\bv'_{k'}$ whenever $\ad\bv_i$ are simple and $\bv_1\dotsm
\bv_k=\bv'_1\dotsm \bv'_{k'}$ in $B^+$.
If $\bv=\bv_1\dotsm \bv_k\in B^+$ where the $\ad\bv_i$ are simple morphisms
of  $\DI$, we  still denote  by $\ad  \bv$ the composed morphism in $\DI$.

Note that for $\bw\phi\in B^+\phi(\cI)(\bI)$, the  centralizer
$\Conj B^+(\cI)(\bI\xrightarrow{\bw\phi}\bI)$ identifies via the forgetful
functor with the monoid
$$B^+_\bw:=\{\bb\in C_{B^+}(\bw\phi)\mid \bI^\bb=\bI\text{ and }
\alpha_\bI(\bb)=1\}.$$
The following theorem gives
a general case where we can describe $\DI(\bI\xrightarrow{\bw\phi}\bI)$:
\begin{theorem}\label{desc endo}
Assume that some power of $\bw\phi$ is divisible on the left by
$\bw_\bI\inv\bw_0$. Then
$\DI(\bI\xrightarrow{\bw\phi}\bI)=\Conj
B^+(\cI)(\bI\xrightarrow{\bw\phi}\bI)$,
thus consists of the morphisms
$\ad\bb$ where $\bb\in B^+_\bw$.
\end{theorem}
\begin{proof} This is a special case of Proposition \ref{Ad=Cyc}.
\end{proof}
Note that if $k$ is the smallest power such that $\lexp{\phi^k}\bI=\bI$ and
$\lexp{\phi^k}\bw=\bw$, then $\bw^{(k)}:=\bw\lexp
\phi\bw\dotsm\lexp{\phi^{k-1}}\bw$ is in $B^+_\bw$. Since $\ad\bw$ is equal
up  to an invertible to  the Garside map of  $\DI$ described in Proposition
\ref{Fcyc  Garside} and $\ad\bw^{(k)}$ is equal  up to an invertible to the
$k$-th power of that map, every element of
$\DI(\bI\xrightarrow{\bw\phi}\bI)$  divides a  power of  $\ad\bw^{(k)}$; it
follows that under the assumptions of Theorem \ref{desc endo} every element
of  $B^+_\bw$ divides  a power  of $\bw^{(k)}$.  In particular, in the case
$\bI=\emptyset$,  Theorem \ref{desc endo} says that
$B^+\cap C_B(\bw\phi)=\End_{\cD^+}(\bw)$, with the notations of \cite[2.1]{endo}.
Since $\bw_0$ divides a power of
$\bw\phi$, hence a power of $\bw^{(k)}$, any element of the group $C_B(\bw\phi)$ 
multiplied by some power of $\bw^{(k)}$ lies in $B^+$, hence the group
$C_B(\bw\phi)$ is generated as  a monoid
by   $\End_{\cD^+}(\bw)$   and
$(\bw^{(k)})\inv$.  Thus Theorem  \ref{desc endo}  in this  particular case
gives a positive answer to conjecture \cite[2.1]{endo}.

As an example of Theorem \ref{desc endo} we get that
$\DI(\bI\xrightarrow{\bpi/\bpi_\bI\phi}\bI)$   identifies   with  $\{\bb\in
C_{B^+}(\bI)^\phi\mid\alpha_\bI(\bb)=1\}$   which   itself   identifies  with
$B^+(\cI)(\bI)^\phi$.

\subsection*{Two examples}
In two cases we show a picture of the category associated with the
centralizer of a periodic element.

We   first  look  at  the  case   of  a  $(4,2)$-periodic  element  $\bw\in
B^+(W(D_4))$; by Proposition \ref{6970 for B}(i) we may assume $\bw^2=\bw_0$;
following Proposition
\ref{6970  for B}(i) we  describe the monoid  $(\cyc B^+(\bw))^{\bw_0}$, in
our  case equal to $\cyc B^+(\bw)$ since  $\bw_0$ is central. As in Theorem
\ref{good  for D_n}, we  choose $\bw$ given  by the word  in the generators
$123423$ where the labeling of the Coxeter diagram is

$\nnode{1}\edge\vertbar{3}{2}\edge\nnode{4}$

By Proposition \ref{6970 for B}(i) the  monoid $\cyc B^+(\bw)$ generates
$C_B(\bw)$;  by \cite[12.5(ii)]{bessis1}, $C_B(\bw)$ is the braid group of
$C_W(w)\simeq G(4,2,2)$.  This  braid group has presentation
$\langle\bx,\by,\bz\mid      \bx\by\bz=\by\bz\bx=\bz\bx\by\rangle$.     The
automorphism  $\bx\mapsto\by\mapsto\bz$  corresponds  to  the  triality  in
$D_4$. One of the generators $\bx$ corresponds to the  morphism $24$  in the
diagram below. The other generators are the conjugates of the similar morphisms
$41$ and $21$ in the other squares.
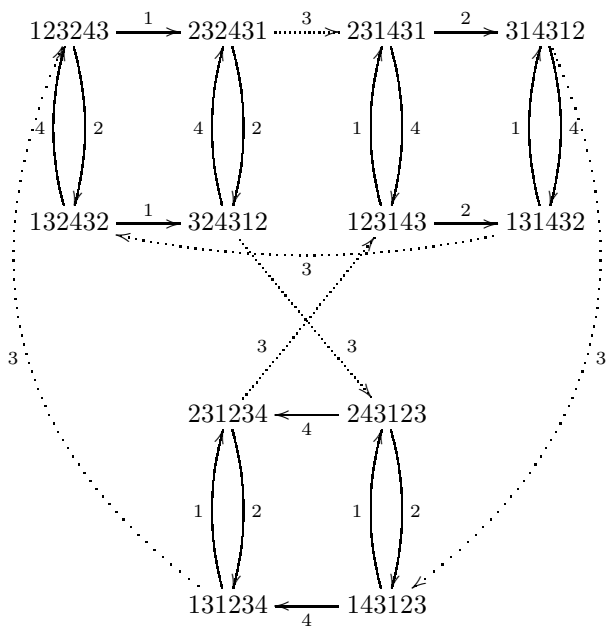
\vspace{-1cm}
$$\xymatrix@R+3pc{ &&&\\
123243\ar[r]^1\ar@/^/[d]^2&232431\ar@/^/[d]^2\ar@{.>}[r]^3&
 231431\ar[r]^2\ar@/^/[d]^4&314312\ar@/^/[d]^4\ar@{.>}@/^4pc/[lddd]^3\\
132432\ar[r]^1\ar@/^/[u]^4&324312\ar@/^/[u]^4\ar@{.>}[rd]^(.7)3&
 123143\ar@/^/[u]^1\ar[r]^2&131432\ar@/^/[u]^1\ar@{.>}@/^1pc/[lll]^3\\
     &231234\ar@/^/[d]^2\ar@{.>}[ur]^(.3)3&243123\ar@/^/[d]^2\ar[l]^4&\\
     &131234\ar@/^/[u]^1\ar@{.>}@/^4pc/[luuu]^3&143123\ar@/^/[u]^1\ar[l]^4&\\
&&&\\
}$$
We now look at a $(3,2)$-periodic $\bw\in B^+(W(A_5))$, that is $\bw^3=\bpi$, 
and following Proposition \ref{6970 for B}(ii) we describe
$\cyc B^+(\bv\Phi)$ where $\Phi$ is the Garside automorphism
$\bb\mapsto\bb^{\bw_0}$ and where $\bw=\bv\Phi(\bv)=\bv\cdot\bv^{\bw_0}$. By
Proposition \ref{6970 for B}(ii) the monoid $\cyc B^+(s\Phi)$
generates $C_B(\bw)$ and,
again by the results of Bessis, $C_B(\bw)$ is the braid group of $C_W(w)\simeq
G(3,1,2)$ (see Theorem \ref{type A}). We choose $\bw$ such that $\bv$ is
given by the word $21325$ in the generators. The generator of
$C_B(\bw)$ lifting the generator of order 3 of $G(3,1,2)$ is given by the word
$531$. The other one is the conjugate of any of the length 2 cycles $23$ in the diagram.
$$
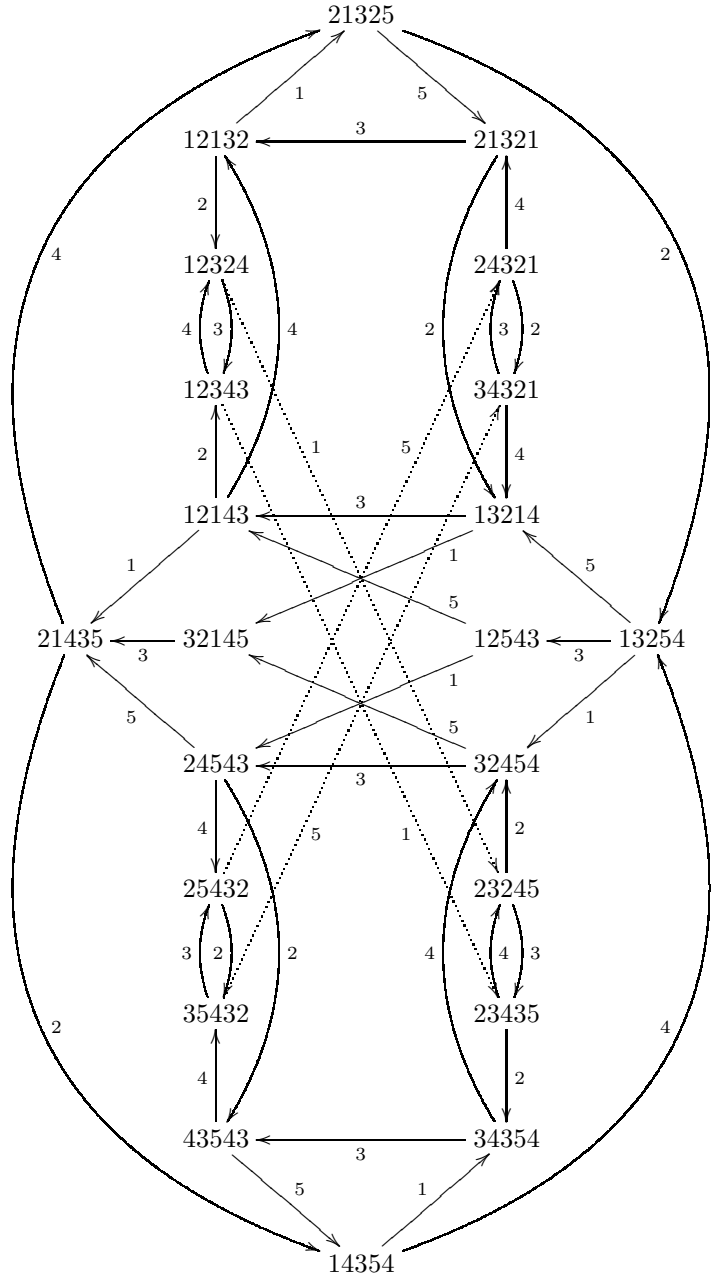
\xymatrix@u@C+.9pc{
     &&     &     &     &
     21435\ar@/_6pc/[llllldd]^2\ar@/^6pc/[rrrrrdd]_4&     &     &     &     &
     &\\
&43543\ar[ld]^5\ar[r]^4&35432\ar@{.>}[ddrrrrr]_(.3)5\ar@/^/[r]^3&25432\ar@{.>}[rrrrrdd]^(.7)5\ar@/^/[l]_2
&24543\ar[ru]^5\ar@/^2pc/[lll]^2\ar[l]_4&
32145\ar[u]^3&12143\ar[ul]_1\ar[r]^2\ar@/_2pc/[rrr]_4
&12343\ar@{.>}[ddlllll]_(.7)1\ar@/^/[r]^4&12324\ar@{.>}[ddlllll]^(.3)1\ar@/^/[l]_3&12132\ar[l]_2\ar[rd]_1&&\\
14354\ar[dr]^1\ar@/_6pc/[rrrrrdd]^4&     &     &&     &      &     &&     &
&21325\ar@/^6pc/[llllldd]_2\ar[dl]_5&\\
&34354\ar[uu]^3\ar@/^2pc/[rrr]^4&23435\ar[l]^2\ar@/^/[r]_4&23245\ar@/^/[l]^3\ar[r]_2
&32454\ar[uu]^3\ar[uur]_(.2)5&12543\ar[uul]^(.2)1\ar[uur]_(.2)5 &13214\ar[uu]_3\ar[uul]^(.2)1
&34321\ar@/^/[r]_3\ar[l]^4&24321\ar@/^/[l]^2\ar[r]_4&21321\ar[uu]_3\ar@/_2pc/[lll]_2&&\\
    &     &     &     &     & 13254\ar[ul]^1\ar[u]^3\ar[ur]_5&
&     &     &     &     &\\
}$$
\section{Representations into bicategories}\label{bicategories}
We  give here a  theorem on  representations of categories with Garside families which
generalizes  a result of Deligne \cite[1.11]{Deligne} about representations
of spherical braid monoids into a category; just as this theorem of Deligne
was  used to attach  a Deligne-Lusztig variety  to an element  of an Artin
monoid,  our theorem will be  used to attach a  Deligne-Lusztig variety to a
morphism  of a ribbon category. Note  that Theorem \ref{bicategory} covers 
the case of non-spherical Artin monoids.

We  follow the  terminology of  \cite[XII.6]{MacLane} for  bicategories. By
{\em representation of category $\cC$ into bicategory $X$} we mean a morphism
of bicategories between $\cC$ viewed as a trivial bicategory into the given
bicategory  $X$. This  amounts to  give a  map $T$  from $\Obj(\cC)$ to the
$0$-cells  of  $X$,  and  for  $f\in\cC$  of  source $x$ and target $y$, an
element  $T(f)\in V(T(x),T(y))$ where $V(T(x),T(y))$  is the category whose
objects  (resp.\ morphisms) are  the 1-cells of  $X$ with domain $T(x)$ and
codomain  $T(y)$ (resp.\ the 2-cells between  them), together with for each
composable  pair  $(f,g)$  an  isomorphism $T(f)T(g)\xrightarrow\sim T(fg)$
such that the resulting square
\begin{equation}\label{bicat}
\xymatrix{T(f)T(f')T(f'')\ar[r]^\sim\ar[d]^\sim&T(ff')T(f'')\ar[d]^\sim\\
T(f)T(f'f'')\ar[r]^\sim&T(ff'f'')\\}
\end{equation}
commutes.

We  define a representation of the Garside family $\cS$ as the same, except
that  the  above  square  is  restricted  to  the case where $f$, $ff'$ and
$ff'f''$  are in $\cS$, (which implies $f',f'',f'f''\in \cS$ since $\cS$ is
closed under right-divisors). We then have
\begin{theorem}\label{bicategory}
Let $\cC$ be a right-Noetherian category which admits conditional right-lcms and
has a Garside family $\cS$.
Then any representation of $\cS$ into a bicategory extends uniquely to a
representation of $\cC$ into the same bicategory.
\end{theorem}
\begin{proof}
The  proof goes exactly as in \cite{Deligne}, in that what must been proven
is  a  simple  connectedness  property  for  the set $E(g)$ 
of decompositions as a
product  of elements  of $\cS$  of an  arbitrary morphism $g\in\cC$--- this
generalizes  \cite[1.7]{Deligne}  and  is  used  in  the  same  way. In his
context,   Deligne  shows   more,  the   contractibility  of   the  set  of
decompositions;  on the other hand our proof, which follows a suggestion by
Serge  Bouc to use a version of  \cite[Lemma 6]{Bouc}, is simpler and holds
in our more general context.

Fix  $g\in\cC$  with  $g\notin  \CCCi$.  We  denote  by  $E(g)$  the set of
decompositions of $g$ into a product of elements of $\cS-\CCCi$.

Then $E(g)$  is a poset, the order being defined by
$$(g_1,\dots,g_{i-1},g_i,g_{i+1},\dots,g_n)>
(g_1,\dots,g_{i-1},a,b,g_{i+1},\dots,g_n)$$ if $ab=g_i \in \cS$.

We  recall  the  definition  of  homotopy  in a poset $E$ (a translation of the
corresponding notion in a simplicial complex isomorphic as a poset to $E$).
A  path from $x_1$ to $x_k$ in $E$ is a sequence $x_1\dotsm x_k$ where each
$x_i$  is comparable to  $x_{i+1}$. The composition  of paths is defined by
concatenation.  Homotopy,  denoted  by  $\sim$,  is  the finest equivalence
relation  on paths compatible  with concatenation and  generated by the two
following  elementary relations: $xyz\sim xz$ if $x\le y\le z$ and both $xyx\sim
x$ and $yxy\sim y$ when $x\le y$. Homotopy classes form a groupoid, as
the  composition  of  a  path  with  source  $x$ and of the inverse path is
homotopic  to  the  constant  path  at  $x$.  For  $x\in  E$  we  denote by
$\Pi_1(E,x)$ the fundamental group of $E$ with base point $x$, which is the
group of homotopy classes of loops starting from $x$.

A  poset $E$ is said to be {\em simply connected} if it is connected (there
is  a path linking  any two elements  of $E$) and  if the fundamental group
with some (or any) base point is trivial.

Note  that  a  poset  with  a  smallest  or  largest  element $x$ is simply
connected   since   any   path   $xyzt\dotsm x$   is   homotopic   to
$xyxzxtx\dotsm x$ which is homotopic to the trivial loop.

\begin{proposition} \label{Deligne} 
The set $E(g)$ is simply connected.
\end{proposition}
\begin{proof}
First  we prove a version  of a lemma from  \cite{Bouc} on order preserving
maps  between posets. For a poset $E$ we put $E_{\ge x}=\{x'\in E\mid x'\ge
x\}$,  which is a simply connected subposet  of $E$ since it has a smallest
element.  If $f:X\to Y$  is an order  preserving map it  is compatible with
homotopy (it corresponds to a continuous map between simplicial complexes),
so it induces a homomorphism $f^*:\Pi_1(X,x)\to \Pi_1(Y,f(x))$.

\begin{lemma}[Bouc]  \label{bouc}  Let  $f:X\to Y$ an  order preserving map
between  two posets. We assume that $Y$ is connected and that for any $y\in
Y$  the poset $\f y$ is connected  and non empty. Then $f^*$ is surjective.
If  moreover  $\f  y$  is  simply  connected  for  all $y$ then $f^*$ is an
isomorphism.
\end{lemma}
\begin{proof}
Let  us first show that $X$ is connected. Let $x,x'\in X$; we choose a path
$y_0\dotsm  y_n$ in $Y$ from $y_0=f(x)$ to $y_n=f(x')$. For $i=0,\dots,n$,
we  choose  $x_i\in\f{y_i}$  with  $x_0=x$  and  $x_n=x'$. Then if $y_i\geq
y_{i+1}$  we have $\f{y_i}\subset \f{y_{i+1}}$ so  that there exists a path
in  $\f{y_{i+1}}$ from  $x_i$ to  $x_{i+1}$; otherwise $y_i<y_{i+1}$, which
implies  $\f{y_i}\supset \f{y_{i+1}}$ and there  exists a path in $\f{y_i}$
from  $x_i$ to $x_{i+1}$. Concatenating these paths gives a path connecting
$x$ and $x'$.

We fix now $x_0\in X$. Let $y_0=f(x_0)$. We prove that
$f^*:\Pi_1(X,x_0)\to\Pi_1(Y,y_0)$  is  surjective.  Let  $y_0y_1\dotsm y_n$
with  $y_n=y_0$ be a loop in $Y$. We lift arbitrarily this loop into a loop
$x_0\dash\cdots\dash x_n$ in $X$ as above, (where $x_i\dash x_{i+1}$ stands
for  a path  from $x_i$  to $x_{i+1}$  which is  either in  $\f{y_i}$ or in
$\f{y_{i+1}}$).  Then  the  path  $f(x_0\dash  x_1\dash\cdots\dash  x_n)$ is
homotopic  to $y_0\dotsm y_n$; this can be seen by induction: let us assume
that   $f(x_0\dash   x_1\cdots\dash   x_i)$   is  homotopic  to  $y_0\dotsm
y_if(x_i)$;  then the same property holds for $i+1$: indeed $y_iy_{i+1}\sim
y_if(x_i)y_{i+1}$  as they are two paths in a simply connected set which is
either   $Y_{\ge   y_i}$   or   $Y_{\ge   y_{i+1}}$;   similarly   we  have
$f(x_i)y_{i+1}f(x_{i+1}) \sim f(x_i\dash x_{i+1})$. Putting things together
gives
$$
\begin{aligned}
y_0\dotsm y_iy_{i+1}f(x_{i+1})&\sim y_0y_1\dotsm y_if(x_i)y_{i+1}f(x_{i+1})\\
&\sim f(x_0\dash\cdots \dash x_i)y_{i+1}f(x_{i+1})\\
&\sim f(x_0\dash\cdots \dash x_i\dash x_{i+1}).
\end{aligned}
$$

We now prove injectivity of $f^*$ when all $\f{y}$ are simply connected.

We   first  prove  that  if   $x_0\dash  \cdots\dash  x_n$  and  $x'_0\dash
\cdots\dash x'_n$ are two loops lifting the same loop $y_0\dotsm y_n$, then
they  are  homotopic.  Indeed,  we  get  by induction on $i$ that $x_0\dash
\cdots\dash  x_i\dash x'_i$  and $x'_0\dash\cdots\dash  x'_i$ are homotopic
paths,  using the fact that $x_{i-1}$, $x_i$, $x'_{i-1}$ and $x'_i$ are all
in  the  same  simply  connected  sub-poset, namely either $\f{y_{i-1}}$ or
$\f{y_i}$.

It remains to prove that we can lift homotopies, which amounts to show that
if we lift as  above two loops which  differ by an elementary homotopy,
the  liftings are homotopic. If $yy'y\sim y$ is an elementary homotopy with
$y<y'$ (resp.\ $y>y'$), then $\f{y'}\subset\f{y}$ (resp.\
$\f{y}\subset\f{y'}$)  and the lifting of $yy'y$ constructed as above is in
$\f{y}$  (resp.\  $\f{y'}$)  so  is  homotopic  to  the  trivial  path.  If
$y<y'<y''$,  a lifting of $yy'y''$ constructed as above is in $\f{y}$ so is
homotopic to any path in $\f{y}$ with the same endpoints. 
\end{proof}

We  now prove Proposition \ref{Deligne}  by contradiction. If  it fails we choose $g\in
\cC$  minimal for proper right-divisibility  such that $E(g)$ is not simply
connected.

Let  $L$ be the set  of elements of $\cS-\CCCi$  which are left-divisors of
$g$. For any $I\subset L$, since the category admits conditional right-lcms and is
right-Noetherian, the elements  of $I$ have an lcm. We fix such an lcm $\Delta_I$.
Let $E_I(g)=\{(g_1,\dots,g_n)\in   E(g)\mid  \Delta_I\preccurlyeq  g_1\}$.  We
claim  that  $E_I(g)$  is  simply  connected  for  $I\neq\emptyset$. 
This is clear if $g\in\Delta_I\CCCi$, in which case
$E_I(g)=\{(g)\}$. Let us assume this is not the case. In the
following, if $\Delta_I\preccurlyeq a$, we denote by $a^I$ the element such
that $a=\Delta_I a^I$.
The set $E(g^I)$ is defined since $g\not\in\Delta_I\CCCi$.
We apply Lemma \ref{bouc} to the map $f: E_I(g)\to E(g^I)$
defined by $$(g_1,\dots,g_n)\mapsto \begin{cases}
(g_2,\dots,g_n)&\text{if        $g_1=\Delta_I$}\\        (g_1^I,g_2,\dots
g_n)&\text{otherwise}\\ \end{cases}.$$
This map preserves the order and any
set $\f{(g_1,\dots,g_n)}$ has a least element, namely
$(\Delta_I,g_1,\dots,g_n)$,  so is  simply connected.  As by minimality of
$g$  the set $E(g^I)$ is simply  connected Lemma \ref{bouc} implies that $E_I(g)$
is simply connected.

Let  $Y$ be the set of non-empty subsets of $L$. We now apply Lemma \ref{bouc} to
the  map  $f:E(g)\to  Y$  defined  by $(g_1,\dots,g_n)\mapsto \{s\in L\mid
s\preccurlyeq  g_1\}$, where $Y$ is ordered by inclusion. This map is order
preserving      since     $(g_1,\dots,g_n)<(g'_1,\dots,g'_n)$     implies
$g_1\preccurlyeq  g'_1$.  We  have  $\f{I}=E_I(g)$,  so  this set is simply
connected.  Since  $Y$,  having  a  greatest  element, is simply connected,
Lemma \ref{bouc} gives that $E(g)$ is simply connected, whence the proposition.
\end{proof}
\end{proof}

\part*{II. Deligne-Lusztig varieties and eigenspaces}

In  this part, we study the Deligne-Lusztig  varieties giving rise to a
Lusztig  induction functor $R_\bL^\bG$ and
generalize  them to  varieties attached to  elements of a ribbon
category.

In Section \ref{eigenspaces and roots} we consider the particular ribbons
describing varieties which play a role in the Brou\'e conjectures;
they are associated with maximal eigenspaces of elements of the Weyl group.

Finally  in Section \ref{section 10} we spell out the geometric form of the
Brou\'e   conjectures,  describing  how  the   action  on  the  $\ell$-adic
cohomology  of the endomorphisms of our varieties coming from 
the conjugacy category
of  the ribbon  category should  factorize through  a  cyclotomic Hecke
algebra.

\section{Parabolic Deligne-Lusztig varieties}\label{section 8}
Let  $\bG$ be a connected reductive algebraic group over $\Fpbar$, and let $F$
be  an isogeny on $\bG$  such that some power  $F^\delta$ is a Frobenius for a
split $\BF_{q^\delta}$-structure (this defines a positive real number $q$ such 
that $q^\delta$ is an integral power of $p$).

Let  $\bL$ be  an $F$-stable  Levi subgroup  of a (non-necessarily $F$-stable)
parabolic  subgroup $\bP$ of  $\bG$ and let  $\bP=\bL\bV$ be the corresponding
Levi decomposition of $\bP$. Let
$$
\begin{aligned}
\bX_\bV&=\{g\bV\in\bG/\bV\mid g\bV\cap F(g\bV)\ne\emptyset\}=
\{g\bV\in\bG/\bV\mid g\inv\lexp F g\in \bV\lexp F\bV\}\\
&\simeq\{g\in \bG\mid g\inv\lexp Fg\in\lexp F\bV\}/
(\bV\cap \lexp F\bV).
\end{aligned}
$$
On this variety  $\bG^F$ acts by left-multiplication and $\bL^F$ acts by
right-multiplication.

We choose a prime number $\ell\ne p$. Then the virtual
$\bG^F$-module-$\bL^F$  given  by  $M=\sum_i  (-1)^i H^i_c(\bX_\bV,\Qlbar)$
defines  the Lusztig  induction $R_\bL^\bG$  which by definition maps
an $\bL^F$-module $\lambda$ to $M\otimes_{\Qlbar\bL^F}\lambda$.

The map  $g\bV\mapsto g\bP$ makes  $\bX_\bV$ an $\bL^F$-torsor over
$$
\begin{aligned}
\bX_\bP&=\{g\bP\in\bG/\bP\mid g\bP\cap F(g\bP)\ne\emptyset\}=
\{g\bP\in\bG/\bP\mid g\inv\lexp F g\in \bP\lexp F\bP\}\\
&\simeq\{g\in \bG\mid g\inv\lexp Fg\in\lexp F\bP\}/
(\bP\cap \lexp F\bP),
\end{aligned}
$$
a $\bG^F$-variety such that $R_\bL^\bG(\Id)=\sum_i (-1)^i
H^i_c(\bX_\bP,\Qlbar)$. The variety $\bX_\bP$ is the prototype of the
varieties we want to study.

Let  $\bT\subset\bB$  be  a  pair  of  an  $F$-stable  maximal torus and an
$F$-stable  Borel subgroup of  $\bG$. With this  choice is associated a basis
$\Pi$  of the  root system  $\Phi$ of  $\bG$ with  respect to  $\bT$, and a
Coxeter system $(W,S)$ for   the   Weyl   group  $W=N_\bG(\bT)/\bT$.  Let
$X_\BR=X(\bT)\otimes\BR$ where $X(\bT)$ is the group of rational characters of
the torus $\bT$. On the vector space $X_\BR$, the isogeny $F$ acts
as  $q\phi$ where $\phi$  is of order  $\delta$ and stabilizes the positive
cone $\BR^+\Pi$; we will still denote by $\phi$ the induced automorphism of
$(W,S)$.

To  a  subset  $I\subset\Pi$  corresponds  a  subgroup  $W_I\subset  W$,  a
parabolic   subgroup  $\bP_I=\coprod_{w\in  W_I}\bB  w\bB$,  and  the  Levi
subgroup $\bL_I$ of $\bP_I$ which contains $\bT$.

Given  any  $\bP=\bL\bV$  as  in the beginning of this section, where  $\bL$  is $F$-stable, there exists
$I\subset\Pi$  such that $(\bL,\bP)$ is $\bG$-conjugate to $(\bL_I,\bP_I)$; if
we  choose the conjugating  element such that  it conjugates a maximally split
torus  of $\bL$ to  $\bT$ and a  rational Borel subgroup  of $\bL$ containing
this  torus to $\bB\cap\bL_I$,  then this element  conjugates $(\bL,\bP,F)$ to
$(\bL_I,\bP_I,\dot   wF)$   where   $\dot   w\in   N_\bG(\bT)$  is  such  that
$\lexp{w\phi}I=I$, where $w$ is the image of $\dot w$ in $W$.

It will be
convenient  to consider $I$ as a subset of $S$ instead of a subset of $\Pi$;
the  condition on $w$ must then be stated as ``$I^w=\lexp\phi I$ and $w$ is
$I$-reduced''. Note that $w$ is then also reduced-$\lexp\phi I$.
Via the above conjugation,  the  variety  $\bX_\bP$  is
isomorphic to the variety $$\bX(I,w\phi)=\{g\bP_I\in\bG/\bP_I\mid g\inv\lexp F
g\in \bP_Iw\lexp F\bP_I\}.$$ We will denote by $\bX_\bG(I,w\phi)$ this variety
when  there is a possible  ambiguity on the group.  If we denote by $\bU_I$
the   unipotent   radical   of   $\bP_I$,   we  have  $\dim  \bX(I,w\phi)=\dim
\bU_I-\dim(\bU_I\cap  \lexp{wF}\bU_I)=l(w)$, the last equality since $w$ is
reduced-$\lexp\phi I$. The  $\ell$-adic cohomology of
the   variety  $\bX(I,w\phi)$  gives  rise   to  the  Lusztig  induction  from
$\bL_I^{\dot wF}$ to $\bG^F$ of the trivial representation;
to avoid ambiguity on the isogenies involved,
we   will  sometimes  denote  this   Lusztig  induction  by  $R_{\bL_I,\dot
wF}^{\bG,F}(\Id)$.

\begin{definition}\label{relative position}
We say  that a pair $(\bP,\bQ)$ of parabolic subgroups is in
relative  position $(I,w,J)$, where $I,J\subset S$ and  $w\in W$,
if $(\bP,\bQ)$ is $\bG$-conjugate to $(\bP_I,\lexp w\bP_J)$. We denote this as
$\bP\xrightarrow{I,w,J}\bQ$.
\end{definition}
Since any pair $(\bP,\bQ)$ of parabolic subgroups share a common maximal torus,
it has a relative position $(I,w,J)$ where $I,J$ is uniquely determined as
well as the double coset $W_I w W_J$. 

Let  $\cP_I$ be  the variety  of parabolic  subgroups conjugate  to $\bP_I$;
this variety is isomorphic to $\bG/\bP_I$.  Via the map 
$g\bP_I\mapsto\lexp g\bP_I$ we have an  isomorphism
$$\bX(I,w\phi)\simeq\{\bP\in\cP_I\mid  \bP\xrightarrow{I,w,\lexp
\phi I}\lexp  F\bP\};$$ 
it is a variety  over $\cP_I\times\cP_{\lexp \phi I}$ by the
first and second projection.

\subsection*{The varieties $\cO$ attached to $B^+(\cI)$.}
In order to have a rich enough monoid of endomorphisms (see Definition
\ref{Dv}), we need to generalize the pairs $(I,w\phi)$ which label our
varieties to the larger set of morphisms of the category $B^+(\cI)$
of Section \ref{artin ribbons}, where $\cI$ is the conjugacy class in $B^+$
of the lift $\bI$ of $I$.

In order to do this, we define in this subsection a representation of $B^+(\cI)$ into the
bicategory  $\bX$ of varieties  over $\cP_I\times \cP_J$,  where $I,J$ vary
over  $\cI$. The bicategory  $\bX$ has $0$-cells  which are the elements of
$\cI$,  has  1-cells  with  domain  $\bI$  and codomain $\bJ$ which are the
$\cP_I\times  \cP_J$-varieties and  has 2-cells  which are  isomorphisms of
$\cP_I\times \cP_J$-varieties. For $\bI,\bJ\in\cI$ we denote by $V(\bI,\bJ)$ the category whose
objects  (resp.\ morphisms) are the 1-cells  with domain $\bI$ and codomain
$\bJ$  (resp.\ the 2-cells  between them); in  other words, $V(\bI,\bJ)$ is
the  category of $\cP_I\times\cP_J$-varieties endowed with the isomorphisms
of   $\cP_I\times\cP_J$-varieties.  The  horizontal  composition  bifunctor
$V(\bI,\bJ)\times V(\bJ,\bK)\to V(\bI,\bK)$ is given by the fibered product
over  $\cP_J$.  The  vertical  composition  is  given by the composition of
isomorphisms.

The representation of $B^+(\cI)$ in $\bX$ we construct will be denoted by $T$,
following  the notations of Section  \ref{bicategories}. For
$\bI\xrightarrow\bb\bJ\in B^+(\cI)$, we will also write
$\cO(\bI,\bb)$ for $T(\bI\xrightarrow\bb\bJ)$, to lighten the notation.
We  first define $T$ on the Garside family $\cS$ of $B^+(\cI)$.
\begin{definition}
For  $(\bI\xrightarrow{\bw}\bJ)\in  \cS$
we define $\cO(\bI,\bw)$ to be the variety
$\{(\bP,\bP')\in\cP_I\times\cP_J\mid\bP\xrightarrow{I,w,J}\bP'\}$,
where $I$, $w$, $J$ are the images in $W$ of $\bI$, $\bw$, $\bJ$.
\end{definition}

The following lemma constructs the isomorphism
$T(f)T(g)\xrightarrow\sim T(fg)$ when $f,g,fg\in \cS$:
\begin{lemma}\label{1.1}
Let         $(\bI\xrightarrow{\bw_1}\bI_2\xrightarrow{\bw_2}\bJ)=(
\bI\xrightarrow{\bw}\bJ)$        where
$\bw=\bw_1\bw_2\in\bW$   be  a  defining  relation  of  $B^+(\cI)$.  Then
$(p',p''):\cO(I,\bw_1)\times_{\cP_{I_2}}\cO(I_2,\bw_2)
\xrightarrow\sim\cO(I,\bw_1\bw_2)$ is an isomorphism, where $p'$ and $p''$
are respectively the first and last projections.
\end{lemma}
\begin{proof}
First notice that for two parabolic subgroups
$(\bP',\bP'')\in\cP_I\times\cP_J$ we have
$\bP'\xrightarrow{I,w,J}\bP''$ if and only if the pair $(\bP',\bP'')$
is conjugate to a pair containing termwise the pair $(\bB,\lexp w\bB)$.
This shows that if $\bP'\xrightarrow{I,w_1,I_2}\bP_1$ and
$\bP_1\xrightarrow{I_2,w_2,J}\bP''$ then $\bP'\xrightarrow{I,w_1w_2,J}\bP''$,
so $(p',p'')$ goes to the claimed variety.

Conversely, we have to show that given $\bP'\xrightarrow{I,w,J}\bP''$ there
is a unique $\bP_1$ such that
$\bP'\xrightarrow{I,w_1,I_2}\bP_1\xrightarrow{I_2,w_2,J}\bP''$.  The  image
of $(\bB,\lexp w\bB)$ by the conjugation which sends $(\bP_I,\lexp w\bP_J)$
to $(\bP',\bP'')$ is a pair of Borel subgroups
$(\bB'\subset\bP',\bB''\subset\bP'')$     in     position     $w$.    Since
$l(w_1)+l(w_2)=l(w)$,  there is a  unique Borel subgroup  $\bB_1$ such that
$\bB'\xrightarrow{w_1}\bB_1\xrightarrow{w_2}\bB''$.  The  unique  parabolic
subgroup  of type $I_2$ containing $\bB_1$ has the desired relative
positions,  so $\bP_1$  exists. And  any other  parabolic subgroup $\bP'_1$
which has the desired relative positions contains a Borel subgroup $\bB'_1$
such  that  $\bB'\xrightarrow{w_1}\bB'_1\xrightarrow{w_2}\bB''$  (take  for
$\bB'_1$  the  image  of  $\lexp{w_1}\bB$  by  the  conjugation  which maps
$(\bP_I,\lexp{w_1}\bP_{I_2})$   to  $(\bP',\bP'_1)$),  which  implies  that
$\bB'_1=\bB_1$ and thus $\bP'_1=\bP_1$.
Thus our map is bijective on points. To show it is an isomorphism, it is
sufficient to check that its target is a normal variety, which is given by
\begin{lemma}\label{smooth}
For $(\bI\xrightarrow{\bw}\bJ)\in \cS$ the variety $\cO(\bI,\bw)$ is smooth.
\end{lemma}
\begin{proof}
Consider the locally trivial fibrations with smooth fibers given by
$\bG\times\bG\xrightarrow p\cP_I\times\cP_J:(g_1,g_2)\mapsto
(\lexp{g_1}\bP_I,\lexp{g_2w}\bP_J)$ and $\bG\times\bG\xrightarrow q\bG:
(g_1,g_2)\mapsto g_1\inv g_2$. It is easy to check that $\cO(\bI,\bw)=
p(q\inv(\lexp w\bP_J))$ thus by for example \cite[2.2.3]{DMR} it is smooth.
\end{proof}
\end{proof}

From  the above lemma we see also  that the square \ref{bicat} commutes for
elements   of  $\cS$,   since  the   isomorphism  ``forgetting  the  middle
parabolic'' has clearly the corresponding property. We have thus defined
a representation $T$ of $\cS$ in $\bX$.

The extension of  $T$  to  the  whole  of  $B^+(\cI)$  associates with a
composition
$\bI\xrightarrow{\bw_1}\bI_2\to\dots\to\bI_k\xrightarrow{\bw_k}\bJ$  with
$\bw_i\in\bW$ the variety 
$$
\cO(\bI,\bw_1)\times_{\cP_{I_2}}\dots\times_{\cP_{I_k}}\cO(\bI_k,\bw_k)
=\{(\bP_1,\dots,\bP_{k+1})\mid
\bP_i\xrightarrow{I_i,w_i,I_{i+1}}\bP_{i+1}\},  
$$
where   $I_1=I$   and
$I_{k+1}=J$.  It  is  a  $\cP_I\times\cP_J$-variety  via the first and last
projections mapping $(\bP_1, \dots,\bP_{k+1})$ respectively to $\bP_1$ and
$\bP_{k+1}$,  and Lemma \ref{1.1} shows that  up to isomorphism it does not
depend   on  the   chosen  decomposition   of  $\bI\xrightarrow{\bw_1\dotsm
\bw_k}\bJ$. Theorem \ref{bicategory} shows that there is actually a unique
isomorphism between the various models attached to different decompositions,
so $T$ associates a well-defined variety to any element of $B^+(\cI)$.

\begin{definition}
For $\bI\xrightarrow{\bb}\bJ\in B^+(\cI)$ 
we denote  by $\cO(\bI,\bb)$ the variety defined by Theorem \ref{bicategory}.
For any decomposition $(\bI\xrightarrow\bb\bJ)=
(\bI_1\xrightarrow{\bw_1}\bI_2\to\cdots\xrightarrow{\bw_k}\bJ)$ into
elements of $\cS$ it has the model $\{(\bP_1,\dots,\bP_{k+1})\mid
\bP_i\xrightarrow{I_i,w_i,I_{i+1}}\bP_{i+1}\}$.
\end{definition}
The variety $\cO(\bI,\bb)$ is endowed with a natural action of $\bG$ by
simultaneous conjugation of the $\bP_i$.

\subsection*{The Deligne-Lusztig varieties attached to $B^+(\cI)$.}
The  automorphism $\phi$ lifts naturally to  an automorphism of $B^+$ which
stabilizes $\bS$, which we will still denote by $\phi$, by abuse of notation.
For   $(\bI\xrightarrow{\bw}\lexp\phi\bI)\in\cS$,  the variety
$\bX(I,w\phi)$  is  the  intersection  of  $\cO(\bI,\bw)$ with the graph of
$F$,  that is,  points whose image under
$(p',p'')$  has  the  form  $(\bP,\lexp  F\bP)$.  Via the correspondance
between $\phi$-conjugacy and conjugacy in the coset, we interpret
$\bI\xrightarrow\bw\lexp\phi\bI$ as the endomorphism
$\bI\xrightarrow{\bw\phi}\bI$ in $B^+\phi(\cI)$. More generally,
\begin{definition}
Let  $\bI\xrightarrow{\bb\phi}\bI$ be an  endomorphism of $B^+\phi(\cI)$; we
define the variety $\bX(\bI,\bb\phi)$ as the intersection of $\cO(\bI,\bb)$
with the graph of $F$.
\end{definition}
The action of $\bG$ on $\cO(\bI,\bb)$ restricts to an action of $\bG^F$ on
$\bX(\bI,\bb\phi)$.
This last  variety  may   be  interpreted  as   an  ``ordinary''  parabolic
Deligne-Lusztig variety in a group which is a restriction of scalars:
\begin{proposition}\label{descente}
For any decomposition $(\bI\xrightarrow\bb\lexp\phi\bI)=
(\bI_1\xrightarrow{\bw_1}\bI_2\to\cdots\xrightarrow{\bw_k}\lexp  \phi\bI)$ in
elements of $\cS$ the variety $\bX(\bI,\bb\phi)$ has the model
$\{(\bP_1,\dots,\bP_{k+1})\mid
\bP_i\xrightarrow{I_i,w_i,I_{i+1}}\bP_{i+1}\text{ and }
\bP_{k+1}=F(\bP_1)\}$.
Let $F_1$ be the isogeny of $\bG^k$ defined by
$F_1(g_1,\dots,g_k)=(g_2,\dots,g_k,F(g_1))$ and let $\phi_1$ be the
corresponding automorphism of $W^k$.
Then the above model is isomorphic to
$\bX_{\bG^k}(I_1\times\dots\times I_k,(w_1,\dots,w_k)\phi_1)$. By this
isomorphism the action of $F^\delta$ corresponds to that of $F_1^{k\delta}$
and the action of $\bG^F$ corresponds to that of
$(\bG^k)^{F_1}$---the isomorphism $\bG^F\xrightarrow\sim(\bG^k)^{F_1}$ is via
the diagonal embedding.
\end{proposition}
\begin{proof}
That $\bX(\bI,\bb\phi)$ has the model given above is a consequence of the
analogous statement for $\cO(\bI,\bb)$.

An element $\bP_1\times\dots\times\bP_k\in
\bX_{\bG^k}(I_1\times\dots\times I_k,(w_1,\dots,w_k)\phi_1)$ by definition 
satisfies
$$\bP_1\times\dots\times\bP_k\xrightarrow{I_1\times\cdots I_k,
(w_1,\dots,w_k),I_2\times\cdots I_k\times\lexp
\phi I_1}\bP_2\times\dots\times\bP_k\times\lexp F\bP_1$$
thus is equivalently given by a sequence $(\bP_1,\dots,\bP_{k+1})$
such that $\bP_i\xrightarrow{I_i,w_i,I_{i+1}}\bP_{i+1}$ with 
$\bP_{k+1}=\lexp F\bP_1$ and $I_{k+1}=\lexp \phi I_1$, which is the same as an 
element
$$(\bP_1,\dots,\bP_{k+1})\in\cO(\bI_1,\bw_1)\times_{\cP_{I_2}}
\cO(\bI_2,\bw_2)\dots\times_{\cP_{I_k}}\cO(\bI_k,\bw_k)$$ such that 
$\bP_{k+1}=\lexp  F\bP_1$.
But this is a model of $\bX_\bG(\bI,\bb \phi)$ as explained above.

One checks easily that this sequence of identifications is compatible with
the actions of $F^\delta$ and $\bG^F$ as described by the proposition.
\end{proof}
\begin{proposition}\label{irreducibility}
The variety $\bX(\bI,\bb \phi)$ is irreducible if and only if $\bI\cup\supp(\bb)$
meets all the orbits of $\phi$ on $\bS$, where $\supp(\bb)$ is the support of
$\bb$ (see the proof of Lemma \ref{all parabolic}).
\end{proposition}
\begin{proof}
This is, using Proposition \ref{descente}, an immediate translation in our setting of the
result \cite[Theorem 2]{BR} of Bonnaf\'e-Rouquier.
\end{proof}
\subsection*{The varieties $\tilde\bX(\bI,\bw\phi)$}
The   conjugation  which  transforms  $\bX_\bP$  into  $\bX(I,w\phi)$  maps
$\bX_\bV$ to the $\bG^F$-variety-$\bL_I^{\dot wF}$ given by
\begin{equation}\label{tildeXwphi}
\tilde\bX(I,\dot  wF)=\{g\bU_I\in\bG/\bU_I\mid g\inv\lexp Fg\in\bU_I \dot
w\lexp F\bU_I\},
\end{equation}
where  $\dot  w$  is  a  representative  of  $w$ (any representative can be
obtained  by choosing  an appropriate  conjugation). The map $g\bU_I\mapsto
g\bP_I$  makes  $\tilde\bX(I,\dot  wF)$  a  $\bL_I^{\dot  wF}$-torsor  over
$\bX(I,w\phi)$. We have written $\dot w$ and $F$ together since 
the  variety  depends  only  on  the  product  $\dot w F\in
N_\bG(\bT)  \rtimes\genby F$; we will write $\tilde\bX(I,\dot w\cdot F)$ to
separate  the Frobenius  endomorphism from  the representative  of the Weyl
group  element when needed, in  the case where the  ambient group is a Levi
subgroup with Frobenius endomorphism of the form $\dot x F$.

In this section, we define a variety $\tilde\bX(\bI,\bw\phi)$ which generalizes
$\tilde\bX(I,\dot  wF)$  by  replacing  $\dot  w$  by elements of the braid
group. Since $\dot w$ represents a choice of a lift of $w$ to $N_\bG(\bT)$,
we  have  to  make  uniformly  such  choices  for all elements of the braid
group, which we do by using a ``Tits homomorphism''.

First, when $\bw\in\bW$, we define a variety $\tilde\cO(I,\dot w)$
``above''   $\cO(\bI,\bw)$   such   that   $\tilde\bX(I,\dot  wF)$  is  the
intersection  of $\tilde\cO(I,w)$ with the graph of $F$, and then we extend
this construction to $\cB^+(\cI)$.

\begin{definition}
Let $(\bI\xrightarrow{\bw}\bJ)\in\cS$, and let $\dot w\in N_\bG(\bT)$ be a
representative of $w$.  We define $\tilde\cO(I,\dot w)=
\{(g\bU_I,g'\bU_J)\in\bG/\bU_I\times\bG/\bU_J\mid g\inv g'\in\bU_I\dot
w\bU_J\}$.
\end{definition}
The variety $\tilde\cO(I,\dot w)$ has a left action of $\bG$ by simultaneous
translation and a right action of $\bL_I$ by
$(g\bU_I,g'\bU_J)\mapsto(gl\bU_I,g'l^{\dot w}\bU_J)$.

We can prove an analogue of Lemma \ref{1.1}.
\begin{lemma}\label{2.2}
Let     $(\bI\xrightarrow{\bw_1}\bI_2\xrightarrow{\bw_2}\bJ)=(
\bI\xrightarrow{\bw_1\bw_2}\bJ)$        where
$\bw_1\bw_2\in\bW$   be  a  defining  relation  of  $B^+(\cI)$, and let
$\dot w_1,\dot w_2$ be representatives of the images of $\bw_1$ and $\bw_2$ in
$W$. Then
$(p',p''):\tilde\cO(I,\dot w_1)\times_{\bG/\bU_{I_2}}\tilde\cO(I_2,\dot w_2)
\xrightarrow\sim\tilde\cO(I,\dot w_1\dot w_2)$ is an isomorphism
where $p'$ and $p''$ are the first and last projections.
\end{lemma}
\begin{proof}
We  first  note  that  if  $\bI\xrightarrow{\bw}\bJ\in B^+(\cI)$  and  $\dot w$ is a
representative  in  $N_\bG(\bT)$  of  the  image  of  $\bw$  in  $W$,  then
$\bU_I\dot  w\bU_J$ is  isomorphic by  the product  morphism to  the direct
product of varieties $(\bU_I\cap \lexp w\bU_J^-)\dot w\times\bU_J$, where
$\bU_J^-$ is the unipotent radical of the parabolic subgroup opposed to
$\bP_J$ containing $\bT$.
We now use the lemma:
\begin{lemma}
Under the assumptions of Lemma \ref{2.2}, the product gives an isomorphism
$(\bU_I\cap\lexp{\dot w_1}\bU_{I_2}^-)\dot w_1\times
(\bU_{I_2}\cap\lexp{\dot w_2}\bU_J^-)\dot w_2\xrightarrow\sim
(\bU_I\cap\lexp{\dot w_1\dot w_2}\bU_J^-)\dot w_1\dot w_2$.
\end{lemma}
\begin{proof}
Since $w$ is $I$-reduced and $I^w=J$,  we have
$\bU_I\cap \lexp w\bU_J^-=\prod_{-\alpha\in\lexp w N(w)}\bU_\alpha$
as a product of root subgroups,
where $N(w)=\{\alpha\in\Phi^+\mid \lexp w\alpha\in\Phi^-\}$.
The lemma is then a consequence of the equality $N(w_1)^{w_2}\coprod
N(w_2)=N(w_1w_2)$ when $l(w_1)+l(w_2)=l(w_1w_2)$.
\end{proof}
The   lemma  proves  in  particular   that  if  $g_1\inv  g_2\in  \bU_I\dot
w_1\bU_{I_2}$  and  $g_2\inv  g_3\in  \bU_{I_2}\dot w_2\bU_J$ then $g_1\inv
g_3\in    \bU_I\dot    w_1\bU_{I_2}\dot    w_2\bU_J=   (\bU_I\cap\lexp{\dot
w_1}\bU_{I_2}^-)\dot    w_1    (\bU_{I_2}\cap\lexp{\dot    w_2}\bU_J^-)\dot
w_2\bU_J=  (\bU_I\cap\lexp{\dot w_1\dot  w_2}\bU_J^-)\dot w_1\dot w_2\bU_J=
\bU_I\dot  w_1\dot w_2\bU_J$,  so the  image of  the morphism $(p',p'')$ in
Lemma \ref{2.2} is indeed in the variety $\tilde\cO(I,\dot w_1\dot w_2)$.

Conversely, we have to show that given $(g_1\bU_I,g_3\bU_\bJ)\in\tilde\cO
(I,\dot w_1\dot w_2)$, there exists a unique $g_2U_{I_2}$ such that 
$(g_1\bU_I,g_2\bU_{I_2})\in\tilde\cO(I,\dot w_1)$
and $(g_2\bU_{I_2},g_3\bU_{I_3})\in\tilde\cO(I_2,\dot w_2)$. The varieties
involved being invariant by left-translation by $\bG$, it is enough to solve
the problem when $g_1=1$. Then we have $g_3\in\bU_I\dot w_1\dot w_2\bU_J$,
and the conditions for $g_2\bU_{I_2}$ is that 
$g_2\bU_{I_2}\subset\bU_I\dot w_1\bU_{I_2}$. Any such coset has then a unique
representative in $(\bU_I\cap\lexp{\dot w_1}\bU_{I_2}^-)\dot w_1$ and we will
look for such a representative $g_2$. But we must have
$g_2\inv g_3\in \bU_{I_2}\dot w_2\bU_J=
(\bU_{I_2}\cap\lexp{\dot w_2}\bU_J^-)\dot w_2\bU_J$ and since by the lemma
the product  gives an isomorphism between
$(\bU_I\cap\lexp{\dot w_1}\bU_{I_2}^-)\dot w_1 \times
(\bU_{I_2}\cap\lexp{\dot w_2}\bU_J^-)\dot w_2\bU_J$ and
$\bU_I\dot w_1\dot w_2\bU_J$, the element $g_3$ can be decomposed in one and
only one way in a product $g_2(g_2\inv g_3)$ satisfying the conditions.
To conclude as in Lemma \ref{1.1} we show that the variety $\tilde\cO(\bI,\dot
w_1\dot  w_2)$ is smooth. An argument similar to the proof of Lemma \ref{smooth},
replacing  $\cP_I$ and  $\cP_J$ by  $\bG/\bU_I$ and $\bG/\bU_J$ respectively
gives the result.
\end{proof}

The isomorphism of Lemma \ref{2.2} is compatible with the action of $\bG$
and of $\bL_I$, $\bL_{I_2}$ respectively.

We will now use a Tits homomorphism, which is a homomorphism $B\xrightarrow
t  N_\bG(\bT)$ which  factors the  projection $B\to  W$ ---the existence of
such a homomorphism is
proved  in \cite{Tits}. Theorem \ref{bicategory} implies that, setting
$T(\bI\xrightarrow\bw\bJ)=\tilde\cO(I,t(\bw))$ for
$(\bI\xrightarrow\bw\bJ)\in\cS$   and  replacing  Lemma   \ref{1.1}  by  Lemma
\ref{2.2},  we can define a representation  of $B^+(\cI)$ in the bicategory
$\tilde\bX$  of varieties above $\bG/\bU_I\times\bG/\bU_J$ for $I,J\in\cI$.
\begin{definition}
The  above  representation  defines  for  any  $\bI\xrightarrow{\bb}\bJ\in
B^+(\cI)$ a variety $\tilde\cO(\bI,\bb)$ which  for any  decomposition
$(\bI\xrightarrow{\bb}\bJ)=(\bI\xrightarrow{\bw_1}\bI_2\to\dots\to\bI_k
\xrightarrow{\bw_k}\bJ)$ into elements of $\cS$ has the model
$\tilde\cO(I,t(\bw_1))\times_{\bG/\bU_{I_2}}
\dots\times_{\bG/\bU_{I_k}}\tilde\cO(I_k,t(\bw_k))$.
\end{definition}
By the remarks after Lemma \ref{2.2} the variety $\tilde\cO(\bI,\bb)$ affords
a natural left action of $\bG$ and right action of $\bL_I$.

\begin{proposition}   There  exists  a  Tits   homomorphism  $t$  which  is
$F$-equivariant, that is such that $t(\phi(\bb))=F(t(\bb))$.
\end{proposition}
\begin{proof} With any simple reflection $s\in S$ is associated a
quasi-simple  subgroup $\bG_s$  of rank  1 of  $\bG$, generated by the root
subgroups  $\bU_{\alpha_s}$ and $\bU_{-\alpha_s}$; the 1-parameter subgroup
of  $\bT$  given  by  $\bT\cap\bG_s$  is  a  maximal  torus  of $\bG_s$. By
\cite[Theorem  4.4]{Tits} if  for any  $s\in S$  we choose a representative
$\dot  s$ of $s$  in $\bG_s$, then  these representatives satisfy the braid
relations,  which implies that  $\bs\mapsto \dot s$  induces a well defined
Tits  homomorphism. We claim that if $s$ is fixed by some power $\phi^d$ of
$\phi$  then there exists $\dot  s\in\bG_s$ fixed by $F^d$;  we then get an
$F$-equivariant  Tits homomorphism by choosing arbitrarily $\dot s$ for one
$s$  in each orbit of  $\phi$. If $s$ is  fixed by $\phi^d$ then $\bG_s$ is
stable  by  $F^d$;  the  group  $\bG_s$  is  isomorphic to either $SL_2$ or
$PSL_2$ and $F^d$ is a Frobenius endomorphism of this group. In either case
the  simple reflection $s$ of $\bG_s$ has an $F^d$-stable representative in
$N_{\bG_s}(\bT\cap\bG_s)$, whence our claim.
\end{proof}

\begin{notation}
We assume now that we have chosen, once and for all, an $F$-equivariant Tits
homomorphism $t$ which is used to define the varieties
$\tilde\cO(\bI,\bb)$. 
\end{notation}
The equivariance of $t$ allows to extend it to a morphism
$B^+\rtimes\genby\phi\to N_\bG(\bT)\rtimes\genby F$---note that
here our convention that $\genby \phi$ is infinite order is useful,
since $F$ is of infinite order. This allows to extend $t$ by 
$t(\phi)=F$ thus we can write indifferently $t(\bb)F$ or $t(\bb\phi)$.
\begin{definition}
For any endomorphism $(\bI\xrightarrow{\bb\phi}\bI)\in B^+\phi(\cI)$ we define
$\tilde \bX(\bI,\bb \phi)=\{x\in\tilde\cO(\bI,\bb)\mid p''(x)=F(p'(x))\}$.
\end{definition}
The action of $\bL_I$ on $\tilde\cO(\bI,\bb)$ restricts to an action
of $\bL_I^{t(\bb\phi)}$ on $\tilde \bX(\bI,\bb \phi)$,
compatible with the first projection $\tilde\bX(\bI,\bb\phi)\to \bG/\bU_{I}$.

When  $\bw\in\bW$ we have  $\tilde \bX(\bI,\bw \phi)= \tilde\bX(I,t(\bw\phi))$,
the variety defined in \ref{tildeXwphi} for $\dot wF=t(\bw\phi)$.
We have the following analogue of Proposition \ref{descente} for
$\tilde\bX(\bI,\bb\phi)$.
\begin{proposition}\label{descente pour X tilde}
Let
$\bI=\bI_1\xrightarrow{\bw_1}\bI_2\to\dots\to\bI_k\xrightarrow{\bw_k}\lexp 
\phi\bI$ be a decomposition into elements of $\cS$ of
$\bI\xrightarrow{\bb}\lexp\phi\bI\in B^+(\cI)$,
let $F_1$ be the isogeny of $\bG^k$ as in Proposition \ref{descente}.

Then $\tilde\bX_\bG(\bI,\bb \phi) \simeq\tilde\bX_{\bG^k}(I_1\times
\dots\times   I_k,  (t(\bw_1),\dots,t(\bw_k))F_1)$.
By this isomorphism the action of $F^\delta$ corresponds to that
of   $F_1^{k\delta}$, the  action  of  $\bG^F$  corresponds  to  that  of
$(\bG^k)^{F_1}$, and the action of $\bL_I^{t(\bb\phi)}$ corresponds to that of
$(\bL_{I_1}\times\dots\times\bL_{I_k})^{(t(\bw_1),\dots,t(\bw_k))F_1}$.
\end{proposition}
\begin{proof}
An element $x_1\bU_{I_1}\times\dots\times x_k\bU_{I_k}\in
\tilde\bX_{\bG^k}(I_1\times\dots\times I_k,(t(\bw_1),\dots,t(\bw_k))F_1)$
by definition satisfies
$(x_i\bU_{I_i},x_{i+1}\bU_{I_{i+1}})\in\tilde\cO(I_i,t(\bw_i))$ 
for $i=1,\dots,k$, where  we have put
$I_{k+1}=\lexp FI_1$ and $x_{k+1}\bU_{I_k+1}=\lexp F(x_1\bU_{I_1})$.
This is the same as an element in the intersection of
$\tilde\cO(\bI_1,\bw_1)\times_{\bG/\bU_{I_2}}
\tilde\cO(\bI_2,\bw_2)\dots\times_{\bG/\bU_{I_k}}\tilde\cO(\bI_k,\bw_k)$
with the graph of $F$. Since, by definition, we have
$$\tilde\cO(\bI,\bb)\simeq
\tilde\cO(\bI_1,\bw_1)\times_{\bG/\bU_{I_2}}
\tilde\cO(\bI_2,\bw_2)\dots\times_{\bG/\bU_{I_k}}
\tilde\cO(\bI_k,\bw_k),$$ 
via  this last isomorphism we get an element of
$\tilde\cO(\bI,\bb)$ which is in $\tilde\bX_\bG(\bI,\bb \phi)$.

One checks easily that this sequence of identifications is compatible with
the actions of $F^\delta$, of  $\bG^F$ and of $\bL_I^{t(\bb\phi)}$
as described by the proposition.
\end{proof}

\begin{lemma}\label{tildeX torsor}
For any endomorphism $(\bI\xrightarrow{\bb\phi}\bI)\in B^+\phi(\cI)$, there is a natural
projection $\tilde\bX(\bI,\bb\phi)\xrightarrow\pi\bX(\bI,\bb\phi)$ which makes
$\tilde\bX(\bI,\bb\phi)$ a $\bL_I^{t(\bb\phi)}$-torsor over $\bX(\bI,\bb\phi)$.
\end{lemma}
\begin{proof}
Let $\bI\xrightarrow{\bw_1}\bI_2\to\dots\to\bI_r
\xrightarrow{\bw_r}\lexp\phi\bI$ 
be a decomposition into elements of $\cS$ of
$\bI\xrightarrow{\bb}\lexp\phi\bI$,    so   that   $\tilde\bX(\bI,\bb\phi)$
identifies with the set of sequences
$(g_1\bU_I,g_2\bU_{I_2},\dots,g_r\bU_{I_r})$ such that $g_j\inv g_{j+1}\in
\bU_{I_j}t(\bw_j)\bU_{I_{j+1}}$   for  $j<r$  and  $g_r\inv  \lexp  Fg_1\in
\bU_{I_r}t(\bw_r) \bU_{\lexp \phi I}$. We define $\pi$ by
$g_j\bU_{I_j}\mapsto  \lexp{g_j}\bP_{I_j}$. It  is easy  to check  that the
morphism $\pi$ thus defined commutes with an ``elementary morphism'' in the
bicategories  of varieties $\tilde\bX$ or  $\bX$ consisting of passing from
the    decomposition    $(\bw_1,\dots,\bw_i,\bw_{i+1},\dots,\bw_r)$    to
$(\bw_1,\dots,\bw_i\bw_{i+1},\dots,\bw_r)$     when    $(\bI_i\xrightarrow
{\bw_i\bw_{i+1}}\bI_{i+2})\in\cS$. Thus by \ref{bicat} the morphism $\pi$ is
well-defined  independently of the chosen decomposition of $\bb$. 

The fact that $\pi$   makes   $\tilde\bX(\bI,\bb\phi)$   a   
$\bL^{t(\bb\phi)}$-torsor over $\bX(\bI,\bb\phi)$ results then via
Proposition \ref{descente pour X tilde} from the same statement on the varieties
of \ref{tildeXwphi}.
\end{proof}

We give an isomorphism which reflects the transitivity of Lusztig's induction.
\begin{proposition}\label{produit fibre general}
Let  $\bI\xrightarrow{\bw\phi}\bI\in  B^+\phi(\cI)$, and let $w$ be the
image  of $\bw$ in  $W$; the automorphism  $w\phi$ of $W_I$ lifts to an automorphism
that  we will still  denote by $w\phi$ of  $B^+_\bI$. For $\bJ\subset\bI$, let
$\cJ$ be the set of $B^+_\bI$-conjugates of $\bJ$ and let
$\bJ\xrightarrow{\bv w\phi}\bJ\in  B^+_\bI w\phi(\cJ)$.  Then
\begin{enumerate}
\item We  have  an
isomorphism $\tilde\bX(\bI,\bw\phi)
\times_\LwF\tilde\bX_{\bL_I}(\bJ,\bv     w\phi)\xrightarrow\sim
\tilde\bX(\bJ,\bv\bw\phi)$ of $\bG^F$-varieties-$\bL_J^{t(\bv\bw\phi)}$,
where the variety $\tilde\bX_{\bL_I}(\bJ,\bv     w\phi)$ is defined via the
(obvious) Tits homomorphism 
$B^+_\bI\rtimes\genby{w\phi}\to N_{\bL_I}(\bT)\rtimes\genby
{t(\bw \phi)}$.
This isomorphism  is compatible  with the  action of
$F^n$   for  any  $n$   such  that  $\bI$,   $\bJ$,  $\bv$  and  $\bw$  are
$\phi^n$-stable.
\item   Through   the   quotient   by  $\bL_J^{t(\bv w\phi)}$  (see  Lemma
\ref{tildeX   torsor})   we   get   an   isomorphism  of  $\bG^F$-varieties
$$\tilde\bX(\bI,\bw\phi)   \times_\LwF\bX_{\bL_I}(\bJ,\bv
w\phi)\xrightarrow\sim \bX(\bJ,\bv\bw\phi).$$
\end{enumerate}
\end{proposition}
\begin{proof}
We  first look at the case  $\bw,\bv\in\bW$ (which implies $\bv\bw\in\bW$), in
which case we seek an isomorphism $$\tilde\bX(I,t(\bw\phi))\times_\LwF
\tilde\bX_{\bL_I}(J,t(\bv w\phi))
\xrightarrow\sim\tilde\bX(J,t(\bv\bw\phi))$$
where 
$$\begin{aligned}
\tilde\bX(I,t(\bw\phi))&=\{g\bU_I\in \bG/\bU_I\mid
g\inv\lexp Fg\in \bU_I t(\bw)\lexp F\bU_I\},\\
\tilde\bX(J,t(\bv\bw\phi))&=\{g\bU_J\in \bG/\bU_J\mid
g\inv\lexp Fg\in \bU_J t(\bv\bw)\lexp F\bU_J\}\\
\text{and }
\tilde\bX_{\bL_I}(J,t(\bv w\phi))&=
\{l\bV_J\in \bL_I/\bV_J\mid
l\inv\lexp{t(w\phi)}l\in \bV_I t(\bv)\lexp{t(w\phi)}\bV_J\},\\
\end{aligned}$$
where $\bV_J=\bL_I\cap  \bU_J$.

This is the content of Lusztig's proof of the transitivity of his induction
(see \cite[Lemma 3]{lusztig}), that we recall and detail in our context. We
claim   that   $(g\bU_I,l\bV_J)\mapsto   g\bU_I   l\bV_J=   gl\bU_J$
induces  the  isomorphism  we  want.  Using  that
$\bU_J=\bU_I\bV_J$   and   that   $\bV_Jt(\bv)\lexp{t(w\phi)}\bV_J$  is  in
$\bL_I$, thus normalizes $\bU_I$, we get
$$\bU_J    t(\bv\bw)\lexp    F\bU_J=   \bU_I\bV_Jt(\bv)\lexp{t(w\phi)}\bV_J
t(\bw)\lexp    F\bU_I=   \bV_Jt(\bv)\lexp{t(w\phi)}\bV_J\bU_I   t(\bw)\lexp
F\bU_I.$$     Hence    if    $(g\bU_I,l\bV_J)\in    \tilde\bX(I,t(\bw\phi))
\times\tilde\bX_{\bL_I}(J,t(\bv w\phi))$, we have
\begin{multline*}
(gl)\inv\lexp F(gl)\in l\inv\bU_It(\bw)\lexp F\bU_I\lexp F l=
l\inv\bU_I\lexp {t(w\phi)}lt(\bw)\lexp F\bU_I\hfill\\
\hfill=l\inv\lexp{t(w\phi)}l\bU_It(\bw)\lexp F\bU_I
\subset\bV_Jt(\bv)\lexp{t(w\phi)}\bV_J\bU_It(\bw)\lexp F\bU_I
=\bU_J t(\bv\bw)\lexp F\bU_J.
\end{multline*}
Hence we have defined a morphism
$\tilde\bX(I,t(\bw\phi)) \times\tilde\bX_{\bL_I}(J,t(\bv w\phi))
\to\tilde\bX(J,t(\bv\bw\phi))$ of
$\bG^F$-varieties-$\bL_J^{t(\bv w\phi)}$.
We  show  now  that  it  is  surjective. 
The unicity in the decomposition
$\bP_I\cap\lexp{ t(\bw\phi)}\bU_I=
\bL_I\cdot(\bU_I\cap\lexp{t(\bw\phi)}\bU_I)$
implies that the product $\bL_I.(\bU_I t(\bw)\lexp
F\bU_I)$ is direct. Hence  an element $x\inv\lexp
Fx\in\bU_J  t(\bv\bw)\lexp F\bU_J$  defines unique elements
$l\in\bV_Jt(\bv)\lexp{t(w\phi)}\bV_J$ and $u\in\bU_It(\bw)\lexp F\bU_I$ such 
that $x\inv\lexp
Fx=lu$.  If, using  Lang's theorem,  we write $l=l^{\prime-1}\lexp{t(w\phi)}l'$
with   $l'\in\bL_I$,  the  element   $g=xl^{\prime-1}$  satisfies  $g\inv\lexp
Fg=l'x\inv\lexp  Fx\lexp  Fl^{\prime-1}= \lexp{t(w\phi)}l'u\lexp Fl^{\prime-1}
\in\lexp{t(w\phi)}l'\bU_It(\bw)\lexp  F\bU_I\lexp  Fl^{\prime-1}=  \bU_It(\bw)
\lexp  F\bU_I$.  Hence  $(g\bU_I,l'\bV_J)$  is  a  preimage  of  $x\bU_J$  in
$\tilde\bX(I,t(\bw\phi)) \times\tilde\bX_{\bL_I}(J,t(\bv w\phi))$.

Let  us  look  now  at  the  fibers  of  the  above  morphism.  If $g'\bU_I
l'\bV_J=g\bU_I l\bV_J$ then $g^{\prime-1}g\in\bP_I$ so we may choose $g'$ in
$g'\bU_I$ such that $g'=g\lambda$ with $\lambda\in\bL_I$; we have then $\lambda
l'\bU_J=l\bU_J$, so that $l\inv\lambda l'\in \bU_J\cap \bL_I=\bV_J$; moreover
if  $g\lambda\bU_I\in\tilde\bX(I,t(\bw\phi))$  with
$\lambda\in\bL_I$,  then $\lambda\inv\bU_It(\bw)\lexp F\bU_I\lexp  F\lambda
=\bU_It(\bw)\lexp F\bU_I$ which implies $\lambda\in\LwF$. Conversely,
the action of $\lambda\in\LwF$ given by $(g\bU_I,l\bV_J)\mapsto
(g\lambda\bU_I,\lambda\inv l\bV_J)$ preserves the subvariety
$\tilde\bX(I,t(\bw\phi)) \times\tilde\bX_{\bL_I}(J,t(\bv w\phi))$,
of $\bG/\bU_I\times\bL_I/\bV_J$. Hence the
fibers are the orbits under this action of $\LwF$.

Now the morphism $j:(g\bU_I,l\bV_J)\mapsto gl\bU_J$ is an isomorphism 
$\bG/\bU_I\times_{\bL_I}\bL_I/\bV_J\simeq\bG/\bU_J$ since $g\bU_J\mapsto(g\bU_I,\bV_J)$
is its inverse. By what we have seen above the restriction of $j$ to the closed subvariety
$\tilde\bX(I,t(\bw\phi)) \times_\LwF\tilde\bX_{\bL_I}(J,t(\bv w\phi))$ maps this variety
surjectively on the closed subvariety $\tilde\bX(J,t(\bv\bw\phi))$ of $\bG/\bU_J$,
hence we get the isomorphism we want.

We now consider the case of generalized varieties. Let $k$ be the number of
terms of the strict normal decomposition of $\bv\bw$ and let
$\bI\xrightarrow{\bw_1}\bI_2
\xrightarrow{\bw_2}\bI_3\to\dots\to\bI_k\xrightarrow{\bw_k}\lexp\phi\bI$
be  a normal decomposition of $\bI\xrightarrow\bw\lexp\phi\bI$ of same
length. We have $\tilde\bX(\bI,\bw\phi)\simeq
\tilde\bX(I_1\times  I_2\times \dots\times  I_k,(t(\bw_1),\dots,t(\bw_k))
F_1)$, where $F_1$  is  as  in  Proposition \ref{descente}.  Let  us  write
$(\bv_1\bw_1,\dots,\bv_k\bw_k)$  for the normal  decomposition of $\bv\bw$, 
with same  notation as in Proposition \ref{normal  form of vw}. Let
$J_1=J$ and $J_{j+1}= J_j^{v_jw_j}\subset I_{j+1}$ for
$j=1,\dots,k-1$. We apply the first part of the proof to the group $\bG^k$
with  isogeny $F_1$ with  $I$, $J$, $w$, $v$ replaced
respectively  by  $I_1\times\dots\times  I_k$,  $J_1\times\dots\times  J_k$,
$(w_1,\dots,w_k)$    $(v_1,\dots,v_k)$.    Using    the   isomorphisms
from Proposition \ref{descente pour X tilde}; 
$$\tilde\bX_{\bG^k}(J_1\times\cdots\times J_k,(t(\bv_1\bw_1),\dots,
t(\bv_k\bw_k))F_1)\simeq  \tilde\bX(\bJ,\bv\bw\phi)$$ and 
$$\tilde\bX_{\bL_{I_1\times\dots\times I_k}} (J_1\times\dots\times
J_k,(v_1,\dots,v_k).(t(\bw_1),\dots,t(\bw_k))F_1)\simeq   
\tilde\bX_{\bL_I}(\bJ,\bv w\phi),$$ we get (i).
Now (ii) is immediate from (i) taking the quotient on both sides
by $\bL_J^{t(\bv w\phi)}$.
\end{proof}
In  the particular case  where $\bI=\emptyset$ we  write $\bX(\bw\phi)$ for
$\bX(\bI,\bw\phi)$.  Let us recall  that in \cite[2.3.2]{DMR}  we defined a
monoid  $\uB^+$ generated  by $B^+$  and symbols  $\uw$ where $w\in W$, and
attached  to  any  $\bu\in\uB^+$  a  Deligne-Lusztig variety $\bX(\bu\phi)$.
This variety is denoted  by $\bX(\bu)$ in  \cite{DMR} and roughly  defined by the property
that   given   $\uw$   attached   to   $w\in   W$,  we  have  $\bX(\bu_1\uw
\bu_2\phi)=\bigcup_{w'\le  w}\bX(\bu_1\bw' \bu_2\phi)$, where $\bw'$ is the
lift  to $B^+$ of $w'$  and where $w'$ runs  over the elements smaller than
$w$  for the Bruhat order. Attached to  $I\subset S$, we have an analogous
monoid  $\uB^+_\bI$  attached  to  $W_I$,  which  has  a  natural embedding
$\uB^+_\bI\subset\uB^+$.
\begin{corollary}\label{X(uwI)}  With these  notations of  \cite{DMR}, for
any $\bI\xrightarrow{\bw\phi}\bI\in B^+\phi(\cI)$ and
any  $\bu\in\uB^+_\bI$, we   have  an  isomorphism
$\bX(\bu\bw\phi)\xrightarrow\sim\tilde\bX(\bI,\bw\phi)
\times_\LwF\bX_{\bL_I}(\bu   w\phi)$   and   a   surjective
morphism  $\bX(\bu\bw\phi)\rightarrow  \bX(\bI,\bw\phi)$  whose  fibers are
isomorphic to $\bX_{\bL_I}(\bu w\phi)$.
\end{corollary}
\begin{proof}
The  variety $\bX(\bu\bw\phi)$ is the union of varieties of the form
$\bX_{\bL_I}(\bv w\phi)$
with  $\bv\in\bW_\bI$.  The  isomorphisms given for each $\bv$ by Proposition
\ref{produit  fibre  general}  applied with  $\bJ=\emptyset$
can  be  glued  together  to give a global morphism of varieties
since  they  are  defined by a formula
independent of $\bv$. We thus get a bijective morphism
$\tilde\bX(\bI,\bw\phi)  \times_\LwF\bX_{\bL_I}(\bu   w\phi)
\to\bX(\bu\bw\phi)$
which   is  an   isomorphism  since   $\bX(\bu\bw\phi)$  is  normal  (see
\cite[2.3.5]{DMR}).  Composing  this  isomorphism  with  the  projection of
$\tilde\bX(\bI,\bw\phi) \times_\LwF\bX_{\bL_I}(\bu   w\phi)$ onto
$\bX(\bI,\bw\phi)$ (see \ref{tildeX torsor}), we get the second assertion of the corollary.
\end{proof}

\subsection*{Endomorphisms of parabolic Deligne-Lusztig varieties ---
the category $\DI$}
\begin{definition}\label{Dv}
Given  $\ad\bv\in\DI(\bI\xrightarrow{\bw\phi}\bI,
\bJ\xrightarrow{\bv\inv\bw\phi\bv}\bJ)$ where $\bJ=\bI^\bv$, we define
morphisms   of   varieties:
\begin{enumerate}
\item
$D_\bv:\bX(\bI,\bw\phi)\to\bX(\bJ,\bv\inv\bw\phi\bv)$ 
as the restriction of the morphism 
\begin{multline*}(a,b)\mapsto (b,\lexp F a):
\cO(\bI,\bw)=\cO(\bI,\bv)\times_{\cP_J}\cO(\bJ,\bv\inv\bw)\to\hfill\\ \hfill
\cO(\bJ,\bv\inv\bw)\times_{\cP_{\lexp\phi I}}\cO(\lexp\phi\bI,\lexp\phi\bv)
=\cO(\bJ,\bv\inv\bw\lexp\phi\bv). 
\end{multline*}
\item $\tilde D_\bv:\tilde\bX(\bI,\bw\phi)\to\tilde\bX(\bJ,\bv\inv\bw\phi\bv)$ 
as the restriction of the morphism 
\begin{multline*}(a,b)\mapsto (b,\lexp F a):
\tilde\cO(\bI,\bw)=\tilde\cO(\bI,\bv)\times_{\bG/\bU_J}
\tilde\cO(\bJ,\bv\inv\bw)\to\hfill\\ \hfill\tilde\cO(\bJ,\bv\inv\bw)
\times_{\bG/\bU_{\lexp\phi I}}\tilde\cO(\lexp\phi\bI,
\lexp\phi\bv)=\tilde\cO(\bJ,\bv\inv\bw\lexp\phi\bv).
\end{multline*}
\end{enumerate}
\end{definition}
Note that the existence of well-defined decompositions as above of 
$\cO(\bI,\bw)$ and of $\tilde\cO(\bI,\bw)$ are consequences of 
Theorem \ref{bicategory}.

Note that
when  $\bv$, $\bw$ and $\bv\inv\bw\lexp\phi\bv$ 
are in $\bW$ the endomorphism $D_\bv$ maps $g\bP_I\in
\bX(I,w\phi)$  to $g'\bP_J\in\bX(J,v\inv  w\phi v)$ such  that $g\inv  g'\in\bP_I
v\bP_J$ and $g^{\prime-1}\lexp F g\in\bP_Jv\inv w\lexp F\bP_I$ and similarly
for $\tilde D_\bv$.

Note also that $D_\bv$ and $\tilde D_\bv$ are equivalences of \'etale sites;
indeed, the proof of \cite[3.1.6]{DMR} applies without change in our case. 

The definition of $\tilde D_\bv$ and $D_\bv$ shows the following property:
\begin{lemma}\label{D tilde to D}
The following diagram is commutative:
$$
\xymatrix{\tilde\bX(\bI,\bw\phi)\ar[r]^-{\tilde D_\bv}\ar[d]&
\tilde\bX(\bJ,\bv\inv\bw\phi\bv)\ar[d]\\
\bX(\bI,\bw\phi)\ar[r]^-{D_\bv}&\bX(\bJ,\bv\inv\bw\phi\bv)}
$$
where the vertical arrows are the respective quotients by $\LwF$
and  $\bL_J^{t(\bv\inv\bw\phi\bv)}$ (see  Lemma \ref{tildeX torsor});
for  $l\in\LwF$  we  have  $\tilde  D_\bv\circ l=l^{t(\bv)}\circ
\tilde D_\bv$.
\end{lemma}
As  a further consequence  of Theorem \ref{bicategory},
the map  which sends  a simple morphism $\ad\bv$ to
$D_\bv$ extends to a natural morphism from
$\DI(\bI\xrightarrow{\bw\phi}\bI,\bJ\xrightarrow{\bv\inv\bw\phi\bv}\bJ)$
to $\Hom_{\bG^F}(\bX(\bI,\bw\phi),\bX(\bJ,\bv\inv\bw\phi\bv))$
whose image consists of 
equivalences of \'etale  sites.  We  still  denote  by  $D_\bv$  the image 
of $\ad\bv$ by this morphism.

\begin{lemma}\label{descente pour Dx}
Via the isomorphism of \ref{descente pour X tilde} and with
the notations of loc.\ cit.\ the morphism $D_{\bw_1}$ with source
$\tilde\bX_\bG(\bI,\bb\phi)$ becomes the morphism $D_{(t(\bw_1),1,\dots,1)}$
with source 
$\tilde\bX_{\bG^k}(I_1\times\dots\times I_k,(t(\bw_1),\dots,t(\bw_k))F_1)$.
\end{lemma}
\begin{proof}
The  endomorphism $D_{\bw_1}$ maps  the element $(g_1\bU_1,\dots,g_k\bU_k)$
of the model of \ref{descente pour X tilde} of $\tilde\bX_\bG(\bI,\bb\phi)$
to  $(g_2\bU_2,\dots,g_k\bU_k,\lexp Fg_1\lexp  F\bU_1)$. On  the other hand
the   isomorphism  of   Proposition  \ref{descente   pour  X   tilde}  maps
$(g_1\bU_1,\dots,g_k\bU_k)$    to   $$(g_1,\dots,g_k)(\bU_1,\dots,\bU_k)\in
\tilde  X_{\bG^k}(I_1\times\dots\times I_k, (t(\bw_1),\dots,t(\bw_k))F_1)$$
which  is  sent  by  $D_{(t(\bw_1),1,\dots,1)}$  to  $(g_2,\dots,g_k,\lexp
Fg_1)(\bU_2,\dots,\bU_k,\lexp      F\bU_1)$ which is the image by the
isomorphism of
Proposition  \ref{descente pour X tilde} of $(g_2\bU_2,\dots,g_k\bU_k,\lexp
Fg_1\lexp F\bU_1)$, whence the lemma.
\end{proof}
\begin{proposition}\label{Dx pour x in B_I}
For  $\bJ\subset\bI$ let  $\cJ$ denote  the set  of $B^+_\bI$-conjugates of
$\bJ$.  With same assumptions  and notation as  in Proposition \ref{produit
fibre general},
let $\bJ\xrightarrow{\bx}\bJ^\bx\in  B^+_\bI(\cJ)$ be a left-divisor of
$\bJ\xrightarrow{\bv}\lexp{\bw\phi}\bJ$.
The following diagram
is commutative:
$$\xymatrix{
\tilde\bX(\bI,\bw\phi)\times_\LwF
\tilde\bX_{\bL_I}(\bJ,\bv\cdot w\phi)\ar[r]^-\sim\ar[d]_{\Id\times\tilde D_\bx}&
\tilde\bX(\bJ,\bv\bw\phi)\ar[d]^{\tilde D_{\bx}} \\
\tilde\bX(\bI,\bw\phi) \times_\LwF
\tilde\bX_{\bL_I}(\bJ^\bx,\bx\inv(\bv\cdot w\phi)\bx)\ar[r]^-\sim &
\tilde\bX(\bJ^\bx,\bx\inv\bv\bw\phi \bx) 
}$$
\end{proposition}
\begin{proof} Decomposing $\bx$ into a product of simples in the category
analogous to $\DI$ where $B^+$ is replaced by $B^+_\bI$ and $\cI$ by $\cJ$,
the  definitions  show  that  it  is  sufficient  to  prove  the result for
$\bx\in\bW$.  We use  then Proposition  \ref{normal form  of vw}  and Lemma
\ref{descente  pour Dx} to reduce the proof to the case where $\bv$, $\bw$
and  $\bx\inv\bv\lexp{w\phi}\bx$  are  in  $\bW$  (in  which  case $\bv\bw$ and
$\bx\inv\bv\bw\lexp{\phi}\bx$   are   in   $\bW$   too):  we  choose  compatible
decompositions  of $\bv$ and  $\bw$ as in  \ref{normal form of  vw} which we
refine if needed so that $\bx$ is the first term of that of $\bv$ and use
Lemma \ref{descente pour Dx} once in $\bG$ and once in in $\bL_I$.

Assume now $\bv$, $\bw$ and $\bx\inv\bv\lexp{w\phi}\bx$ in $\bW$. We start with
$(g\bU_I,l\bV_J)\in\tilde\bX(I,t(\bw\phi))\times\tilde\bX_{\bL_I}(J,v
w\phi)$.  This element is sent  by the top isomorphism  of the diagram to
$gl\bU_J$.  On the other hand, we have seen above Lemma \ref{D  tilde to D} 
that it is sent by
$\Id\times \tilde D_\bx$  to  $(g\bU_I,l'\bV_{J^x})$   where  $l\inv  l'\in\bV_J
x\bV_{J^x}$       and      $l^{\prime-1}\lexp{t(\bw\phi)}l\in\bV_{J^x}x\inv
v\lexp{wF}\bV_J$.  This element is  sent in turn to $gl'\bU_{J^x}$  by the bottom
isomorphism of the diagram. We have to check that
$gl'\bU_{J^x}=\tilde D_\bx(gl\bU_J)$.   But   $(gl)\inv   gl'=l\inv   l'$   is  in
$\bV_Jx\bV_{J^x}\subset \bU_J x\bU_{J^x}$ and
\begin{multline*}
(gl')\inv\lexp F(gl)=l^{\prime-1}g\inv\lexp Fg\lexp Fl\in
l^{\prime-1}\bU_It(\bw)\lexp   F\bU_I\lexp Fl=\bU_Il^{\prime-1}\lexp{t(w\phi)}
lt(\bw)\lexp  F\bU_I\hfill\\  \hfill  \subset\bU_I\bV_{J^x}x\inv  vw\lexp
F\bV_J\lexp F\bU_I=\bU_{J^x}x\inv vw\lexp F\bU_J,
\end{multline*}
so that $(gl'\bU_{J^x})=\tilde D_\bx(gl\bU_J)$.
\end{proof}

Using Proposition \ref{produit fibre general}(ii), Proposition \ref{Dx pour x in B_I} and Lemma \ref{D tilde to D} we get
\begin{corollary}
The following diagram
is commutative:
$$\xymatrix{
\tilde\bX(\bI,\bw\phi)\times_\LwF
\bX_{\bL_I}(\bJ,\bv\cdot w\phi)\ar[r]^-\sim \ar[d]_{\Id\times D_\bx} &
 \bX(\bJ,\bv\bw\phi)\ar[d]^{D_{\bx}} \\
\tilde\bX(\bI,\bw\phi) \times_\LwF
\bX_{\bL_I}(\bJ^\bx,\bx\inv(\bv\cdot w\phi)\bx)\ar[r]^-\sim &
 \bX(\bJ^\bx,\bx\inv\bv\bw\phi \bx) 
}$$
\end{corollary}

\subsection*{Affineness}
Until  the end of  the text, we  will be specially interested in
varieties $\bX(\bI,\bb\phi)$
which satisfy the assumption of Theorem \ref{desc endo}, that is some power
of  $\bb\phi$ is left-divisible by  $\bw_\bI\inv\bw_0$. They have many nice
properties.  We show in  this subsection that  they are affine, by adapting
the proof of Bonnaf\'e and Rouquier \cite{BR2}; we use the existence of the
varieties  $\tilde\cO(\bI,\bb)$  and  $\tilde\bX(\bI,\bb\phi)$  to  replace
doing a quotient by $\bL_I$ by doing a quotient by $\LwF$.

\begin{proposition}   Assume   the   morphism   $\bI\xrightarrow{\bb}\bJ\in
B^+(\cI)$ is left-divisible by
$\Delta_\cI=\bI\xrightarrow{\bw_\bI\inv\bw_0}\bI^{\bw_0}$. Then the variety
$\tilde\cO(\bI,\bb)$ is affine.
\end{proposition}
\begin{proof}
By assumption there exists a decomposition into elements of $\cS$ of
$\bI\xrightarrow{\bb}\bJ$ of the form
$\bI\xrightarrow{\bw_\bI\inv\bw_0}\bI_1
\xrightarrow{\bv_1}\bI_2
\xrightarrow{\bv_2}\bI_3\to\dots\to\bI_r\xrightarrow{\bv_r}\bJ$.
We show that the map $\varphi$ defined by:
\begin{multline*}
\bG\times\prod_{i=1}^{i=r}(\bU_{I_i}\cap\lexp{t(\bv_i)}\bU^-_{I_{i+1}})t(\bv_i)
\rightarrow\hfill\\
\hfill\tilde\cO(I,t(\bw_\bI\inv\bw_0))\times_{\bG/\bU_{I_1}}\tilde\cO(I_1,
t(\bv_1))\dots\times_{\bG/\bU_{I_r}}\tilde\cO(I_r,t(\bv_r))\\
(g,h_1,\dots,h_r)\mapsto\hfill\\
\hfill(g\bU_I,gt(\bw_\bI\inv\bw_0)\bU_{I_1},
gt(\bw_I\inv\bw_0)h_1\bU_{I_2},\dots,gt(\bw_I\inv\bw_0)h_1\dotsm h_r\bU_J)
\end{multline*}
is an isomorphism; since the first variety is a product of affine varieties
this will prove our claim.

Since
$\bU_{I_i}t(\bv_i)\bU_{I_{i+1}}$ is isomorphic to
$(\bU_{I_i}\cap\lexp{t(\bv_i)}\bU^-_{I_{i+1}})t(\bv_i)\times \bU_{I_{i+1}}$,
by composition with the first projection we get a morphism $\eta_i:
\bU_{I_i}t(\bv_i)\bU_{I_{i+1}}\rightarrow
(\bU_{I_i}\cap\lexp{t(\bv_i)}\bU^-_{I_{i+1}})t(\bv_i)$
for $i=1,\dots,r$, where $I_{r+1}=J$. Similarly we have a morphism $\eta: 
\bU_It(\bw_\bI\inv\bw_0)\bU_{I_1}\rightarrow
(\bU_I\cap\lexp{t(\bw_\bI\inv\bw_0)}\bU^-_{I_{1}})t(\bw_\bI\inv\bw_0)$. For
\begin{multline*}
x=(g\bU_I,g_1\bU_{I_1},g_2\bU_{I_2},\dots,g_r\bU_{I_r},g_{r+1}\bU_J)\hfill\\
\hfill\in
\tilde\cO(I,t(\bw_\bI\inv\bw_0))\times_{\bG/\bU_{I_1}}\tilde\cO(I_1,t(\bv_1))
\dots\times_{\bG/\bU_{I_r}}\tilde\cO(I_r,t(\bv_r))
\end{multline*}
let $\psi(x)=g\eta(g\inv g_1)$, $\psi_1(x)=\psi(x)t(\bw_0)$,
$\psi_i(x)=\eta_i((\psi(x)\psi_1(x)\dotsm\psi_{i-1}(x))\inv g_i)$.
We claim that the map
$\psi$ (resp.\ $\psi_i$) is well defined, that is does not depend on the
representative $g$ (resp.\ $g_i$) chosen; the morphism
$x\mapsto(\psi(x),\psi_1(x),\dots,\psi_r(x))$ is
then clearly inverse to $\varphi$. Since $\eta_i(hu)=\eta_i(h)$ for all
$h\in\bU_{I_i}t(\bv_i)\bU_{I_{i+1}}$ and all $u\in\bU_{I_{i+1}}$,
we get that all $\psi_i$ are well-defined. Since moreover
$\eta(uh)=u\eta(h)$ for 
all $h\in\bU_It(\bw_\bI\inv\bw_0)\bU_{I_1}$ and all $u\in\bU_I$,
we get that $\psi$ also is
well-defined, whence our claim.
\end{proof}

\begin{proposition} \label{tildeX affine}
Assume that we are under the assumptions of Theorem 
\ref{desc endo}, that is $(\bI\xrightarrow{\bw\phi}\bI)\in B^+\phi(\cI)$ has
some  power  divisible  by  $\Delta_\cI$,  or  equivalently  some  power of
$\bw\phi$  is left-divisible by  $\bw_\bI\inv\bw_0$. 
Then $\tilde\bX(\bI,\bw\phi)$ is affine.
\end{proposition}
\begin{proof}
Let  us  define $k$  as the smallest
integer such that $\lexp{\phi^k}\bI=\bI$,   $\lexp{\phi^k}\bw=\bw$    and
$\bw_\bI\inv\bw_0\preccurlyeq\bw^{(k)}$,     where     $\bw^{(k)}:=\bw\lexp
\phi\bw\dotsm\lexp{\phi^{k-1}}\bw$.

We will embed $\tilde\bX(\bI,\bw\phi)$ as a closed
subvariety in $\tilde\cO(\bI,\bw^{(k)})$, which will prove it to be affine.

Let $\bI\xrightarrow{\bw_1}\bI_2
\xrightarrow{\bw_2}\bI_3\to\dots\to\bI_r\xrightarrow{\bw_r}\lexp\phi\bI$ 
be a decomposition of $\bI\xrightarrow{\bw}\lexp\phi\bI$
into elements of $\cS$, so that
$\tilde\cO(\bI,\bw^{(k)})$ identifies with the set of sequences
$$\begin{aligned}
(&g_{1,1}\bU_I,g_{1,2}\bU_{I_2},\dots,g_{1,r}\bU_{I_r},\\
&g_{2,1}\bU_{\lexp \phi I},g_{2,2}\bU_{\lexp \phi I_2},\dots,g_{2,r}\bU_{\lexp
\phi I_r},\\
&\dots,\\
&g_{k,1}\bU_{\lexp{\phi^{k-1}}I},g_{k,2}\bU_{\lexp
  {\phi^{k-1}}I_2},\dots,g_{k,r}\bU_{\lexp{\phi^{k-1}}I_r},\\&g_{k+1,1}\bU_I)
\end{aligned}$$
such that for $j<r$ we have $g_{i,j}\inv g_{i,j+1}\in \bU_{\lexp{\phi^{i-1}}I_j}
t(\lexp{\phi^{i-1}}{\bw}_j)\bU_{\lexp{\phi^{i-1}}I_{j+1}}$ 
and $g_{i,r}\inv g_{i+1,1}\in \bU_{\lexp{\phi^{i-1}}I_r}
t(\lexp{\phi^{i-1}}{\bw}_r) \bU_{\lexp{\phi^i}I}$.

Similarly $\tilde\bX(\bI,\bw\phi)$ identifies with the set of sequences
$(g_1\bU_I,g_2\bU_{I_2},\dots,g_r\bU_{I_r})$ such that
$g_j\inv g_{j+1}\in \bU_{I_j}t(\bw_j)\bU_{I_{j+1}}$ for $j<r$ and
$g_r\inv \lexp Fg_1\in \bU_{I_r}t(\bw_r) \bU_{\lexp \phi I}$. It is thus clear
that the map
$$\begin{aligned}
(g_1\bU_I,g_2\bU_{I_2},\dots,g_r\bU_{I_r})\mapsto
(&g_1\bU_I,g_2\bU_{I_2},\dots,g_r\bU_{I_r},\\
&\lexp Fg_1\bU_{\lexp \phi I},\lexp Fg_2\bU_{\lexp \phi I_2},\dots,\lexp
Fg_r\bU_{\lexp \phi I_r},\\
&\dots,\\
&\lexp{F^{k-1}}g_1\bU_{\lexp{\phi^{k-1}}I},
\dots,\lexp{F^{k-1}}g_r\bU_{\lexp{\phi^{k-1}}I_r},
\lexp{F^k}g_1\bU_I)\\
\end{aligned}$$
identifies $\tilde\bX(\bI,\bw\phi)$ with the closed subvariety of
$\tilde\cO(\bI,\bw^{(k)})$ defined by $g_{i+1,j}\bU_{\lexp{\phi^i}I_j}=
\lexp F(g_{i,j}\bU_{\lexp{\phi^{i-1}}I_j})$ for all $i,j$.
\end{proof}
\begin{corollary} Under the assumptions of Theorem 
\ref{desc endo}, that is $(\bI\xrightarrow{\bw\phi}\bI)\in B^+(\cI)$ has some power 
divisible by $\Delta_\cI$, or equivalently some power of $\bw\phi$ is divisible
on the left by $\bw_\bI\inv\bw_0$, the variety 
$\bX(\bI,\bw\phi)$ is affine.
\end{corollary}
\begin{proof} Indeed, by Proposition \ref{tildeX affine} and Lemma \ref{tildeX torsor},
$\bX(\bI,\bw\phi)$ is
the quotient of an affine variety by a finite group, so it is affine.
\end{proof}
\subsection*{Shintani descent identity}
In this subsection we give a formula for the Leftschetz number of a variety
$\bX(\bI,\bw F)$ which we deduce from a ``Shintani descent identity''.

Let $m$ be a multiple of $\delta$;
if we identify $\bG/\bB$ with the variety
$\cB$ of Borel subgroups of $\bG$, 
the $\bG^{F^m}$-module $\Qlbar(\bG/\bB)^{F^m}$ identifies with
the permutation module of $\bG^{F^m}$ on $\cB^{F^m}$.
Its endomorphism algebra $\cH_{q^m}(W):=\End_{\bG^{F^m}}( \Qlbar\cB^{F^m})$
has a basis consisting of the operators $(T_w)_{w\in W}$ where
$$T_w: \bB'\mapsto\sum_{\{\bB''\in\cB^{F^m}\mid\bB''\xrightarrow w\bB'\}}\bB''$$
(see \cite[Chapitre IV \S 2, exercice 22]{bbk}).

Similarly, since $I$ is $F^m$-stable, the algebra
$\cH_{q^m}(W,W_I):=\End_{\bG^{F^m}}(\Qlbar\cP_I^{F^m})$ has a
$\Qlbar$-basis consisting of the operators
$$X_w: \bP\mapsto\sum_{\{\bP'\in\cP_I^{F^m}\mid\bP'\xrightarrow{I,w,I}\bP\}}\bP',$$
where  $w$  runs  over  a  set  of  representatives  of  the  double cosets
$W_I\backslash  W/W_I\simeq  \bP_I^{F^m}\backslash  \bG^{F^m}/\bP_I^{F^m}$.
The  map  $\gamma$  which  sends  $\bP\in\cP_I^{F^m}$  to  the  sum of all its
$F^m$-stable  Borel  subgroups  makes  $\Qlbar\cP_I^{F^m}$  into a direct
summand  of $\Qlbar\cB^{F^m}$.  Indeed the  image of  $\gamma$ identifies with
that  of the idempotent  $X_1= |(\bP_I/\bB)^{F^m}|\inv \sum_{v\in W_I}
T_v$,   and  $\gamma$ has  a left-inverse  given  up to a scalar by  mapping
$\bB\in\cB^{F^m}$  to the unique (thus  $F^m$-stable) parabolic subgroup in
$\cP_I$  containing it. The operator  $X_w$ identifies with the restriction
of $X_1T_w$ to the image $\Qlbar\cP_I^{F^m}$ of $X_1$.

We may define a $\Qlbar$-representation of
$B^+(\cI)(\bI)$ on $\Qlbar\cP_I^{F^m}$
by sending $\bI\xrightarrow{\bw}\bI$ to the operator $X_\bw\in \cH(W,W_I)$ defined by
$$X_\bw(\bP)=\sum_{\{x\in\cO(\bI,\bw)^{F^m}\mid p''(x)=\bP\}}p'(x).$$
When $\bw\in\bW$, with image $w$ in $W$, the operators $X_\bw$ and $X_w$ coincide.
In the particular case where $I=\emptyset$ we get an operator denoted by $T_\bw$,
defined for any $\bw$ in $B^+$.
The operator $X_\bw$ identifies with the restriction of $X_1T_\bw$
to the image $\Qlbar\cP_I^{F^m}$ of $X_1$.

Similarly, to $\bI\xrightarrow{\bw\phi}\bI\in B^+\phi(\cI)$
we associate an endomorphism $X_{\bw\phi}$ of
$\Qlbar\cP_I^{F^m}$ by the formula
$$X_{\bw\phi}(\bP)=\sum_{\{x\in\cO(\bI,\bw)^{F^m}\mid p''(x)=F(\bP)\}}p'(x).$$
When $\phi(I)=I$ we have $X_{\bw\phi}=X_\bw F$. In general we have
$X_{\bw\phi}=X_1T_\bw F$ on $\Qlbar\cP_I^{F^m}$ seen as a subspace of
$\Qlbar\cB^{F^m}$: on this latter module one can separate
the action of $F$; the operator $F$ sends the submodule
$\Qlbar\cP_I^{F^m}$ to $\Qlbar\cP_{\phi(I)}^{F^m}$
which is sent back to $\Qlbar\cP_I^{F^m}$ by $X_1T_\bw$.
The endomorphism  $X_{\bw\phi}$ commutes with $\bG^{F^m}$ like $F$, hence
normalizes $\cH_{q^m}(W,W_I)$; its action identifies with the conjugation
action of $T_\bw\phi$ on $\cH_{q^m}(W,W_I)$ inside 
$\cH_{q^m}(W)\rtimes\genby\phi$ .

Recall that the Shintani descent 
$\Sh_{F^m/F}$ is the ``norm'' map which maps the
$F$-class of $g'=h.\lexp Fh\inv\in\bG^{F^m}$ to the class of
$g=h\inv .\lexp{F^m} h \in\bG^F$.
\begin{proposition}[Shintani descent identity]
Let $\bI\xrightarrow{\bw\phi}\bI\in B^+\phi(\cI)$,
and let $m$ be a multiple of $\delta$. We have the following equality of
functions on $\bG^F$:
$$(g\mapsto|\bX(\bI,\bw\phi)^{gF^m}|)=\Sh_{F^m/F}
(g'\mapsto \Trace(g' X_{\bw\phi}\mid \Qlbar\cP_I^{F^m})).$$
\end{proposition}
\begin{proof}
Let $g=h\inv .\lexp{F^m} h$ and $g'=h.\lexp Fh\inv$, so that the class of $g$ is
the image by $\Sh_{F^m/F}$ of the $F$-class of $g'$; we have
$\bX(\bI,\bw\phi)^{gF^m}=\{x\in\cO(\bI,\bw)\mid \lexp {F^mh}x=\lexp hx
\text{ and } p''(\lexp hx)=\lexp {g'F}p'(\lexp hx)\}$.
Taking $\lexp hx$ as a variable in the last formula we get
$|\bX(\bI,\bw\phi)^{gF^m}|=|\{x\in\cO(\bI,\bw)^{F^m}\mid p''(x)=\lexp{g'F}p'(x)\}|$.
Putting $\bP=p'(x)$ this last number becomes
$\sum_{\bP\in\cP_I^{F^m}}
|\{x\in\cO(\bI,\bw)^{F^m}\mid  p'(x)=\bP\text{ and }p''(x)=\lexp{g'F}\bP\}|$. On
the  other hand the trace of $g'X_{\bw\phi}$ is the sum over $\bP\in\cP_I^{F^m}$
of   the   coefficient   of   $\bP$   in   
$\sum_{\{x\in\cO(\bI,\bw)^{F^m}\mid p''(x)=F(\bP)\}}g'p'(x)$. 
This coefficient is equal to
$|\{x\in\cO(\bI,\bw)^{F^m}\mid  g'p'(x)=\bP\text{  and  }p''(x)=\lexp F\bP\}|=
|\{x\in\cO(\bI,\bw)^{F^m}\mid  p'(x)=\bP\text{  and }p''(x)=\lexp{g'F}\bP\}|$,
this last equality by changing $g'x$ into $x$.
\end{proof}
The above computation can be done along
different  lines, without mentioning $\Qlbar\cP_I^{F^m}$;
one can use instead Corollary \ref{X(uwI)} for $\bu=\uw_I$, which gives
a $\bG^F$-equivariant morphism $\bX(\uw_I\bw\phi)\to\bX(\bI,\bw\phi)$ whose
fibers  are isomorphic  to the  variety of  Borel subgroups of $\bL_I$; the
action  of $F$ induces that  of $t(w\phi)$ on the  fibers. One may then use
directly  \cite[3.3.7]{DMR}  to  get $|\bX(\uw_I\bw\phi)^{gF^m}|= \Trace(g'
T_{\uw_I}T_\bw\phi\mid   \Qlbar\cB^{F^m})$,   where   $T_{\uw_I}=\sum_{v\in
W_I}T_v$.

By, for example, \cite[II, 3.1]{DM} 
the algebras $\cH_{q^m}(W)$ and $\cH_{q^m}(W)\rtimes\genby\phi$
split over $\Qlbar[q^{m/2}]$;
corresponding to the specialization $q^{m/2}\mapsto 1:\cH_{q^m}(W)\to
\Qlbar W$, there is a bijection $\chi\mapsto\chi_{q^m}:
\Irr(W)\to\Irr(\cH_{q^m}(W))$.
Choosing an extension $\tilde\chi$ to $W\rtimes\genby\phi$ of each
character in $\Irr(W)^\phi$, we get a corresponding extension 
$\tilde\chi_{q^m}\in \Irr(\cH_{q^m}(W)\rtimes\genby\phi)$ which takes its
values in $\Qlbar[q^{m/2}]$. If
$U_\chi\in\Irr(\bG^{F^m})$ is the corresponding character of $\bG^{F^m}$,
we get a corresponding extension $U_{\tilde\chi}$ of $U_\chi$ to
$\bG^{F^m}\rtimes\genby F$ (see \cite[III th\'eor\`eme 1.3 ]{DM}).
With these notations, the Shintani descent identity becomes

\begin{proposition}\label{shintani1}
$$(g\mapsto|\bX(\bI,\bw\phi)^{gF^m}|)=
\sum_{\chi\in\Irr(W)^\phi}
\tilde\chi_{q^m}(X_1 T_\bw\phi)\Sh_{F^m/F}U_{\tilde\chi}$$ and
the only characters $\chi$ in that sum which give a non-zero contribution are those which
are a component of $\Ind_{W_I}^W\Id$.
\end{proposition}
\begin{proof}
We have $\Trace(g' X_{\bw\phi}\mid \Qlbar\cP_I^{F^m})=
\Trace(g' X_1T_\bw\phi\mid \Qlbar\cB^{F^m})$
since $X_1$ is the projector onto $\Qlbar\cP_I^{F^m}$.
Hence
$(g\mapsto|\bX(\bI,\bw\phi)^{gF^m}|)=\sum_{\chi\in\Irr(W)^\phi}
\tilde\chi_{q^m}(X_1 T_\bw\phi)\Sh_{F^m/F}U_{\tilde\chi}$.
Since $X_1$ acts by 0 on the representation of character $\chi$ if $\chi$ is
not a component of $\Ind_{W_I}^W\Id$, we get the second assertion.
\end{proof}

Finally, if $\lambda_\rho$ is the root of unity attached to
$\rho\in\cE(\bG^F,1)$ as in \cite[3.3.4]{DMR}, the above formula translates,
using \cite[III, 2.3(ii)]{DM} as
\begin{corollary}\label{shintani2}
$$|\bX(\bI,\bw\phi)^{gF^m}|=
\sum_{\rho\in\cE(\bG^F,1)}\lambda_\rho^{m/\delta}\rho(g)
\sum_{\chi\in\Irr(W)^\phi}
\tilde\chi_{q^m}(X_1 T_\bw\phi)\langle \rho,R_{\tilde\chi}\rangle_{\bG^F},$$
where $R_{\tilde\chi}=|W|\inv\sum_{w\in
W}\tilde\chi(w\phi)R^\bG_{\bT_w}(\Id)$.
The only characters $\chi$ in the above sum which give a non-zero contribution
are those which are a component of $\Ind_{W_I}^W\Id$.
\end{corollary}
Using the Lefschetz formula and taking the ``limit for $m\to0$'' 
(see for example \cite[3.3.8]{DMR}) we get the equality of virtual characters
\begin{corollary}\label{Lefschetz character}
$$\sum_i (-1)^i H^i_c(\bX(\bI,\bw\phi),\Qlbar)=
\sum_{\{\chi\in\Irr(W)^\phi\mid \langle\Res^W_{W_I}\chi,\Id\rangle_{W_I}\ne 0\}}
\tilde\chi(x_1 w\phi)R_{\tilde\chi},$$
where $w$ is the image of $\bw$ in $W$ and $x_1=|W_I|\inv\sum_{v\in W_I}v$.
\end{corollary}

\subsection*{Cohomology}
If $\pi$ is the projection of Lemma \ref{tildeX torsor}, the sheaf
$\pi_!\Qlbar$  decomposes  into  a  direct  sum  of  sheaves indexed by the
irreducible characters of $\bL_\bI^{t(\bw\phi)}$. We will denote by $\bchi$
the  subsheaf indexed by the character $\chi\in\Irr(\bL_\bI^{t(\bw\phi)})$,
and in particular by $\bSt$ the subsheaf indexed by the Steinberg character
$\St\in\Irr(\bL_\bI^{t(\bw\phi)})$. We have the isomorphism of
$\bG^F\times\LwF$-modules
$$H^i_c(\tilde\bX(\bI,\bw\phi),\Qlbar)=
\oplus_{\chi\in\Irr(\LwF)} H^i_c(\bX(\bI,\bw\phi),\bchi)\otimes V_\chi$$
where $V_\chi$ is an $\LwF$-module of character $\chi$ and
$H^i_c(\bX(\bI,\bw\phi),\bchi) $ is a $\bG^F$-module.
When $\chi$ is $F^\delta$-stable there is an action of $F^\delta$ on
$V_\chi$ such that the inclusion of $H^i_c(\bX(\bI,\bw\phi),\bchi)\otimes V_\chi$
into $H^i_c(\tilde\bX(\bI,\bw\phi),\Qlbar)$
is an inclusion of $\bG^F\times\LwF\rtimes\genby{F^\delta}$-modules

The  following corollary of Proposition \ref{produit fibre general} relates
the  cohomology of a general variety  $\bX(\bI,\bw\phi)$ to the case of the
varieties  $\bX(\bu\phi)$  considered  in  \cite{DMR};  its  part (ii) is a
refinement   of  Corollary  \ref{Lefschetz  character}.  In  the  following
corollary, if $M$ is a $\Qlbar$-vector space on wich $F$ acts, we denote by
$M(n)$ for $n\in\BZ$ the $n$-th Tate twist of $M$.

\begin{corollary}\label{inclusion des cohomologies}
Let $\bI\xrightarrow{\bw}\lexp\phi\bI\in B^+(\cI)$. Then
\begin{enumerate}
\item For any unipotent $F^\delta$-stable character $\chi\in\Irr(\LwF)$, for any
$\bu\in\uB_\bI^+$ and any $i,j$ we have the inclusion of
$\bG^F\times\genby   {F^\delta}$-modules   $$H^i_c(\bX(\bI,\bw\phi),\bchi)
\otimes (H^j_c(\bX_{\bL_\bI}(\bu w \phi),\Qlbar)
\otimes_\LwF V_{\overline\chi})\subset H^{i+j}_c(\bX(\bu\bw\phi),\Qlbar).$$
\item
For  all $\bv\in B^+_\bI$ and  all $i$ we have  the following inclusions of
$\bG^F\times\genby {F^\delta}$-modules:
$$H^i_c(\bX(\bI,\bw\phi),\Qlbar)\subset
H^{i+2l(\bv)}_c(\bX(\bv\bw\phi),\Qlbar)(-l(\bv))$$ and
$$H^i_c(\bX(\bI,\bw\phi),\bSt)\subset
H^{i+l(\bv)}_c(\bX(\bv\bw\phi),\Qlbar)$$
\item
For all $i$ we have the following equality of 
$\bG^F\times\genby{F^\delta}$-modules:  $$H^i_c(\bX(\uw_I\bw\phi),\Qlbar)=
\sum_{j+2k=i}H^j_c(\bX(\bI,\bw\phi),\Qlbar)\otimes \Qlbar^{n_{I,k}}(k)$$
where $n_{I,k}=|\{v\in W_I\mid l(v)=k\}|$.
\end{enumerate}
\end{corollary}
Note that in (iii) above we have $\bX(\uw_I\bw\phi)=
\bigcup_{\bv\in \bW_\bI}\bX(\bv\bw\phi)$.
\begin{proof}
We apply  the  K\"unneth  formula  to  the  isomorphism  of Corollary
\ref{X(uwI)} and decompose the equality obtained
according  to the characters of $\LwF$; we get that for any
$\bu\in\uB^+_\bI$, we have
$$
\bigoplus_{0\le j\le 2l(\bu)\atop  \chi\in\Irr(\LwF)}
H^{i-j}_c(\tilde\bX(\bI,\bw\phi),\Qlbar)_\chi\otimes_\LwF
H^j_c(\bX_{\bL_\bI}(\bu w \phi),\Qlbar)_{\overline\chi}
\simeq H_c^i(\bX(\bu\bw\phi),\Qlbar),
$$
which can be written
\begin{multline}\label{kunneth}
\bigoplus_{0\le j\le 2l(\bu)\atop  \chi\in\Irr(\LwF)}
H^{i-j}_c(\bX(\bI,\bw\phi),\bchi)\otimes(
H^j_c(\bX_{\bL_\bI}(\bu w \phi),\Qlbar)\otimes_\LwF V_{\overline\chi})
\\ \simeq H_c^i(\bX(\bu\bw\phi),\Qlbar).
\end{multline}
This gives (i).
We get also (ii) from equation \ref{kunneth}
and the facts that for $\bv\in B^+_\bI$
\begin{itemize}
\item
the  only  $j$  such  that  $H^j_c(\bX_{\bL_I}(\bv w\phi),\Qlbar)_{\Id}$ is
non-trivial  is  $j=2l(\bv)$  and  that isotypic component  is irreducible
and $t(\bw \phi)$ acts by $q^{l(\bv)}$ on it (see \cite[3.3.14]{DMR}) and $t(\bw\phi)^{k\delta}$
is equal to $F^{k\delta}$ for some $k$.
\item
the  only  $j$  such  that  $H^j_c(\bX_{\bL_I}(\bv w\phi),\Qlbar)_{\St}$ is
non-trivial  is $j=l(\bv)$ and  that isotypic component  is irreducible
with trivial action of $t(\bw \phi)$ (see \cite[3.3.15]{DMR}).
\end{itemize}
Hence the term $\bchi=\Qlbar$ in the LHS of \ref{kunneth} for
$\bu=\bv$ and $j=2l(\bv)$ is
$H^{i-2l(\bv)}_c(\bX(\bI,\bw\phi),  \Qlbar)\otimes
\Qlbar(-l(\bv))$ and is a submodule of
$H^i_c(\bX(\bI,\bw\bv\phi),\Qlbar)$.
Similarly the term $\bchi=\bSt$ in the LHS 
for $\bu=\bv$ and $j=l(\bv)$ is
$H^{i-l(\bv)}_c(\bX(\bI,\bw\phi), \bSt)$ and is a submodule of
$H^i_c(\bX(\bI,\bw\phi),\Qlbar)$.

We  now prove (iii). By Corollary  \ref{X(uwI)} applied with $\bu=\uw_I$ we
have     an     isomorphism     $\tilde\bX(\bI,\bw\phi)    \times_\LwF\cB_I
\xrightarrow\sim\bX(\uw_I\bw\phi)$  where $\cB_I$  is the  variety of Borel
subgroups of $\bL_I$. We get (iii) from the fact that $H^k_c(\cB_I,\Qlbar)$
is  0 if $k$ is odd and if $k=2k'$ is a trivial $\LwF$-module
of  dimension  $n_{I,k'}$,  where  $F^\delta$  acts  by  the scalar
$q^{\delta k'}$; this
results  for  example  from  the  cellular decomposition into affine spaces
given  by  the  Bruhat  decomposition  and  the  fact  that  the  action of
$\LwF$ extends to the connected group $\bL_I$ so that it acts
trivially on the cohomology.
\end{proof}
\begin{corollary}\label{omega_rho}
Let $\bI\xrightarrow{\bw\phi}\bI\in B^+\phi(\cI)$, let
$\chi\in\Irr(\LwF)$ be unipotent and $F^\delta$-stable, and let $i\in\BN$. Then
\begin{enumerate}
\item
The  $\bG^F$-module  $H^i_c(\bX(\bI,\bw\phi),\bchi)$  is  unipotent.  Given
$\rho\in\Irr(\bG^F)$   unipotent,   the   eigenvalues   of   $F^\delta$  on
$H^i_c(\bX(\bI,\bw\phi),\bchi)_\rho$ are in
$q^{\delta\BN}\lambda_\rho\omega_\rho$,   where  $\lambda_\rho$  is  as  in
Corollary   \ref{shintani2}   and   $\omega_\rho$   is   the   element   of
$\{1,q^{\delta/2}\}$  attached to $\rho$ as  in \cite[3.3.4]{DMR};
$\lambda_\rho$ and $\omega_\rho$ are independent of $i$ and $\bw$.
\item
The eigenvalues of $F^\delta$ on $H^i_c(\bX(\bI,\bw\phi),\bchi)$ have
absolute value at most $q^{\delta i/2}$.
\item
We have $H^i_c(\bX(\bI,\bw\phi),\bchi)=0$ unless $l(\bw)\le i \le 2l(\bw)$.
\item
The Steinberg representation does not occur in
$H^i_c(\bX(\bI,\bw\phi),\bchi)$  unless  $\bchi=\bSt$  and  $i=l(\bw)$,  in
which  case it occurs with multiplicity $1$, associated with the eigenvalue
$1$ of $F^\delta$.
\item
The trivial representation does not occur in
$H^i_c(\bX(\bI,\bw\phi),\bchi)$  unless $\bchi=\Qlbar$  and $i=2l(\bw)$, in
which  case it occurs with multiplicity $1$, associated with the eigenvalue
$q^{\delta l(\bw)}$ of $F^\delta$.
\end{enumerate}
\end{corollary}
\begin{proof}
(i)  is  a  straightforward  consequence  of  equation  \ref{kunneth}
applied for any $\bu$ such that some term
$H^j_c(\bX_{\bL_\bI}(\bu w \phi),\Qlbar)_{\overline\chi}$ is not $0$
for some $j$,
since  the  result  is   known  for  $  H^i_c(\bX(\bu\bw \phi),\Qlbar)$
(see \cite[3.3.4]{DMR} and \cite[3.3.10 (i)]{DMR}).

(ii) and (iii) are a consequence of \ref{kunneth} applied for $\bu\in B^+_\bI$
of minimal length such that $\overline\chi$ appears in some $H^j_c(\bX_{\bL_I}
(\bu w\phi),\Qlbar)$. Then by \cite[3.3.21]{DMR} 
$\overline\chi$ appears in $H^{l(\bu)}_c(\bX_{\bL_I}(\bu w\phi),\Qlbar)$
and the corresponding eigenvalue of $F^\delta$ has module $q^{\delta
l(\bu)/2}$. It follows then from \ref{inclusion des cohomologies}(i) applied
with $j=l(\bu)$ that $H^i_c(\bX(\bI,\bw\phi),\bchi)\otimes V
\subset H^{i+l(\bu)}(\bX(\bu\bw\phi),\Qlbar)$ where $V$ is an
$F^\delta$-module where the eigenvalues of $F^\delta$ are of module $q^{\delta
l(\bu)/2}$. The result follows from the facts that
$H^{i+l(\bu)}(\bX(\bu\bw\phi),\Qlbar)=0$ for $i<l(\bw)$ and that the
eigenvalues of $F^\delta$ on it have a module at most $q^{\delta(i+l(\bu))/2}$.

For (iv), we use
\begin{lemma}  If $\chi\in\Irr(\LwF)$  is unipotent  and $\chi\ne\St$ there
exists  $\bu\in\uB^+_\bI-B^+_\bI$ and $j$  such that $H^j_c(\bX_{\bL_I}(\bu
w\phi),\Qlbar)_\chi\ne 0$.
\end{lemma}
\begin{proof} First, assume that $\chi$ is not in the principal series, and
let  $\bv\in B^+_\bI$ be of minimal length such that $\chi$ appears in some
$H^j_c(\bX_{\bL_I}(\bv   w\phi),\Qlbar)$.  Since  $\chi$   is  not  in  the
principal  series  we  have  $l(\bv)>0$  thus there exists $\bs\in\bI$ and
$\bv'\in  B^+_\bI$ such that $\bv=\bs\bv'$. Then $H^j_c(\bX_{\bL_I}(\us\bv'
w\phi),\Qlbar)_\chi=  H^j_c(\bX_{\bL_I}(\bv w\phi),\Qlbar)_\chi\ne 0$ because
of the minimality of $\bv$ and the long exact sequence resulting from $\bX_{\bL_I}(\us\bv'
w\phi)=\bX_{\bL_I}(\bv   w\phi)\coprod\bX_{\bL_I}(\bv'  w\phi)$  where  the
first  (resp.\ second) term of the RHS  is an open (resp.\ closed) subvariety
of the LHS.

When $\chi$ is in the principal series, 
we use that if $J$ is a $w\phi$-stable
subset of $I$ and $\bu\in \uB_J^+$, then 
$H^j_c(\bX_{\bL_I}(\bu w\phi),\Qlbar)=R_{\bL_J}^{\bL_I}
H^j_c(\bX_{\bL_J}(\bu w\phi),\Qlbar)$.
It follows that if $\chi$ is of the form $\rho_\psi$
for $\psi\in\Irr(W_I^{w\phi})$ (see \cite[5.3.1]{DMR}), 
and $\psi_1$ is a component of $\Res^{W_I^{w\phi}}_{W_J^{w\phi}}\psi$ such that
$\scal{H^j_c(\bX_{\bL_J}(\bu w\phi),\Qlbar)}{\rho_{\psi_1}}{\bL_J^F}\ne 0$,
then $\scal{H^j_c(\bX_{\bL_I}(\bu w\phi),\Qlbar)}{\rho_\psi}{\bL_I^F}\ne 0$.
If $J$ is a $w\phi$-orbit in $I$, the group $W_J^{w\phi}$ is a Coxeter
group of type $A_1$ and the restriction to $W_J^{w\phi}$ of 
a character $\psi$ other than the sign character cannot be isotypic of type
sign for all orbits $J$ ($\psi$ would then be itself isotypic of type sign).
We are thus reduced to the case where $I$ is a single $w\phi$-orbit, so
that $L_I^{t(w\phi)}$ has only two unipotent characters, $\Id$
and $\St$. For such a group the identity character is a component of 
$H^2_c(\bX_{\bL_I}(\us w\phi)\Qlbar)$ where $W_I^{w\phi}=\genby s$,
so that the lemma is true.
\end{proof}
Since for $\bu$ as in the lemma we have $H^*_c(\bX(\bu\bw\phi),\Qlbar)_{\St}=0$
(see \cite[3.3.15]{DMR}),
by \ref{kunneth} we deduce that for $\bchi\ne\bSt$ we have
$H^i_c(\bX(\bI,\bw\phi),\bchi)_{\St}=0$ for all $i$.
Thus, for any $\bu\in B_\bI^+$, using that
$H^j_c(\bX_{\bL_\bI}(\bu w \phi),\Qlbar)\otimes_\LwF V_{\St}=0$ when
$j\ne l(\bu)$, the $\St$-part of \ref{kunneth} reduces to
$$
H^{i-l(\bu)}_c(\bX(\bI,\bw\phi),\bSt)_{\St}\otimes(
H^{l(\bu)}_c(\bX_{\bL_\bI}(\bu w \phi),\Qlbar)\otimes_\LwF V_{\St})
\simeq H_c^i(\bX(\bu\bw\phi),\Qlbar)_{\St}.
$$
We apply this for $\bu=\bv\in B^+_\bI$ in which case
$H^{l(\bv)}_c(\bX_{\bL_I}(\bv w\phi),\Qlbar)\otimes_\LwF V_{\St}=\Qlbar$ with
trivial action of $F^\delta$, which gives the isomorphism of $\bG^F\times\genby
F^\delta$-modules
$$
H^i_c(\bX(\bI,\bw\phi),\bSt)_{\St}
\simeq H_c^{i+l(\bv)}(\bX(\bv\bw\phi),\Qlbar)_{\St}.
$$
using the values of the RHS (known by \cite[3.3.15]{DMR})
$H^{i+l(\bv)}_c(\bX(\bv\bw\phi),\Qlbar)_{\St}
=\begin{cases}0&\text{if }i\ne l(\bw)\\ \Qlbar&\text{ with trivial action of
$F^\delta$ otherwise}
\end{cases}$, we get (iv).

For (v), we use
\begin{lemma}\label{degreeId}  If $\chi\in\Irr(\LwF)$  is unipotent  and $\chi\ne\Id$ there
exists  $\bu\in B^+_\bI$ and $j\ne 2l(\bu)$ such that $H^j_c(\bX_{\bL_I}(\bu
w\phi),\Qlbar)_\chi\ne 0$.
\end{lemma}
\begin{proof} First, assume that $\chi$ is not in the principal series, and
let  $\bu\in B^+_\bI$ be of minimal length such that $\chi$ appears in some
$H^j_c(\bX_{\bL_I}(\bu  w\phi),\Qlbar)$. Then by \cite[3.3.21 (ii)]{DMR} we
have  $H^{l(\bu)}_c(\bX_{\bL_I}(\bu w\phi),\Qlbar)_\chi\ne 0$. Since $\chi$
is  not in  the principal  series we  have $l(\bu)\ne  2l(\bu)$, whence the
lemma in this case.

Now  assume  $\chi$  in  the  principal  series and take $\bu=\bpi_\bI$. It
results  for example from \cite[3.3.8 (i)]{DMR}  that there exists $j$ such
that  $H^j_c(\bX_{\bL_I}(\bpi_\bI w\phi),\Qlbar)\ne 0$.  On the other hand,
it   results  for   example  from   Proposition  \ref{irreducibility}  that
$\bX_{\bL_I}(\bpi_\bI w\phi)$ is irreducible, thus
$H^{2l(\bpi_\bI)}_c(\bX_{\bL_I}(\bpi_\bI w\phi),\Qlbar)$ is a 1-dimensional
module  affording only the trivial representation of $\bG^F$. It follows that
$j\ne 2l(\bpi_\bI)$, whence the lemma.
\end{proof}
Applying \ref{kunneth} for an $\bu$ as in Lemma \ref{degreeId} and using that
$H^i_c(\bX(\bu\bw\phi),\Qlbar)_{\Id}=0$ for $i\ne 2(l(\bw)+l(\bu))$,
we deduce that for $\bchi\ne\Id$ we have
$H^i_c(\bX(\bI,\bw\phi),\bchi)_{\Id}=0$ for all $i$.
Taking now $\bu=1$ and using that
$H^0_c(\bX_{\bL_I}(w\phi),\Qlbar)\otimes_\LwF\Id=\Qlbar$, 
the $\Id$-part of \ref{kunneth} reduces to
$$
H^i_c(\bX(\bI,\bw\phi),\Id)_{\Id}
\simeq H_c^i(\bX(\bw\phi),\Qlbar)_{\Id}.
$$
whence the result using the value of the RHS given by \cite[3.3.14]{DMR}.
\end{proof}
\section{Eigenspaces and roots of $\bpi/\bpi_\bI$}
\label{eigenspaces and roots}
Let $\ell\ne p$ be a prime such that a Sylow $\ell$-subgroup $S$ of $\bG^F$ is
abelian.

Then  ``generic block theory'' (see \cite{BMM}) associates with $\ell$ a root
of unity $\zeta$ and some $w\phi\in W\phi$ such that its $\zeta$-eigenspace
$V$   in   $X:=X_\BR\otimes\BC$   is   non-zero   and   maximal  among
$\zeta$-eigenspaces  of elements  of $W\phi$;  for any  such $\zeta$, there
exists  a  unique  minimal  subtorus  $\bS$  of  $\bT$  such that 
$V\subset X(\bS)\otimes\BC$.  The space $X(\bS)\otimes\BC$ is
the  kernel  of  $\Phi(w\phi)$,  where,
if the coset $W\phi$  is rational (that is, $\phi$ preserves $X(\bT)$) then 
$\Phi$  is  the  $d$-th  cyclotomic polynomial, where $d$ is the order of 
$\zeta$.
Otherwise,  in the ``very  twisted'' cases $\lexp  2B_2, \lexp 2F_4$ (resp.
$\lexp  2G_2$) we  have to  take for $\Phi$ the irreducible cyclotomic
polynomial  over  $\BQ(\sqrt  2)$  (resp.\  $\BQ(\sqrt 3)$) of which
$\zeta$  is a root. The torus $\bS$  is $wF$-stable thus has an $F$-stable
$\bG$-conjugate $\bS'$ in a maximal torus of type $w$; the torus
$\bS'$ is called a $\Phi$-Sylow; we have $|\bS^{\prime F}|=\Phi(q)^{\dim V}$.

The relationship with $\ell$ is that $S$ is a subgroup of $\bS^{\prime F}$, 
and thus that $|\bG^F|/|\bS^{\prime F}|$ is prime to $\ell$; we have
$N_{\bG^F}(S)=N_{\bG^F}(\bS')=N_{\bG^F}(\bL)$  where $\bL:=C_\bG(\bS')$  is a
Levi  subgroup  of  $\bG$  whose  Weyl  group  is $C_W(V)$. Conversely, any
non-zero maximal  $\zeta$-eigenspace determines some primes $\ell$
giving an  abelian Sylow, those which  divide $\Phi(q)$ and  no other
cyclotomic factor of $|\bG^F|$.

The  classes $C_W(V)w\phi$, where $V=\Ker(w\phi-\zeta)$  is maximal, form a
single   orbit  under   $W$-conjugacy  [see   eg.  \cite[5.6(i)]{Br}];  the
maximality   implies  that   all  elements   of  $C_W(V)w\phi$   have  same
$\zeta$-eigenspace.

We  will see in Theorem \ref{bonne racine}(i) that  up to conjugacy  we may assume
that  $C_W(V)$  is  a  standard  parabolic  group  $W_I$;  then the Brou\'e
conjectures  predict that for an  appropriate choice of coset $C_W(V)w\phi$
in  its $N_W(W_I)$-conjugacy  class the  cohomology complex  of the variety
$\bX(\bI,\bw  \phi)$  should  be  a  tilting  complex  realizing  a derived
equivalence  between the unipotent parts  of the principal $\ell$-blocks of
$\bG^F$  and of $N_{\bG^F}(\bS')$. We want to describe explicitly what should be
a ``good'' choice of $w$ (see Conjectures \ref{conjecture}).

Since  it is no more effort  to have a result in  the context of any finite
real  reflection group than for a context which includes the Ree and Suzuki
groups,  we give a  more general statement.  Our situation generalizes that
studied  in \cite{BM},  which corresponds  to the  case $\bI=\emptyset$, or
$\zeta$-regular  elements,  that  is  elements  of  $W\phi$  which  have an
eigenvector  for the eigenvalue $\zeta$  outside the reflecting hyperplanes
(see  \cite[above 6.5]{Springer}); in particular Theorem \ref{bonne racine}
generalizes  \cite[3.11,  6.5]{BM}  and  Theorem  \ref{racine}  generalizes
\cite[3.12,  6.6]{BM};  in  the  \cite{BM}  case,  the  ``$d$-good periodic
maximal''  elements  we  consider  here  reduce  to  ``good $d$-th $\phi$-roots of
$\bpi$''.  Note that we  focus our study  on the $\ell$-principal block (or
$\Phi(q)$-principal  block), which corresponds  to the maximality condition
on  eigenspaces and to what we call ``non-extendable'' periodic elements.
Extendable periodic  elements would  be  needed  in considering more general blocks.

In  what follows we look at real reflection cosets $W\phi$ of finite order,
that  is $W$ is a  finite reflection group acting  on the real vector space
$X_\BR$  and $\phi$  is an  element of  $N_{\GL(X_\BR)}(W)$, such that
$W\phi$  is  of  finite  order  $\delta$,  that is $\delta$ is the smallest
integer  such  that  $(W\phi)^\delta=W$  (equivalently  $\phi$ is of finite
order).  
Since $W$  is transitive  on the  chambers of  the real hyperplane
arrangement it determines, one can always choose $\phi$ in its coset so that it
preserves  a  chamber  of  this  arrangement.  We will do this; thus $\phi$
is $1$-regular, since it has a fixed point outside the
reflecting hyperplanes, thus is of order $\delta$ since $1$ is the only
$1$-regular element of $W$.

\begin{theorem}\label{bonne  racine} Let $W\phi\subset\GL(X_\BR)$
be  a  finite  order  real  reflection  coset, such that $\phi$ preserves a
chamber  of the hyperplane  arrangement on $X_\BR$  determined by $W$, thus
induces  an automorphism of  the Coxeter system  $(W,S)$ determined by this
chamber.  We call again $\phi$ the  induced automorphism of the braid group
$B$  of $W$, and denote by $\bS,\bW$ the  lifts of $S,W$ to $B$ (see 
Example \ref{artin monoids}).

Let $\zeta=e^{2i\pi k/d}$, and let $V$ be a subspace of
$X:=X_\BR\otimes\BC$ on which some element of $W\phi$ acts by $\zeta$. Then
we may choose $V$ in its $W$-orbit such that:
\begin{enumerate}
\item $C_W(V)=W_I$ for some  $I\subset S$.
\item  If $W_Iw\phi$ is the $W_I$-coset of elements which act by $\zeta$ on
$V$, where  $w$  is  $I$-reduced, then $l((w\phi)^i)=(2ik/d)l(w_0w_I\inv)$
for $2ik\le d$, where  we  have extended the length
function to $W\rtimes\genby\phi$ by $l(w\phi^i)=l(w)$.
\end{enumerate}
Further,   we  may  lift  $w$  as  in   (ii)  to  $\bw\in  B^+$  such  that
$\lexp{\bw\phi}\bI=\bI$  and  $(\bw\phi)^d=\phi^d(\bpi/\bpi_\bI)^k$,  where
$\bI\subset\bS$ lifts $I$. Thus $\bI\xrightarrow{\bw\phi}\bI$ is a
$(d,2k)$-periodic element in $B^+\phi(\cI)$, where $\cI$ is the set of subsets
of $\bS$ conjugate to $\bI$.
\end{theorem}
Note  that  the  last  part  implies  that  for  $w$  as  in  (ii)  we have
$(w\phi)^d=\phi^d$.  Note also that  if $2k\le d$,  then (ii) is applicable
for  $i=1$ and we  get $l(\bw)=l(w)=(2k/d)l(w_0w_I\inv)$ thus  $\bw$ is the
unique lift of $w$ to $\bW$.

Since  we  assume  $W\phi$  real,  if  $e^{2i\pi  k/d}$ is an eigenvalue of
$w\phi$,  then  the  complex  conjugate  $e^{2i\pi  (d-k)/d}$  is  also  an
eigenvalue, for the complex conjugate eigenspace; thus we may always assume
that $2k\le d$, so that $\bw\in\bW$.

If  the coset $W\phi$ preserves a  $\BQ$-structure on $X_\BR$ (which is the
case  for cosets  associated with  finite reductive  groups, except for the
``very  twisted'' cases $\lexp 2B_2, \lexp 2G_2$ and $\lexp 2F_4$), we have
more  generally that  if $e^{2i\pi  k/d}$ is  an eigenvalue of $w\phi$,
with $k$ prime $d$, the
Galois  conjugate  $e^{2i\pi/d}$  is  also  an  eigenvalue,  for  a  Galois
conjugate eigenspace; in these cases we may assume $k=1$.

Recall that by our conventions, even though $\phi$ is a finite order
automorphism of $B^+$, in the semi-direct product $B^+\rtimes\genby\phi$ we
take $\genby\phi$ of infinite order.
\begin{proof}[Proof of  Theorem \ref{bonne racine}]
Since  $W\genby\phi$ is finite, we may find  a scalar product on $X_\BR$ 
(extending to an Hermitian product on $X$) invariant by  $W$  and  $\phi$.  
The subspace $X'_\BR$ of $X_\BR$ orthogonal to the fixed points of $W$
(the subspace spanned by the root lines of $W$) identifies with the
reflection representation of the Coxeter system $(W,S)$ (see for example
\cite[Chapitre V \S 3]{bbk}).
We  will  use  the root system $\Phi$ on $X'_\BR$ consisting of the
vectors  of norm $1$  (for the scalar  product) along the  root lines of $W$,
which  is thus preserved by $W\genby \phi$.
By \cite[Chapitre V \S 3 Proposition 1]{bbk} the centralizer of
any subspace of $X$ is a parabolic subgroup of $W$, 
hence conjugate to a standard parabolic subgroup, whence (i).

To  prove (ii) we reprove (i) by changing the order on  $\Phi$, which is
equivalent to do a conjugation by some element of $W$. 
Let $v$  be a regular  vector in $V$,  that is $v\in V$
such that $C_W(v)=C_W(V)$. Multiplying $v$ if needed by a complex number of
absolute  value $1$,  we may  assume that  for any  $\alpha\in\Phi$ we have
$\Re\langle  v,\alpha\rangle=0$ if and only if $\langle v,\alpha\rangle=0$.
Then   there   exists   an   order   on  $\Phi$  such  that  $\Phi^+\subset
\{\alpha\in\Phi\mid  \Re(\langle v,\alpha\rangle)\ge 0\}$. Let $\Pi$ be the
corresponding    basis;    the    subset   $I=\{\alpha\in\Pi|   \Re(\langle
v,\alpha\rangle)=0\}$ is such that $C_W(V)=C_W(v)=W_I$, and
$\Phi_I=\{\alpha\in\Phi\mid\langle  v,\alpha\rangle=0\}$  is  a root system
for $W_I$.

Note that $(w\phi)^d=\phi^d$. Indeed $(w\phi)^d$ fixes $v$, thus preserves the
sign  of any root not in $\Phi_I$; since $w$ is chosen $I$-reduced we have
$\lexp{w\phi}I=I$, so that $w\phi$ also preserves the
sign  of roots in $\Phi_I$.  It is thus equal  to the only element $\phi^d$ of
$W\phi^d$ which preserves the signs of all roots. We get also that
$\lexp{\phi^d}I=I$. If we notice that we may lift $\phi$ to
$\phi\bpi/\bpi_\bI$, this completes the proof in the case $d=1$.

We now assume that $d\ne 1$ and we first prove the theorem in the case $k=1$.
Since      $\langle     v,     \lexp{(w\phi)^m}     \alpha\rangle=     \langle
\lexp{(w\phi)^{-m}}v,\alpha\rangle=  \zeta^{-m} \langle  v, \alpha\rangle$, we
get that all orbits of $w\phi$ on $\Phi-\Phi_I$ have cardinality a multiple of
$d$;  it is  thus possible  by partitioning  suitably those  orbits, to  get a
partition   of  $\Phi-\Phi_I$  in   subsets  $\cO$  of   the  form  $\{\alpha,
\lexp{w\phi}\alpha,  \dots, \lexp{(w\phi)^{d-1}}  \alpha\}$; and  the numbers
$\{\langle  v,\beta\rangle\mid \beta\in\cO\}$ for a  given $\cO$ form  the 
vertices of a
regular  $d$-gon centered at $0\in\BC$; the action of $w\phi$ is the rotation
by  $-2\pi/d$ of this  $d$-gon. Looking at  the real parts  of the vertices of
this  $d$-gon, we see that for $m\le d/2$, exactly $m$ positive roots in $\cO$
are  sent to negative roots by $(w\phi)^m$. Since this holds for all $\cO$, we
get that for $m\le d/2$ we have
$l((w\phi)^m)=\frac{m|\Phi-\Phi_I|}d$;  thus if $\bw$  is the lift of
$w$ to $\bW$ we have $(\bw\phi)^i\in\bW\phi^i$ if $2i\le d$.

Now we finish the case $k=1, d\ne 1$ with the following
\begin{lemma}\label{bw good} Assume that $\lexp{w\phi}W_I=W_I$, that $w$ is
$I$-reduced, and that for some $d>1$ we have
$(w\phi)^d=\phi^d$  and  $l((w\phi)^i)=(2i/d)l(w_0w_I\inv)$ if
$2i\le d$. Then if $\bw$ is the lift of $w$ to $\bW$ we have
$\lexp{\bw\phi}\bI=\bI$ and $(\bw\phi)^d=\phi^d\bpi/\bpi_\bI$.
\end{lemma}
\begin{proof}
Since  $w$ is  $I$-reduced and  $w\phi$ normalizes  $W_I$ we  get that $w\phi$
stabilizes $I$; these properties imply in the braid monoid the equality
$\lexp{\bw\phi}\bI=\bI$.

Assume  first $d$  even and  let $d=2d'$  and $x=\phi^{-d'}(w\phi)^{d'}$. Then
$l(x)=(1/2) l(\bpi/\bpi_\bI)=l(w_0)-l(w_I)$ and since $x$ is reduced-$I$ it is
equal  to the only reduced-$I$ element of  that length which is $w_0 w_I\inv$.
Since  the lengths add  we can lift  the equality $(w\phi)^{d'} = \phi^{d'}w_0
w_I\inv  $ to the braid monoid as $(\bw\phi)^{d'}=\phi^{d'}\bw_0\bw_I\inv$. By
a   similar  reasoning  using  that  $(w\phi)^{d'}\phi^{-d'}$  is  the  unique
$I$-reduced element of its length, we get also
$(\bw\phi)^{d'}=\bw_I\inv\bw_0\phi^{d'}$. Thus
$(\bw\phi)^d=\bw_I\inv\bw_0\phi^{d'}\phi^{d'}\bw_0\bw_I\inv
=\phi^d\bpi/\bpi_\bI$,  where the  last equality  uses that $\phi^d=(w\phi)^d$
preserves $\bI$.

Assume  now that  $d=2d'+1$; then  $(w\phi)^{d'}\phi^{-d'}$ is $I$-reduced and
$\phi^{-d'}(w\phi)^{d'}$  is reduced-$I$. Using  that any reduced-$\bI$ element
of $\bW$ is  a  right-divisor  of $\bw_0\bw_I\inv$ (resp.\ any $\bI$-reduced
element of $\bW$ is a left-divisor of $\bw_I\inv\bw_0$), we get that there exists
$\bt,\bu\in\bW$  such  that $\phi^{d'}\bw_I\inv\bw_0=\bt(\bw\phi)^{d'}$ and
$\bw_0\bw_I\inv\phi^{d'}=(\bw\phi)^{d'}\bu$.      Thus     $\phi^d\bpi/\bpi_\bI=
\bw_0\bw_I\inv\phi^d\bw_I\inv\bw_0=(\bw\phi)^{d'}\bu\phi\bt(\bw\phi)^{d'}$,
the  first equality  since $\lexp{\phi^d}I=I$.  The image  in $W\phi^d$ of the
left-hand  side is $\phi^d$, and $(w\phi)^d=\phi^d$.  We deduce that the image
in  $W\phi$ of $\bu\phi\bt$ is $w\phi$. 
If $d\ne 1$ then $d'\ne 0$ and we have $l(\bu)=l(\bt)=l(\bw)/2$; thus
$\bu\phi\bt=\bw\phi$ and $(\bw\phi)^d=\phi^d\bpi/\bpi_\bI$.
\end{proof}
We  now consider the case $k\ne 1$, $d\ne 1$. We have seen (before assuming
$k=1$)  that (i) holds  and that the  $I$-reduced element $w$  of the coset
$W_I w\phi$ acting by $\zeta$ on $V$ satisfies $(w\phi)^d=\phi^d$.

We first consider the case when $k$ is prime to $d$.
Let   $d',k'$  be  positive   integers  such  that   $kd'=1+dk'$,  and  let
$w_1\phi_1=(w\phi)^{d'}$,  where $\phi_1=\phi^{d'}$.  Then $w_1\phi_1$ acts
on $V$ by $e^{2i\pi/d}$, so by the case $k=1$ we have
$l((w_1\phi_1)^i)=(2i/d)l(w_0w_I\inv)$ for $2i\le d$. Since
$(w_1\phi_1)^{ik}=(w\phi)^{ikd'}=(w\phi)^{i(1+dk')}=  (w\phi)^i\phi^{idk'}$,
we get (ii).

By Lemma \ref{bw good} the lift $\bw_1$ of $w_1$ to $B$
satisfies     $\lexp{\bw_1\phi_1}\bI=\bI$    and    $(\bw_1\phi_1)^d=\phi_1^d
\bpi/\bpi_\bI$, thus if we define $\bw$ by $(\bw_1\phi_1)^k=\bw\phi^{1+dk'}$,
then $\bw$  lifts $w$ and satisfies $(\bw\phi)^d=\phi^d(\bpi/\bpi_\bI)^k$,
using $\lexp{\phi^d}I=I$.

We finally consider the general case $d=\lambda d_1$, $k=\lambda k_1$ where
$d_1$  is prime to  $k_1$. The theorem  holds for $d_1,k_1$; statement (ii)
depends  only  on  $k/d$  thus  holds, and  we  just have to raise the
equation     $(\bw\phi)^{d_1}=(\bpi/\bpi_\bI)^{k_1}\phi^{d_1}$    to    the
$\lambda$-th power to get the desired equation
$(\bw\phi)^d=(\bpi/\bpi_\bI)^k\phi^d$.
\end{proof}

We give now a kind of converse of Theorem \ref{bonne racine}.

\begin{theorem}\label{racine}
Let $(W,S)$, $\phi$, $X_\BR$, $X$, $\bS, B, B^+$ be as in Theorem
\ref{bonne racine}. For $d\in\BN$, let $\bw\in B^+$ be such that
$(\bw \phi)^d=\phi^d(\bpi/\bpi_\bI)^k$ for some $\phi^d$-stable
$\bI\subset\bS$. Then
\begin{enumerate}
\item $\lexp{\bw\phi}\bI=\bI$, and $\bI\xrightarrow{\bw\phi}\bI$ is a
$(d,2k)$-periodic element in $B^+\phi(\cI)$, where $\cI$ is the set of subsets
of $\bS$ conjugate to $\bI$.
\end{enumerate}
Denote by $w$ and $I$ the images in $W$ of $\bw$ and $\bI$, let
$\zeta=e^{2i\pi k/d}$,   let  $V\subset  X$  be  the  $\zeta$-eigenspace  of
$w\phi$,   and  let  $X^{W_I}$  be  the  fixed  point  space  of  $W_I$;  then
\begin{itemize}
\item[(ii)] $W_I=C_W(X^{W_I}\cap V)$, in particular $C_W(V)\subset W_I$.
\end{itemize}
Further, the following two assertions are equivalent:
\begin{itemize}
\item[(iii)] No element of the coset $W_Iw\phi$ has a non-zero
$\zeta$-eigenvector on the subspace spanned by the root lines of $W_I$.
\item[(iv)] $\bw\phi$ is ``non-extendable'', that is,  there   do  not   exist  a
$\phi^d$-stable   $\bJ\subsetneq\bI$   and   $\bv\in   B_\bI^+$  such  that
$(\bv\bw\phi)^d=\phi^d(\bpi/\bpi_\bJ)^k$.
\end{itemize}
\end{theorem}
\begin{proof}
We will deduce the general case from the case $k=1$. 

So we first assume $k=1$. Then
(i) is already in Proposition \ref{periodic ribbons} which also states that
there    exists   $\bI\xrightarrow\bv\bJ\in   B^+(\cI)$    such   that   if
$\bw'\phi=(\bw\phi)^\bv$ then $\bw'\phi\in B^+\phi$,
$(\bw'\phi)^d=\phi^d\bpi/\bpi_\bJ$       and       $(\bw'\phi)^{\lfloor\frac
d2\rfloor}\in \bW \phi^{\lfloor\frac d2\rfloor}$.

As  (ii) and the equivalence of (iii) and (iv) are invariant by a conjugacy
in  $B$ which sends $\bw\phi$  to $B^+\phi$ and $\bI$  to another subset of
$\bS$,  we may replace $(\bw\phi,\bI)$ by the conjugate $(\bw'\phi,\bJ)$,
thus assume that $w$ and $I$ satisfy the assumptions of the next lemma.

\begin{lemma} \label{pas merdique}
Let $w\in W, I\subset S$ be such that
$(w\phi)^d=\phi^d$, $\lexp{w\phi}I=I$ and such that 
$l((w\phi)^i)=\frac{2i}d l(w_I\inv w_0)$ for any $i\le d/2$. We have
\begin{enumerate}
\item If $\Phi$ is a root system for $W$ and $\Phi^+$ is chosen such that
$\phi(\Phi^+)=\Phi^+$ (as in the proof of Theorem \ref{bonne racine}), then
$\Phi-\Phi_I$ is the disjoint union of sets of the form
$\{\alpha,\lexp{w\phi}\alpha,\dots,\lexp{(w\phi)^{d-1}}\alpha\}$ with
$\alpha,\lexp{w\phi}\alpha,\dots,\lexp{(w\phi)^{\lfloor d/2\rfloor-1}}
\alpha$ of same sign and
$\lexp{(w\phi)^{\lfloor d/2\rfloor}}\alpha,\dots,\lexp{(w\phi)^{d-1}}\alpha$
of the opposite sign.
\item The order of $w\phi$ is $\lcm(d,\delta)$.
\item If $d>1$, then $W_I=C_W(X^{W_I}\cap\ker(w\phi-\zeta))$.
\end{enumerate}
\end{lemma}
\begin{proof}
The statement is empty for $d=1$ so in the following proof we assume $d>1$.

For $x\in W\rtimes\langle\phi\rangle$ let
$N(x)=\{\alpha\in\Phi^+\mid \lexp x\alpha\in\Phi^-\}$; 
for $x\in W$ we have $l(x)=|N(x)|$ (see \cite[Chapitre VI \S1, Corollaire 2]{bbk}).
This still holds for $x=w\phi^i\in W\rtimes\langle\phi\rangle$
since $N(w\phi^i)=\lexp{\phi^{-i}}N(w)$. It follows that for $x,y\in
W\rtimes\langle\phi\rangle$ we have $l(xy)=l(x)+l(y)$ if and only if
$N(xy)=N(y)\coprod \lexp{y\inv}N(x)$.
In particular
$l((w\phi)^i)=i l(w\phi)$ for $i\le d/2$ implies
$\lexp{(w\phi)^{-i}}N(w\phi)\subset\Phi^+$ for $i\le d/2-1$. 

Let us partition each $w\phi$-orbit in $\Phi-\Phi_I$ into ``pseudo-orbits'' 
of the form $\{\alpha,\lexp{w\phi}\alpha,\dots,\lexp{(w\phi)^{k-1}}\alpha\}$,
where  $k$ is  minimal such  that $\lexp{(w\phi)^k}\alpha=\lexp{\phi^k}\alpha$  
(then $k$ divides  $d$); a pseudo-orbit is
an  orbit if $\phi=1$. The  action of $w\phi$ defines  a cyclic order on each
pseudo-orbit. The previous paragraph shows that when there is a sign change in
a  pseudo-orbit, at least  the next $\lfloor  d/2\rfloor$ roots for the cyclic
order  have the same sign. On the other hand, as $\phi^k$ preserves $\Phi^+$,
each pseudo-orbit contains an even number of sign changes. Thus if there is at
least  one sign change  we have $k\ge  2\lfloor d/2\rfloor$. Since $k$ divides
$d$,  we must have $k=d$ for pseudo-orbits  which have a sign change, and then
they  have exactly two  sign changes. As  the total number  of sign changes is
$2l(w)=2|\Phi-\Phi_I|/d$,  there are $|\Phi-\Phi_I|/d$ pseudo-orbits with sign
changes;  their total cardinality is $|\Phi-\Phi_I|$,  thus there are no other
pseudo-orbits  and  up  to  a  cyclic  permutation  we  may  assume  that each
pseudo-orbit  consists of  $\lfloor d/2\rfloor$  roots of  the same  sign
followed by $d- \lfloor d/2\rfloor$ of the opposite sign. We have proved (i).

Let $d'=\lcm(d,\delta)$.
The proof of (i) shows that the order of $w\phi$ is a multiple of $d$.
Since the order of $(w\phi)^d=\phi^d$ is $d'/d$, we get (ii). 

We  now  prove  (iii). Let $V=\ker(w\phi-\zeta)$.
Since  $W\genby\phi$ is finite, we may find  a scalar product on $X$ 
invariant by  $W$  and  $\phi$. We have then $X^{W_I}=\Phi_I^\perp$.
The map $p=\frac 1{d'}\sum_{i=0}^{d'-1}\zeta^{-i} (w\phi)^i$ is a
$w\phi$-invariant  projector  on  $V$,  thus  is  the  orthogonal
projector on $V$.

We  claim that  $p(\alpha)\not\in<\Phi_I>$ for  any $\alpha\in\Phi-\Phi_I$. As
$p((w\phi)^i\alpha)=\zeta^i  p(\alpha)$ it is enough to assume that $\alpha$
is  the  first  element  of  a  pseudo-orbit;  replacing if needed $\alpha$ by
$-\alpha$  we may even assume $\alpha\in  \Phi^+$. Looking at imaginary parts,
we   have  $\Im(\zeta^i)\ge   0$  for   $0\le  i<\lfloor   d/2\rfloor$,  and
$\Im(\zeta^i)<0$  for $\lfloor d/2\rfloor\le i<d$. Let $\lambda$ be a linear
form  such  that  $\lambda$  is  0  on  $\Phi_I$  and  is real strictly positive on
$\Phi^+-\Phi_I$;   we   have   $\lambda(\lexp{(w\phi)^i}\alpha)>0$  for  $0\le
i<\lfloor  d/2\rfloor$,  and  $\lambda(\lexp{(w\phi)^i}\alpha)<0$ for $\lfloor
d/2\rfloor\le i<d$; it follows that
$\Im(\lambda(\zeta^i\lexp{(w\phi)^i}\alpha))=
\Im(\zeta^i\lambda(\lexp{(w\phi)^i}\alpha))>0$  for  all  elements  of  the
pseudo-orbit.   If  $d'=d$   we  have   thus  $\Im(\lambda(p(\alpha)))>0$,  in
particular $p(\alpha)\not\in<\Phi_I>$. If $d'>d$, since $\phi^d\alpha$ is also
a  positive  root  and  the  first  term  of  the  next  pseudo-orbit the same
computation applies to the other pseudo-orbits and we conclude the same way.

Now  $C_W(X^{W_I}\cap V)$ is generated by  the reflections whose root is
orthogonal   to  $X^{W_I}\cap V$,   that  is   whose  root is in
$<\Phi_I>+V^\perp$. If $\alpha$ is such  a root we have $p(\alpha)\in<\Phi_I>$,
whence $\alpha\in\Phi_I$ by the above claim. This proves that
$C_W(X^{W_I}\cap V)\subset  W_I$. Since  the reverse  inclusion is true, we
get (iii).
\end{proof}

We return to the proof of the case $k=1$ of Theorem \ref{racine}.
Assertion (iii) of Lemma \ref{pas merdique} gives the second assertion of the theorem.
We now show $\neg$(iv)$\Rightarrow\neg$(iii). If $\bw\phi$ is extendable, 
there exists a $\phi^d$-stable $\bJ\subsetneq\bI$ and $\bv\in B_\bI^+$
such that $(\bv\bw\phi)^d=\phi^d\bpi/\bpi_\bJ$, which implies
$\lexp{\bv\bw\phi}\bJ=\bJ$.
If we denote by $\psi$ the automorphism of $B_\bI$ induced by the automorphism
$\bw\phi$ of $\bI$, we have $\lexp{\bv\psi}\bJ=\bJ$ and 
$(\bv\psi)^d=\psi^d\bpi_\bI/\bpi_\bJ$.
Let $X_I$ be the subspace of $X$ spanned by $\Phi_I$.
It follows from the first part of
the theorem applied with $X$, $\phi$, $\bw$, $w$ respectively
replaced with $X_I$, $\psi$, $\bv$, $v$ that $v\psi=vw\phi$
has a non-zero $\zeta$-eigenspace in $X_I$, since if $V'$ is the
$\zeta$-eigenspace of $vw\phi$ we get
$C_{W_I}(V')\subset W_J\subsetneq W_I$; this contradicts (iv).

We show finally that $\neg$(iii)$\Rightarrow\neg$(iv).
If some element of $W_I\psi$ has a non-zero $\zeta$-eigenvector on
$X_I$, by Theorem \ref{bonne racine} applied to $W_I\psi$ acting on $X_I$ we get the
existence of $\bJ\subsetneq\bI$ and $\bv\in B_\bI^+$ satisfying 
$\lexp{\bv\psi}\bJ=\bJ$ and $(\bv\psi)^d=\psi^d\bpi_\bI/\bpi_\bJ$. 
Using that $(\bw\phi)^d=\phi^d\bpi/\bpi_\bI$, it
follows that $(\bv\bw\phi)^d=(\bw\phi)^d\bpi_\bI/\bpi_\bJ=
\phi^d\bpi/\bpi_\bI\cdot\bpi_\bI/\bpi_\bJ=\phi^d\bpi/\bpi_\bJ$  so $\bw\phi$  is 
extendable.

We now deal with the general case $k\ne 1$. This time we use
\ref{ribbon-Bestvina}, which gives immediately (i).
Let us first consider the case when $k$ is prime to $d$.
Then, by \ref{ribbon-Bestvina},
up to conjugacy in $B^+(\cI)$, which we may as well do as observed at the
beginning of the proof, we get that with $d'$ and $k'$ as
in \ref{ribbon-Bestvina} we have $(\bw\phi)^{d'}\preccurlyeq
(\bpi/\bpi_\bI)^{k'}$ and the element $\bw_1$ defined by
$(\bw\phi)^{d'}\bw_1\phi^{-d'}=(\bpi/\bpi_\bI)^{k'}$ satisfies 
$ (\bw_1\phi^{-d'})^k=(\bw\phi)\phi^{-k'd}$
and $ (\bw_1\phi^{-d'})^d=\bpi/\bpi_{\bI}\phi^{-dd'}$.
Since $\bI$ is $\phi^{-dd'}$-stable the last equality shows that we may apply the case
$k=1$ to $\bw_1\phi^{-d'}$. Since $k$ is prime to $d$ the defining relation
for $\bw_1$ gives in $W$ that $(w\phi)^{-d'}=w_1\phi^{-d'}$, where $w_1$ is
the  image of $\bw_1$ in $W$, which (since $d'$ is prime to $d$) shows that
that  the $\zeta$-eigenspace of $w\phi$  is the $e^{2i\pi/d}$-eigenspace of
$w_1\phi^{-d'}$. This gives (ii).

Similarly the coset $W_I w_1\phi^{-d'}$ is the $-d'$-th power of the coset
$W_I w\phi$, so condition (iii) for $w\phi$ and $\zeta$ is
equivalent to (iii) for $w_1\phi^{-d'}$ and $e^{2i\pi/d}$.

Item (ii) of the
following lemma completes the proof of the case $\gcd(d,k)=1$ since by
Lemma \ref{pas Dirichlet} we may choose $d'$ prime to $\delta$;
\begin{lemma}\label{w and w1 extendable}
Let $k,d,k',d'$ be positive integers satisfying
$dk'=1+kd'$ with $d'$ prime to the order of $\phi$. Let $\bw_1\phi^{-d'}$
be $(d,2)$-periodic element. Define 
$\bw\phi$ by $\bw\phi= (\bw_1\phi^{-d'})^k\phi^{k'd}$. Then
\begin{enumerate}
\item $\bw\phi$ is $(d,2k)$-periodic.
\item 
$\bw\phi$ is non-extendable if and only if $\bw_1\phi^{-d'}$ is non-extendable.
\end{enumerate}
\end{lemma}
\begin{proof}
Assertion (i) is an immediate translation of
\ref{x1 periodic implies x periodic}.
Assume  $\bw_1\phi^{-d'}$  extendable,  that  is  there
exists $\bv_1$ such that $(\bv_1\bw_1\phi^{-d'})^d=
\phi^{-dd'}\bpi/\bpi_\bJ$  for some $\bJ\subsetneq\bI$. The $k$-th power of
this equality gives $(\bv(\bw_1\phi^{-d'})^k)^d=
(\bv\bw\phi\cdot\phi^{-k'd})^d=\phi^{-kdd'}(\bpi/\bpi_\bJ)^k$,  where $\bv$
is  defined by  $(\bv_1\bw_1\phi^{-d'})^k=\bv(\bw_1\phi^{-d'})^k$. Since
$\bw_1\phi^{-d'}$ is $(d,2)$-periodic, it is $\phi^{dd'}$-stable, and the
defining equality for $\bv_1$ shows that $\bv_1$ also is $\phi^{dd'}$-stable.
It follows that $\bv\bw\phi$ is also $\phi^{dd'}$-stable.
Since $d'$ is prime to $\delta$ any element commuting to $\phi^{dd'}$ commutes to
$\phi^d$,   in   particular   $(\bv\bw\phi\cdot\phi^{-k'd})^d\phi^{kdd'}=
(\bv\bw\phi)^d\phi^{-k'd^2+kdd'}=(\bv\bw\phi)^d\phi^{-d}$,    whence    the
result.

For the converse, if $w\phi$ denotes the automorphism of $B^+_\bI$
induced  by $\bw\phi$, using that $(\bw\phi)^d=\phi^d(\bpi/\bpi_\bI)^k$ and
that  $\bpi/\bpi_\bJ=(\bpi/\bpi_\bI)(\bpi_\bI/\bpi_\bJ)$  we  may write the
equation $(\bv\bw\phi)^d=\phi^d(\bpi/\bpi_\bJ)^k$ as $(\bv
w\phi)^d=\phi^d(\bpi_\bI/\bpi_\bJ)^k$.  We  now  use  a relative version of
\ref{ribbon-Bestvina},  where we replace $B^+(\cI)$ by $B^+_\bI(\cJ)$ where
$\cJ$  is  the  set  of  $B^+_\bI$-conjugates  of  $\bJ$, replace $\phi$ by
$w\phi$  and replace $\bb$ by  $\bv$; we get the  existence of $\bv_1$ such
that $(\bv_1(w\phi)^{-d'})^d=\bpi_\bI/\bpi_\bJ(w\phi)^{-dd'}$, which can be
written
$(\bv_1(\bw\phi)^{-d'})^d(\bpi/\bpi_\bI)^{kd'}=\bpi_\bI/\bpi_\bJ\phi^{-dd'}$
or   $(\bv_1(\bw\phi)^{-d'}(\bpi/\bpi_\bI)^{k'})^d=\bpi/\bpi_\bJ\phi^{-dd'}$
which   using   that   $(\bw\phi)^{d'}\bw_1\phi^{-d'}=(\bpi/\bpi_\bI)^{k'}$
transforms into the equality we seek
$(\bv_1\bw_1\phi^{-d'})^d=\bpi/\bpi_\bJ\phi^{-dd'}$.
\end{proof}

We   now  consider   the  case   when  $\lambda=\gcd(d,k)\ne   1$.  We  set
$d_1=d/\lambda$  and $k_1=k/\lambda$. Up to  cyclic conjugacy, which we may
as  well do, we may  assume by \ref{ribbon-Bestvina} that $(\bw\phi)^{d_1}=
(\bpi/\bpi_\bI)^{k_1}\phi^{d_1}$.  Since $e^{2i\pi  k_1/d_1}=e^{2i\pi k/d}$
we  have (i), (ii) of the theorem as  well as the equivalence of (iii) with
the  ``$d_1$-extendability''  of  $\bw$,  that  is the existence of $\bv\in
B_\bI^+$  such that  $(\bv\bw\phi)^{d_1}=\phi^{d_1} (\bpi/\bpi_\bJ)^{k_1}$.
The  $d_1$-extendability implies trivially  the $d$-extendability by raising
the equation to the $\lambda$-th power. Conversely, using as above that the
equation  $(\bv\bw\phi)^d=\phi^d (\bpi/\bpi_\bJ)^k$ is  equivalent to $(\bv
w\phi)^d=\phi^d(\bpi_\bI/\bpi_\bJ)^k$     the    relative     version    of
\ref{ribbon-Bestvina}  as used above  shows that up  to cyclic conjugacy we
have  $(\bv w\phi)^{d_1}=\phi^{d_1}(\bpi_\bI/\bpi_\bJ)^{k_1}$ which in turn
is equivalent to $(\bv\bw\phi)^{d_1}=\phi^{d_1} (\bpi/\bpi_\bJ)^{k_1}$.
\end{proof}

The non-extendability condition (iii) or (iv) of Theorem \ref{racine} is equivalent to
the conjunction of two others, thanks to the
following lemma which holds for any complex reflection coset and any $\zeta$.
For definitions and basic results on complex reflection groups we refer to
\cite{Br}. Recall that a complex reflection group is a finite group generated
by pseudo-reflections acting on a finite dimensional complex vector space and that
the fixator of a subspace is called a parabolic subgroup. It is still a
complex reflection group.

\begin{lemma} \label{equiv complex}
Let $W$ be finite a reflection group on the complex 
vector space $X$ and let $\phi$ be an automorphism of $X$ of finite order
which normalizes
$W$.  Let $V$ be the $\zeta$-eigenspace of an element $w\phi\in W\phi$.
Assume that $W'$ is a parabolic subgroup of $W$ which is $w\phi$-stable
and such that $C_W(V)\subset W'$, and let $X'$ denote the subspace of
$X$ spanned by the root lines of $W'$. Then the condition
\begin{itemize}
\item[(i)] $V\cap X'=0$.
\end{itemize}
is equivalent to
\begin{itemize}
\item[(ii)] $C_W(V)=W'$.
\end{itemize}
While the stronger condition
\begin{itemize}
\item[(iii)]
No element of the coset $W'w\phi$ has a non-zero
$\zeta$-eigenvector on $X'$.
\end{itemize}
is equivalent to the conjunction of (ii) and
\begin{itemize}
\item[(iv)] The space
$V$ is maximal among the $\zeta$-eigenspaces of elements of $W\phi$.
\end{itemize}
\end{lemma}
\begin{proof} 
Since $W\genby\phi$ is finite we may endow $X$ with a
$W\genby\phi$-invariant Hermitian scalar product, which we shall do.

We  show  (i)  $\Leftrightarrow$  (ii).  Assume  (i);  since $w\phi$ has no
non-zero  $\zeta$-eigenvector in $X'$  and $X'$ is  $w\phi$-stable, we have
$V\perp  X'$, so  that $W'\subset  C_W(V)$, whence  (ii) since  the reverse
inclusion  is true by  assumption. Conversely, (ii)  implies that $V\subset
X^{\prime\perp}$ thus $V\cap X'=0$.

We  show (iii) $\Rightarrow$ (iv). There exists an element of $W\phi$ whose
$\zeta$-eigenspace   $V_1$   is   maximal   with   $V\subset   V_1$.   Then
$C_W(V_1)\subset  C_W(V)\subset W'$ and the $C_W(V_1)$-coset of elements of
$W\phi$  which act  by $\zeta$  on $V_1$  is a  subset of the coset $C_W(V)
w\phi$  of elements which act by $\zeta$ on  $V$. Thus this coset is of the
form  $C_W(V_1)v  w\phi$  for  some  $v\in  W'$.  By (i) $\Rightarrow$ (ii)
applied with $w\phi$ replaced by $vw\phi$ we get $C_W(V_1)=W'$. Since $v\in
W'$  this implies that  $vw\phi$ and $w\phi$  have same action  on $V_1$ so
that $w\phi$ acts by $\zeta$ on $V_1$, thus $V_1\subset V$.

Conversely,  assume that (ii)  and (iv) are  true. If there  exists $v\in W'$
such  that $vw\phi$ has a non-zero $\zeta$-eigenvector in $X'$, then since $v$
acts trivially on $V$ by (ii), the element $vw\phi$ acts by $\zeta$ on $V$ and
on  a non-zero vector of $X'$ so has a $\zeta$-eigenspace strictly larger that
$V$, contradicting (iv).
\end{proof}

Let  us give now examples  which illustrate the need  for the conditions in
Theorem \ref{racine} and Lemma \ref{equiv complex}.

We  first give an example where $\bw\phi$ is a root of $\bpi/\bpi_\bI$ which
is extendable in the sense of Theorem \ref{racine}(iv) and $\ker(w\phi-\zeta)$
is not maximal: let us  take $W=W(A_3)$, $\phi=1$, $d=2$, $\zeta=-1$,
$\bI=\{\bs_2\}$  (where the conventions for the generators of $W$ are as in
the appendix, see Subsection \ref{An}), $\bw=\bw_\bI\inv\bw_0$. We have
$\bw^2=\bpi/\bpi_\bI$  but $\ker(w+1)$ is not maximal: its dimension is 1
and a 2-dimensional $-1$-eigenspace is obtained for $\bw=\bw_0$.

In  the above  example we  still have  $C_W(V)=W_I$ but  even this need not
happen;   at  the   same  time   we  illustrate   that  the  maximality  of
$V=\ker(w\phi-\zeta)$   does  not   imply  the non-extendability of   $\bw$  if
$C_W(V)\subsetneq  W_I$; we  take $W=W(A_3)$,  $\phi=1$, $d=2$, $\zeta=-1$,
but   this   time   $I=\{\bs_1,\bs_3\}$,   $\bw=\bw_I\inv\bw_0$.   We  have
$\bw^2=\bpi/\bpi_\bI$  and  $\ker(w+1)$  is  maximal  ($w$  is conjugate to
$w_0$, thus $-1$-regular) but $\bw$ is extendable. In this case $C_W(V)=\{1\}$.

The  smallest example with a non-extendable $\bw$ and non-trivial $\bI$ is for
$W=W(A_4)$,  $\phi=1$,  $d=3$,  $\bw=  \bs_1\bs_2\bs_3\bs_4\bs_3\bs_2$  and
$\bI=\{\bs_3\}$.   Then  $\bw^3=\bpi/\bpi_\bI$;  this  corresponds  to  the
smallest example with a non-regular eigenvalue (we call regular an
eigenvalue of a regular element for which the eigenspace has trivial
centralizer): $\zeta_3$ is not regular in $A_4$.

Finally we give an example which illustrates the necessity of the condition
$\phi^d(\bI)=\bI$ in Theorem \ref{racine}. We take $W\phi$ of
type $\lexp 3D_4$, thus $\phi$ is the
triality automorphism $\bs_1\mapsto\bs_4\mapsto\bs_2$. Let
$\bw=\bw_0\bs_1\inv\bs_2\inv\bs_4$.  Then, for $\bI=\{\bs_1\}$
we have $(\bw\phi)^2=\bpi/\bpi_\bI\phi^2$, but $\bI^{\bw\phi}=\{\bs_4\}$.
The other statements of Theorem \ref{racine} also fail: if $V$ is the $-1$-eigenspace
of $w\phi$ the group $C_W(V)$ is the parabolic subgroup generated by
$s_1, s_2$ and $s_4$.

\begin{lemma}\label{N(V)=N(WIwphi)}
Let $W\phi$ be a complex reflection coset and let $V$ be the
$\zeta$-eigenspace of $w\phi\in W\phi$; then
\begin{enumerate}
\item $N_W(V)=N_W(C_W(V)w\phi)$.
\item If $W\phi$ is real,  and $C_W(V)=W_I$ where $(W,S)$ is a Coxeter system 
and $I\subset S$, and $w$ is $I$-reduced, then the subgroup 
$\{v\in C_W(w\phi)\cap N_W(W_I)\mid \text{$v$ is $I$-reduced}\}$
is a section of $N_W(V)/C_W(V)$ in $W$.
\end{enumerate}
\end{lemma}
\begin{proof}
Let $W_1$ denote the parabolic subgroup $C_W(V)$.
All  elements  of $W_1w\phi$  have the  same $\zeta$-eigenspace $V$, so
$N_W(W_1w\phi)$   normalizes  $V$;  conversely,   an  element  of  $N_W(V)$
normalizes  $W_1$ and conjugates  $w\phi$ to an  element $w'\phi$ with same
$\zeta$-eigenspace, thus $w$ and $w'$ differ by an element of $W_1$, whence (i).

For  the second  item, $N_W(W_Iw\phi)/W_I$  admits as  a section the set of
$I$-reduced  elements, and  such an  element will  conjugate $w\phi$ to the
element  of the coset  $W_Iw\phi$ which is  $I$-reduced, so will centralize
$w\phi$.
\end{proof}

We call {\em essential rank} of a (complex) reflection coset $W\phi\subset\GL(X)$ the
dimension  of the space generated  by its root lines  (the dimension of $X$
minus  the dimension of  the intersection of  the reflection hyperplanes of
$W$).

We  call {\em $\zeta$-rank} of  an element of $W\phi$  the dimension of its
$\zeta$-eigenspace, and $\zeta$-rank of $W\phi$ the maximal $\zeta$-rank of
its elements.

Let   us  say  that  a   $(d,2k)$-periodic  element  of  $B^+\phi(\cI)$  is
non-extendable   if  it   is  non-extendable   in  the   sense  of  Theorem
\ref{racine}(iv).  Another way to state the non-extendability of a periodic
element  $\bI\xrightarrow{\bw\phi}\bI\in B^+\phi(\cI)$  is to  require that
$|\bI|$  be  no  more  than  the  essential rank  of  the  centralizer  of  a maximal
$\zeta$-eigenspace   of  $W\phi$,  where  $\zeta=e^{2ik\pi/d}$:  indeed  if
$\bI\xrightarrow{\bw\phi}\bI$ is extendable there exists $\bJ$ and $\bv$ as
in  Theorem \ref{racine}(iv) and,  since condition \ref{racine}(iv) implies
Lemma   \ref{equiv   complex}(iii),   the   element  $vw\phi$  has  maximal
$\zeta$-rank,  and  the  centralizer  of  its  $\zeta$-eigenspace  has
essential rank
$|\bJ|<|\bI|$.  Note  that  the  notion of non-extendable $(d,2k)$-periodic
element  makes sense  without specifying  $\cI$, as $\zeta=e^{2ik\pi/d}$ is
determined  by  $k/d$,  and  $\cI$  in  turn  is determined as the class of
parabolic  subgroups  which  are  centralizers  of  $\zeta$-eigenspaces  of
elements of $W\phi$ of maximal $\zeta$-rank.

The correspondence between maximal eigenspaces and non-extendable periodic
elements, as described by Theorems \ref{bonne racine} and \ref{racine},
can be summarized as follows:
\begin{corollary}\label{maximal}
Let  $V'$ be the  $\zeta$-eigenspace of an  element of $W\phi$ of maximal
$\zeta$-rank,  where $\zeta=e^{2i\pi  k/d}$. Then  there is a $W$-conjugate
$V$ of $V'$ and $I\subset S$ such that $C_W(V)=W_I$ and the corresponding 
$I$-reduced $w\phi$ (see Theorem   \ref{bonne   racine}(ii)) lifts to a
non-extendable
$(d,2k)$-periodic element $\bI\xrightarrow{\bw\phi}\bI$.   Conversely,  for  a
$(d,2k)$-periodic  non-extendable $\bI\xrightarrow{\bw\phi}\bI$  the image $w\phi$
in $W\phi$ has maximal $\zeta$-rank.
\end{corollary}

We conjecture that Bessis's theorem \cite[11.21]{bessis1} extends to
\begin{conjecture}\label{roots conjugate}
Two non-extendable $(d,2k)$-periodic elements of  $B^+\phi(\cI)$
are cyclically conjugate. 
\end{conjecture}
Note that because of Lemma \ref{equiv complex} the non-extendability
condition is necessary in the above.

By  \ref{power} a $(d,2k)$-periodic  element is cyclically  conjugate to an
element    which    satisfies    in    addition    $(\bw\phi)^{\lfloor\frac
d{2k}\rfloor}\in  \bW \phi^{\lfloor\frac d{2k}\rfloor}$.  We will call {\em
good}   a  $(d,2k)$-periodic   element  which   satisfies  this  additional
condition.

When $k=1$ we can give  conditions purely in  terms of $W$  for an element to
lift  to a  good $(d,2)$-periodic (resp.\  non-extendable good $(d,2)$-periodic)  
element. 

\begin{lemma}\label{good-zeta-maximal}
Let $W\phi\subset\GL(X_\BR)$ be a finite order real reflection coset such that
$\phi$  preserves  the  chamber  of  the  corresponding hyperplane arrangement
determining  the  Coxeter  system  $(W,S)$.  

Let $w\in W$ and $I\subset S$ and let $\bw\in\bW$ and $\bI\subset\bS$ be
their lifts; let $\cI$ be the conjugacy orbit of $\bI$, then $\bw$ induces a morphism 
$\bI\xrightarrow{\bw\phi}\bI\in B^+(\cI)$ if and only if:
\begin{itemize}
\item[(i)] $\lexp{w\phi}I=I$ and $w$ is $I$-reduced.
\end{itemize}

If $\bw$ satisfies (i), for $d>1$ the element $\bI\xrightarrow{\bw\phi}\bI$
is good $(d,2)$-periodic if  and only if the
following two additional conditions are satisfied.
\begin{itemize}
\item[(ii)] $l((w\phi)^i)=\frac{2i}d l(w_I\inv  w_0)$ for $0\leq 2i \le d$.
\item[(iii)] $(w\phi)^d=\phi^d$.
\end{itemize}
If, moreover,\begin{itemize}
\item[(iv)] $W_Iw\phi$  has $\zeta$-rank $0$ on the subspace spanned by the
root lines of  $W_I$ where $\zeta=e^{2i\pi/d}$,
\end{itemize}
then $\bw$ is non-extendable in the sense of Theorem \ref{racine}(iv).
\end{lemma}
\begin{proof}

By definition $w$ induces a morphism $\bI\xrightarrow{\bw\phi}\bI$
if and only if it satisfies (i).
By definition again if $\bI\xrightarrow{\bw\phi}\bI$ is good $(d,2)$-periodic 
then (ii) and (iii) are satisfied.
Conversely, Lemma \ref{bw good} shows that the morphism induced by the
lift of a $w$ satisfying
(i), (ii), (iii) is good $(d,2)$-periodic. 

Property (iv) means that no element $vw\phi$ with $v\in W_I$ has an eigenvalue
$\zeta$ on the subspace spanned by the root lines of $W_I$ which is exactly
the characterization of Theorem \ref{racine}(iv) of a non-extendable element.
\end{proof}
Note  that $d$ and  $I$ are uniquely  determined by $w\phi$ satisfying (i),
(ii),  (iii) above since  $d$ is the  smallest power of  $w\phi$ which is a
power of $\phi$ and $I$ is determined by $(\bw\phi)^d=\bpi/\bpi_\bI\phi^d$.

\begin{definition}\label{good}
We  say that $w\phi\in W\phi$ is  {\em $d$-good} 
if it satisfies (i), (ii), (iii) in Lemma \ref{good-zeta-maximal}.

We say  $w\phi$ is $d$-good {\em maximal} if it satisfies in addition
(iv) in Lemma \ref{good-zeta-maximal}.
\end{definition}

In  particular, $d$-good  elements lift  to good $(d,2)$-periodic elements,
and  $d$-good maximal elements lift to good non-extendable $(d,2)$-periodic
elements.   In   the   appendix,   we   will   construct  a  non-extendable
$(d,2k)$-periodic  element  for  each  $W\phi$,  each  $d$  and  each  $k$.
Actually,  we will do this only for $k=1$ (by constructing $d$-good maximal
elements of $W\phi$), which is sufficient by
\begin{lemma} 
\begin{enumerate}
\item If $\lambda=\gcd(d,k)$ and we set $d_1=d/\lambda$ and $k_1=k/\lambda$
and    $\bw\phi$    is    $(d_1,2k_1)$-periodic    (resp.\  non-extendable
$(d_1,2k_1)$-periodic)   then   $\bw\phi$   is   $(d,2k)$-periodic   (resp.
non-extendable $(d,2k)$-periodic).
\item  If $k$ is prime to $d$ there exists integers $k'$ and $d'$ such that
$dk'=1+kd'$  such  that  if  $\bw_1\phi^{-d'}$  is  $(d,2)$-periodic (resp.
non-extendable  $(d,2)$-periodic)  then  the  element  $\bw\phi$ defined by
$(\bw_1\phi^{-d'})^k=(\bw\phi)\phi^{-k'd}$   is   $(d,2k)$-periodic  (resp.
non-extendable $(d,2k)$-periodic).
\end{enumerate}
\end{lemma}
\begin{proof}  (i) is part of  what is proved in  the last paragraph of the
proof  of \ref{racine} and (ii) is Lemma \ref{w and w1 extendable}.
\end{proof}

Any element of $W\phi$ with a maximal $\zeta$-eigenspace is conjugate to an
element  of $C_W(V)w\phi$ since the  maximal eigenspaces are conjugate, see
\cite[Theorem  3.4(iii) and Theorem 6.2(iii)]{Springer}.  If $w\phi$ is the
image of a non-extendable $(d,2k)$-periodic element, where
$\zeta=e^{2ik\pi/d}$,   it  is   $1$-regular  in   this  coset  by  Theorem
\ref{racine}  (ii)  which  implies  that  it  preserves  a  chamber  of  the
corresponding  real  arrangement  (see  remarks  above  Theorem  \ref{bonne
racine}).  The  following  lemma  shows  that  the  images  in  $W\phi$  of
non-extendable  $(d,2k)$-periodic  elements  (thus  in  particular $d$-good
maximal   elements)  belong  to   a  single  conjugacy   class  under  $W$,
characterized by the above property.

\begin{lemma}\label{zeta maximal are conjugate}
Let  $W\phi$ be a finite order real reflection coset. The
elements  of  $W\phi$  which  have  a  $\zeta$-eigenspace  $V$  of  maximal
dimension  and among those, have the largest dimension of fixed points, are
conjugate.
\end{lemma}
\begin{proof}
As remarked above, up to $W$-conjugacy we may fix a $\zeta$-eigenspace $V$ and
consider only elements of the coset $C_W(V)w\phi$ where $w\phi$ is some element
with $\zeta$-eigenspace $V$; then $W$-conjugacy is reduced to
$C_W(V)$-conjugacy.
Since $C_W(V)$ is a parabolic subgroup of the Coxeter group
$W$  and is normalized by  $w\phi$, 
the  coset $C_W(V)w\phi$  is a real reflection
coset;  in this coset there are $1$-regular elements, which are those which
preserve  a chamber of  the corresponding real  hyperplane arrangement; the
$1$-regular  elements  have  maximal  $1$-rank,  that  is  have the largest
dimension  of fixed  points, and  they form  a single  $C_W(V)$-orbit under
conjugacy, whence the lemma.
\end{proof}

\begin{lemma}\label{regular in parabolic}
Let  $w\phi$ be the image in  $W\phi$ of a non-extendable $(d,2k)$-periodic
element  $\bI\xrightarrow{\bw\phi}\bI$, let $I$  be the image  of $\bI$ and
let  $V_1$ be the fixed  point subspace of $w\phi$  in the space spanned by
the   root  lines  of   $W_I$;  then  $w\phi$   is  regular  in  the  coset
$C_W(V_1)w\phi$.
\end{lemma}
\begin{proof}Let $W'=C_W(V_1)$; we first note that since $w\phi$
normalizes $V_1$ it normalizes also $W'$, so $W'w\phi$ is indeed a reflection
coset. We have thus only to prove that $C_{W'}(V)$
is trivial, where $V$ is the $e^{2ik\pi/d}$-eigenspace of $w\phi$. This last group is
generated by the reflections with
respect  to roots both  orthogonal to $V$  and to $V_1$,  which are the
roots of $W_I=C_W(V)$ orthogonal to $V_1$.
Since $w\phi$ preserves a chamber of
$W_I$, the sum $v$ of the positive roots of $W_I$ with respect to the order 
defined by this chamber is in $V_1$  and is in the chamber: this is well known
for a true root system; here we have taken all the roots to be of length
1 but the usual proof (see \cite[Chapitre VI \S1, Proposition 29]{bbk})
is still valid. Since no root is orthogonal to a vector $v$ inside a chamber,
$W_I$ has no root orthogonal to $V_1$, hence $C_{W'}(V)=\{1\}$.
\end{proof}
One could hope that the above lemma reduces the classification of
$d$-good maximal elements to that of regular elements; however the
map $C_{W'}(w\phi)=N_{W'}(V)\to  N_W(V)/C_W(V)$ with the notations of the
above proof  is injective, but  not always surjective: for example, 
if $W$ of type $E_7$, and
$\phi=\Id$  and $d=4$, then $N_W(V)/C_W(V)$ is
the  complex  reflection  group  $G_8$,  while  $W'$  is  of type $D_4$ and
$N_{W'}(V)/C_{W'}(V)$ is the complex reflection group $G(4,2,2)$. However,
there are only 3 such cases for irreducible groups $W$; the group  
$N_W(V)/C_W(V)$ was determined in appendix 1 in all other cases by the
equality $C_{W'}(w\phi)\simeq  N_W(V)/C_W(V)$, which is proved
by checking that $C_{W'}(w\phi)$ and $N_W(V)/C_W(V)$  have the same reflection 
degrees, a 
simple arithmetic check on the reflection degrees of $W$ and $W'$; indeed,
recall that when $V$ is a maximal $\zeta$-eigenspace, the group
$N_W(V)/C_W(V)$ is a complex reflection group acting on $V$, with
reflection degrees the reflection degrees of $W$ satisfying the arithmetic
condition given for instance in \cite[5.6]{Br}
(when $\phi=\Id$, the reflection degrees divisible by $d$).

\section{Conjectures}\label{section 10}
The following conjectures extend those of \cite[\S 2]{endo}.
They are a geometric form of Brou\'e conjectures.
\begin{conjectures}\label{conjecture}
Let $\bI\xrightarrow{\bw\phi}\bI\in B^+(\cI)\phi$ be 
non-extendable $(d,2k)$-periodic. Then
\begin{enumerate}
\item
The  group $B_\bw$ generated by the  monoid $B^+_\bw$ of Theorem \ref{desc endo} is
isomorphic   to   the   braid   group   of  the  complex  reflection  group
$W(w\phi):=N_W(W_Iw\phi)/W_I$.
\item
The natural morphism $\DI(\bI\xrightarrow{\bw\phi}\bI)
\to\End_{\bG^F}(\bX(\bI,\bw\phi))$ (see below Lemma \ref{D tilde to D})
gives rise to a morphism
$B_\bw\to\End_{\bG^F}H^*_c(\bX(\bI,\bw\phi))$ which factors through a 
special representation of a $\zeta$-cyclotomic Hecke algebra $\cH_\bw$
for $W(w\phi)$, where $\zeta=e^{2ik\pi/d}$.
\item
The odd and even $H^i_c(\bX(\bI,\bw\phi))$ are disjoint $\bG^F$-modules, and the above morphism
extends to a surjective morphism $\Qlbar[B_\bw]\to
\End_{\bG^F}(H^*_c(\bX(\bI,\bw\phi)))$.
\end{enumerate}
\end{conjectures}
The group $W(w\phi)$ above is a complex reflection group by the
remarks at the end of last section and Lemma \ref{N(V)=N(WIwphi)} (i).

The  condition that the periodic elements we consider are non-extendable is
necessary  for  assertion  (ii)  above  to  hold; in the case of extendable
periodic  elements  the  endomorphism  algebra  should,  instead of being a
deformation  of the  group algebra  of $W(w\phi)$,  be a  deformation of an
endomorphism algebra of an induced representation from a complex reflection
group   to  another.   Whenever  a   periodic  element   is  extendable,  a
decomposition as in Theorem \ref{produit fibre general} can be applied. See
\cite[1.3]{dudas} for such computations.

David Craven has made (iii) above more specific by giving a conjectural
formula  computing  the  cohomology  degree  in  which  a  given  unipotent
character  should  occur  (see  \cite{Craven});  Craven's formula should be
valid  for any $(d,2k)$-periodic element, not only the non-extendable ones.
In  the  current  paper  we  focus  on the study of non-extendable periodic
elements;  this should  be a  start for  the general  study of all periodic
elements.

\begin{lemma}\label{conjecture+}
Let $\bI\xrightarrow{\bw\phi}\bI\in B^+\phi(\cI)$ be non-extendable
$(d,2k)$-periodic and assume Conjectures \ref{conjecture}; then
for any $i\ne j$ the $\bG^F$-modules
$H^i_c(\bX(\bI,\bw\phi))$ and $H^j_c(\bX(\bI,\bw\phi))$ are disjoint.
\end{lemma}
\begin{proof}
Since the image of the morphism of Conjecture \ref{conjecture}(ii) consists of 
equivalences of \'etale sites, it follows that the action of $\cH_\bw$ on 
$H^*_c(\bX(\bI,\bw\phi))$ preserves individual cohomology groups.
The surjectivity of the morphism of (iii) implies that for $\rho\in\Irr(\bG^F)$,
the $\rho$-isotypic part of $H^*_c(\bX(\bI,\bw\phi))$ affords an irreducible
$\cH_\bw$-module; this would not be possible if this $\rho$-isotypic part was
spread over several distinct cohomology groups.
\end{proof}

We will now explore the information given by the Shintani descent identity
on the above conjectures
\begin{lemma}\label{chi(X1TwF)}
Let $\bI\xrightarrow{\bw\phi}\bI\in B^+\phi(\cI)$ be $(d,2k)$-periodic
With   the   notations   of  Proposition  \ref{shintani1},  
we have
$\tilde\chi_{q^m}(X_1 T_\bw\phi)=q^{m\frac
kd(l(\bpi/\bpi_\bI)-a_\chi-A_\chi)} \tilde\chi(e_I w F)$
for $\chi\in\Irr(W)^\phi$,
where $a_\chi$ (resp.\ $A_\chi$) is the valuation
(resp.\ the degree) of the generic degree of $\chi$ and 
$e_I=|W_I|\inv\sum_{v\in W_I}v$.
\end{lemma}
\begin{proof}
We have
$(X_1T_\bw\phi)^d=X_1(T_\bpi/T_{\bpi_\bI})^k\phi^d=
q^{-kl(\bpi_\bI)}X_1T_\bpi^k \phi^d$
since $X_1$ commutes with $T_\bw\phi$ and
since for any $v\in W_I$ we have $X_1 T_v=q^{l(v)}X_1$.
Since $T_\bpi$ acts on the representation of character $\chi_{q^m}$ as the 
scalar $q^{m(l(\bpi)-a_\chi-A_\chi)}$ (see \cite[Corollary 4.20]{BM}),
it follows that all the non-zero eigenvalues
of $X_1T_\bw\phi$ on this representation are equal to 
$q^{m\frac kd(l(\bpi/\bpi_\bI)-a_\chi-A_\chi)}$ times a root of unity.
To compute the sum of these roots of unity, we may use the specialization
$q^{m/2}\mapsto 1$, through which $\tilde\chi_{q^m}(X_1 T_\bw\phi)$ specializes
to $\tilde\chi(e_I w \phi)$.
\end{proof}
\begin{proposition}\label{shintani3}
Let $\bI\xrightarrow{\bw\phi}\bI\in B^+\phi(\cI)$ be $(d,2k)$-periodic.
For any $m$ multiple of $\delta$, we have $$
|\bX(\bI,\bw\phi)^{gF^m}|=\sum_{\rho\in\cE(\bG^F,1)}\lambda_\rho^{m/\delta}
q^{m\frac kd(l(\bpi/\bpi_\bI)-a_\rho-A_\rho)}\langle\rho,R^{\bG,F}_{\bL_I,
t(\bw\phi)}\Id \rangle_{\bG^F}\rho(g),$$
where $a_\rho$  and $A_\rho$ are respectively the valuation
and the degree of the generic degree of $\rho$.
\end{proposition}
\begin{proof}We start with Corollary \ref{shintani2}, whose statement reads,
using the value of $\tilde\chi_{q^m}(X_1 T_\bw\phi)$ given by
Lemma \ref{chi(X1TwF)}:
\begin{multline*}
|\bX(\bI,\bw\phi)^{gF^m}|=
\sum_{\rho\in\cE(\bG^F,1)}\lambda_\rho^{m/\delta}\rho(g)\hfill\\
\hfill\sum_{\chi\in\Irr(W)^\phi}
q^{m\frac kd(l(\bpi/\bpi_\bI)-a_\chi-A_\chi)}\tilde\chi(e_I w \phi)
\langle \rho,R_{\tilde\chi}\rangle_{\bG^F}.
\end{multline*}
Using that for any $\rho$ such that 
$\langle \rho,R_{\tilde\chi}\rangle_{\bG^F}\ne 0$ we have $a_\rho=a_\chi$ and
$A_\rho=A_\chi$ (see \cite{BM} around (2.4))
the right-hand side can be rewritten
$$
\sum_{\rho\in\cE(\bG^F,1)}\lambda_\rho^{m/\delta}
q^{m\frac kd(l(\bpi/\bpi_\bI)-a_\rho-A_\rho)}\rho(g)
\langle \rho, \sum_{\chi\in\Irr(W)^\phi}
\tilde\chi(e_I w \phi)R_{\tilde\chi}\rangle_{\bG^F}.
$$

The proposition is now just a matter of observing that
\begin{multline*}
\sum_{\chi\in\Irr(W)^\phi}\tilde\chi(e_I w \phi)R_{\tilde\chi}=
|W_I|\inv\sum_{v\in W_I}\sum_{\chi\in\Irr(W)^\phi}\tilde\chi(v w \phi)R_{\tilde\chi}=
\hfill\\
\hfill |W_I|\inv \sum_{v\in W_I}R_{\bT_{vw}}^\bG(\Id)=
R_{\bL_I,t(\bw\phi)}^{\bG,F}(\Id).
\end{multline*}
Where the last equality is obtained by transitivity of $R_\bL^\bG$ and the
equality $\Id_\LwF=|W_I|\inv\sum_{v\in W_I}R_{\bT_{vw}}^{\bL_I,
t(\bw\phi)}(\Id)$, a torus $\bT$ of $\bL_I$ of type $v$ for the isogeny $t(\bw\phi)$ being
conjugate to $\bT_{vw}$ in $\bG$.
\end{proof}
\begin{corollary}\label{eigenvalue of F}
Let $\bI\xrightarrow{\bw\phi}\bI\in B^+\phi(\cI)$ be  non-extendable
$(d,2k)$-periodic and assume Conjectures \ref{conjecture}; then
for any $\rho\in\Irr(\bG^F)$ such that
$\langle \rho,R_{\bL_I,t(\bw\phi)}^{\bG,F}(\Id)\rangle_{\bG^F}\ne 0$ the isogeny $F^\delta$
has a single eigenvalue on the $\rho$-isotypic part of 
$H^*_c(\bX(\bI,\bw\phi))$, equal to  $\lambda_\rho
q^{\delta\frac kd(l(\bpi/\bpi_\bI)-a_\rho-A_\rho)}$.
\end{corollary}
\begin{proof}
This follows immediately, in view of Lemma \ref{conjecture+}, from the comparison
between Proposition \ref{shintani3} and the Lefschetz formula:
$$|\bX(\bI,\bw\phi)^{gF^m}|=\sum_i (-1)^i\Trace(gF^m\mid
H^i_c(\bX(\bI,\bw\phi),\Qlbar)).$$
\end{proof}

In view of Corollary \ref{omega_rho}(i) it follows that if 
$\langle \rho,R_{\bL_I}^\bG(\Id)\rangle_{\bG^F}\ne 0$ then
if $\omega_\rho=1$ then $\frac kd(l(\bpi/\bpi_\bI)-a_\rho-A_\rho)\in\BN$,
and if $\omega_\rho=\sqrt{q^\delta}$ then
$\frac kd(l(\bpi/\bpi_\bI)-a_\rho-A_\rho)\in\BN+1/2$.

Assuming Conjectures \ref{conjecture}, we choose once and for all a
specialization $q^{1/a}\mapsto\zeta^{1/a}$, where $a\in\BN$ is large enough
such that $\cH_\bw\otimes\Qlbar[q^{1/a}]$ is split. This gives a bijection
$\varphi\mapsto\varphi_q:\Irr(W(w\phi))\to\Irr(\cH_\bw)$, and the conjectures give
a further bijection $\varphi\mapsto\rho_\varphi$ between $\Irr(W(w\phi))$ and the
set $\{\rho\in\Irr(\bG^F)\mid
\langle \rho,R_{\bL_I}^\bG(\Id)\rangle_{\bG^F}\ne 0\}$, which is such that
$\langle \rho_\varphi,R_{\bL_I}^\bG(\Id)\rangle_{\bG^F}=\varphi(1)$.

\begin{corollary}
Under the assumptions of Corollary \ref{eigenvalue of F},
if $\omega_\varphi$ is the central character of $\varphi$, then
$$\lambda_{\rho_\varphi}=\omega_\varphi((w\phi)^\delta)\zeta^{-\delta
\frac kd(l(\bpi/\bpi_\bI)-a_{\rho_\varphi}-A_{\rho_\varphi})}.$$
\end{corollary}
\begin{proof}
We first note that it makes sense to apply $\omega_\varphi$ to
$(w\phi)^\delta$, since $(w\phi)^\delta$ is a central element of  $W(w\phi)$.
Actually $(\bw\phi)^\delta$ is a central element of $B_\bw$ and maps by the 
morphism of Conjecture \ref{conjecture}(iii) to $F^\delta$, thus the eigenvalue of
$F^\delta$ on the $\rho_\varphi$-isotypic part of
$H^*_c(\bX(\bI,\bw\phi))$ is equal to $\omega_{\varphi_q}((\bw\phi)^\delta)$;
thus $\omega_{\varphi_q}((\bw\phi)^\delta)=\lambda_{\rho_{\varphi}}
q^{\delta\frac kd(l(\bpi/\bpi_\bI)-a_{\rho_\varphi}-A_{\rho_\varphi})}$.
The statement follows by applying the specialization
$q^{1/a}\mapsto\zeta^{1/a}$ to this equality.
\end{proof}
\section{Appendix: $d$-good maximal elements in finite Coxeter cosets.}

We  will describe,  in a  finite Coxeter  coset, for  each $d$,  a $d$-good
maximal element.

As  explained the introduction of Section \ref{eigenspaces and roots}, when
the  Coxeter coset is attached to a  reductive group $\bG$, such an element
defines  a parabolic Deligne-Lusztig  variety whose cohomology  should be a
tilting  complex  for  the  Brou\'e  conjectures  for  an  $\ell$  dividing
$\Phi_d(q)$.  The properties of  this variety do  not depend on the isogeny
type, thus it is sufficient to study the case when $\bG$ is semi-simple and
simply connected. Now, a semi-simple and simply connected group is a direct
product of restrictions of scalars of simply connected quasi-simple groups.
A restriction of scalars is a group of the form $\bG^n$, with an isogeny
$F_1$  such that  $F_1(x_0,\dots,x_{n-1})=(x_1,\dots,x_{n-1},F(x_0))$. Then
$(\bG^n)^{F_1}\simeq\bG^F$.  If $F$ induces $\phi$ on the Weyl group $W$ of
$G$ then $(\bG^n,F_1)$ corresponds to the reflection coset
$W^n\cdot\sigma$, where 
$\sigma(x_0,\dots,x_{n-1})=(x_1,\dots,x_{n-1},\phi(x_0))$.

\subsection{Restrictions of scalars}

Restrictions  of scalars as above appear in the classification of arbitrary
complex  reflection cosets. Arbitrary cosets $W\phi$ are direct products of
cosets  where $\phi$ is transitive on the irreducible components of $W$; we
call {\em restriction of scalars} a complex reflection coset with this last
property. It is of the form $W^n\cdot\sigma\subset\GL(V^n)$, where $V$ is a
complex vector space and $W\phi\subset\GL(V)$ is a complex reflection coset
and where $\sigma(x_0,\dots,x_{n-1})=(x_1,\dots,x_{n-1},\phi(x_0))$. We say
that  $W^n\sigma$ is a {\em restriction  of scalars} of $W\phi$, by analogy
with the terminology for reductive groups.

We  first look at the  invariant theory of a restriction of scalars. Recall (see
for  example \cite{Br})  that, if  $S_W$ is  the coinvariant algebra of $W$
(the  quotient of the symmetric algebra of  $V^*$ by the ideal generated by
the  $W$-invariants of positive degree), for  any $W$-module $X$ the graded
vector  space  $(S_W\otimes  X^*)^W$  admits  a homogeneous basis formed of
eigenvectors  of  $\phi$.  The  degrees  of  the elements of this basis are
called the $X$-exponents of $W$ and the corresponding eigenvalues of $\phi$
the  $X$-factors  of  $W\phi$.  For  $X=V$, the $V$-exponents $n_i$ satisfy
$n_i=d_i-1$  where the $d_i$'s  are the reflection  degrees of $W$, and the
$V$-factors  $\varepsilon_i$ are called  the {\em factors}  of $W\phi$. For
$X=V^*$,  the $n_i-1$  where $n_i$  are the  $V^*$-exponents are called the
{\em   codegrees}  $d_i^*$  of  $W$  and  the  corresponding  $V^*$-factors
$\varepsilon_i^*$  are called the  {\em cofactors} of  $W\phi$. By Springer
\cite[6.4]{Springer},  for  a  root  of  unity $\zeta$, the $\zeta$-rank of
$W\phi$ is equal to $|\{i\mid \zeta^{d_i}=\varepsilon_i\}|$. By analogy, we
define the $\zeta$-corank of $W\phi$ as $|\{i\mid
\zeta^{d^*_i}=\varepsilon^*_i\}|$.  By  for  example  \cite[5.19.2]{Br}  an
eigenvalue is regular if it has same rank and corank.

\begin{proposition}\label{reg}
Let $W^n\cdot\sigma$ be a restriction of scalars of the complex reflection
coset $W\phi$. Then
the $\zeta$-rank (resp.\ corank) of $W^n\cdot\sigma$ is equal to the 
$\zeta^n$-rank (resp.\ corank) of $W\phi$.

In particular, $\zeta$ is regular for $W^n\cdot\sigma$ if and
only if $\zeta^n$ is  regular for $W\cdot\phi$.
\end{proposition}
\begin{proof}
It  is easy  to see  from the  construction that  the pairs of a reflection
degree   and   the   corresponding   factor   of  $\sigma$  for  the  coset
$W^n\cdot\sigma$ are the pairs $(d_i,\eta_{i,j})$, where
$i\in\{1,\dots,r\}$  and  where  $\{\eta_{i,j}\}_{j\in\{1\ldots  n\}}$ run
over  the  $n$-th  roots  of  $\varepsilon_i$.  Similarly,  the  pairs of a
reflection codegree and the corresponding cofactor are
$(d_i^*,\eta^*_{i,j})$  where  $\{\eta^*_{i,j}\}_{j\in\{1\ldots  n\}}$ run
over the $n$-th roots of $\varepsilon^*_i$.

In   particular  the  $\zeta$-rank  of  $W^n\cdot\sigma$  is  $|\{(i,j)\mid
\zeta^{d_i}=\eta_{i,j}\}|$   and   the   $\zeta$-corank   is  $|\{(i,j)\mid
\zeta^{d_i^*}=\eta^*_{i,j}\}|$.

Given  $d$, there  is at  most one  $j$ such that $\zeta^d=\eta_{i,j}$, and
there is one if and only if $\zeta^{nd}=\varepsilon_i$. Thus 
$|\{(i,j)\mid \zeta^{d_i}=\eta_{i,j}\}|=|\{i\mid \zeta^{nd_i}=\varepsilon_i\}|$
and similarly for the corank, whence the two assertions of the statement.
\end{proof}

The  next lemma can also be used to give a direct proof of the statement on
$\zeta$-ranks.
\begin{lemma}\label{descent w}
Let $W^n\cdot\sigma$ be a restriction of scalars of $W\phi$. Then
\begin{enumerate}
\item[(i)]
Any element of $W^n\sigma$ is conjugate
to an element of the form $(1,\ldots,1,w)\sigma$.
\item[(ii)] The vector $(x_0,\ldots,x_{n-1})\in V^n$ is a
$\zeta$-eigenvector  of $(1,\ldots,1,w)\sigma$  if and  only if  $x_0$ is a
$\zeta^n$-eigenvector of $w\phi$ and $x_i=\zeta^ix_0$ for $i=1,\dots,n-1$.
\end{enumerate}
\end{lemma}
\begin{proof}
The element $(1,w_0,w_0w_1,\ldots,w_0w_1\dots w_{n-2})$ conjugates
$(w_0,\ldots,w_{n-1})\sigma$ to $(1,\ldots,1,w_0\dots w_{n-1})\sigma$, whence
(i). Property (ii) results from an immediate computation.
\end{proof}
In view of Lemma \ref{reg}, the following proposition is enough to determine
all possible non-extendable $(d,2k)$-periodic elements of $W^n\sigma$.
\begin{proposition}
Let $W^n\cdot\sigma$ be a restriction of scalars of the finite Coxeter coset
$W\phi$. Let $(B^+)^n\sigma$ and $B^+\phi$ be the corresponding cosets
of braid monoids. Then
\begin{enumerate}
\item[(i)]
Any element in $(B^+)^n\sigma$
is conjugate under $(B^+)^n$ to an element of the form $(1,\ldots,1,\bw)\sigma$.
\item[(ii)]
The element $(1,\ldots,1,\bw)\sigma\in (B^+)^n\sigma$ is $(nd,2k)$-periodic 
if and only if $\bw\phi$ is $(d,2k)$ periodic. Moreover the latter is  non-extendable
if and only if the former is non-extendable.
\end{enumerate}
\end{proposition}
\begin{proof}
The element $(1,\bw_0,\bw_0\bw_1,\ldots,\bw_0\bw_1\dots \bw_{n-2})$ conjugates
$(\bw_0,\ldots,\bw_{n-1})\sigma$ to $(1,\ldots,1,\bw_0\dots \bw_{n-1})\sigma$, whence
(i).

For  (ii),  we  have  $((1,\ldots,1,\bw)\sigma)^{nd}=((\bw\phi)^d\phi^{-d},
\ldots,(\bw\phi)^d\phi^{-d})\sigma^{nd}$,   whence  the   first  assertion:
$(\bw\phi)^d=(\bpi/\bpi_\bI)^k\phi^d$ is equivalent to
$((1,\ldots,1,\bw)\sigma)^{nd}=
((\bpi/\bpi_\bI)^k,\ldots,(\bpi/\bpi_\bI)^k)\sigma^{nd}$.

If  the last  equalities hold  and $\bw\phi$  is extendable,  that is there
exist $\bv\in B^+_\bI$ and $\bJ\subsetneq\bI$ such that
$(\bv\bw)^d=(\bpi/\bpi_\bJ)^k\phi^d$, then
$((1,\ldots,1,\bv)(1,\ldots,1,\bw)\sigma)^{nd}=
((\bpi/\bpi_\bJ)^k,\ldots,(\bpi/\bpi_\bJ)^k)\sigma^{nd}$,      so      that
$(1,\ldots,1,\bw)\sigma$ is extendable.

Conversely assume that
$(1,\ldots,1,\bw)\sigma$ is extendable, that is, there exist
$(\bv_0\ldots,\bv_{n-1})\in (B^+_\bI)^n$ and
$\bJ_0\times\cdots\times\bJ_{n-1}\subsetneq\bI^n$ such that
$$((\bv_0,\ldots,\bv_{n-2},\bv_{n-1}\bw)\sigma)^{nd}=(\bpi/\bpi_{\bJ_0},\ldots,
\bpi/\bpi_{\bJ_{n-1}})^k\sigma^{nd}.$$
Then since $(\bv_0,\ldots,\bv_{n-2},\bv_{n-1}\bw)\sigma$ stabilizes
$\bJ_0\times\cdots\times\bJ_{n-1}$, we have  $\bJ_i=\lexp{\bv_i}\bJ_{i+1}$
for $i<n-1$ so that $\bJ_i\subsetneq\bI$ for all $i$.
By the same conjugation as in the first line of the proof (by 
$(1,\bv_0,\bv_0\bv_1,\ldots,\bv_0\bv_1\dotsm\bv_{n-2})$) the above equality
conjugates to
$((1,\ldots,\bv_0\dotsm\bv_{n-1}\bw)\sigma)^{nd}=
(\bpi/\bpi_{\bJ_0},\ldots,\bpi/\bpi_{\bJ_0})^k\sigma^{nd}$, or equivalently
$(\bv_0\dotsm\bv_{n-1}\bw\phi)^d=(\bpi/\bpi_{\bJ_0})^k\phi^d$, thus
$\bw\phi$ is extendable.
\end{proof}

\subsection{Case of irreducible Coxeter cosets}
\label{irreducible cosets}
We  are going to give, for each  irreducible finite Coxeter group $W$, each
possible  corresponding coset $W\phi$  where $\phi$ preserves  a chamber of
the   corresponding  hyperplane  arrangement,  and  each  possible  $d$,  a
representative   $w\phi$   of   the   $d$-good   maximal   elements.  Since
conjecturally all non-extendable $(d,2)$-periodic elements are conjugate in
the  ribbon category  (see Conjecture  \ref{roots conjugate}),  this should
describe also these elements.

We   also  describe  the  corresponding   $\zeta_d$-eigenspace  $V$  where
$\zeta_d=e^{2i\pi/d}$,  the  set  $I$  and  the relative complex reflection
group   $W(w\phi):=N_W(V)/C_W(V)$.  In   the  cases   where  the  injection
$C_{W'}(w\phi)\to   N_W(V)/C_W(V)=W(w\phi)$  of  the   remark  after  Lemma
\ref{regular in parabolic}, is surjective, where $W'=C_W(V_1)$ and $V_1$ is
the  fixed point subspace of $w\phi$ in the space spanned by the root lines
of  $W_I$,  we  use  it  to  deduce $W(w\phi)$ from $W'=C_W(V_1)$ using
the description of centralizers of regular elements in \cite[Annexe 1]{BM}.

\subsection*{Types $A_n$ and $\protect\lexp 2A_n$
$\nnode{s_1}\edge\nnode{s_2}\cdots\nnode{s_n}$}\label{An}
$\lexp 2A_n$ is defined by
the diagram automorphism $\phi$ which exchanges $s_i$ and $s_{n+1-i}$.

For any integer $1<d\leq n+1$, we define $$v_d=s_1s_2\dotsm
s_{n-\lfloor\frac d2\rfloor} s_ns_{n-1}\dotsm s_{\lfloor
\frac{d+1}2\rfloor}\text{ and }
J_d=\{s_i\mid\lfloor\frac{d+1}2\rfloor+1\leq    i\leq    n    -\lfloor\frac
d2\rfloor\}.$$  If  $d$  is  odd  we  have  $v_d=v'_d\lexp\phi v'_d$, where
$v'_d=s_1s_2\dotsm s_{n-\lfloor\frac d2\rfloor}$.

Now,  for $1<d\leq  n+1$, let $kd$  be the largest  multiple of $d$ less
than   or   equal   to   $n+1$,   so   that  $\frac{n+1}2<kd\leq  n+1$  and
$k=\lfloor\frac{n+1}d\rfloor$.
We  then define  $w_d=v_{kd}^k$, $I_d=J_{kd}$  and if  $d$ is odd we define
$w'_d$  by  $$w'_d\phi=\begin{cases}(v'_{kd}\phi)^k&\text{  if  }k\text{ is
odd,}\\  v_{kd}^{k/2}\phi&\text{ if  }k\text{ is  even,}\end{cases}$$

\begin{theorem}\label{type A} 
For $W=W(A_n)$, $d$-good maximal elements exist for
$1<d\leq n+1$; a representative is $w_d$, with  $I=I_d$ and
$W(w_d)=G(d,1,\lfloor\frac{n+1}d\rfloor)$.

For $W\phi$, $d$-good maximal elements exist for the following
$d$ with representatives as follows:
\begin{itemize}
\item $d\equiv 0\pmod 4$, $1<d\le n+1$; a representative is
$w_d\phi$ with $I=I_d$
and $W(w_d\phi)=G(d,1,\lfloor\frac{n+1}d\rfloor)$.
\item $d\equiv 2\pmod 4$, $1<d\leq 2(n+1)$; a representative is
$w'_{d/2}\phi$ with $I=I_{d/2}$
and $W(w'_{d/2}\phi)=G(d/2,1,\lfloor\frac{2(n+1)}d\rfloor)$.
\item $d$ odd, $1<d\leq \frac{n+1}2$. If $d\ne 1$ a representative is
$w_{2d}^2\phi$  with $I=I_{2d}$
and $W(w_{2d}^2\phi)=G(2d,1,\lfloor\frac{n+1}{2d}\rfloor)$.
\end{itemize}
\end{theorem}
\begin{proof}
We identify the Weyl group of type
$A_n$ as usual with $\Sgot_{n+1}$ by $s_i\mapsto (i,i+1)$;
the automorphism $\phi$ maps to the exchange of $i$ and $n+2-i$.
An easy computation shows that the element $v_d$ maps to the $d$-cycle
$(1,2,\dots,\lfloor\frac{d+1}2\rfloor,n+1,n,
\dots,n+2-\lfloor\frac d2\rfloor)$
and that for  $d$  odd $v'_d$ maps to the cycle
$(1,2,\dots,n-\frac{d-3}2)$.
\begin{lemma}
If $d$ is even $v_d$ and $w_d$ are $\phi$-stable.
If $d$ is odd we have $w_d=w'_d.\lexp\phi w'_d$.
\end{lemma}
\begin{proof} 
That    $d$   is   even   implies   $\lfloor\frac{d+1}2\rfloor=\lfloor\frac
d2\rfloor$,  thus in  the above  cycle $\phi$  exchanges the  two sequences
$1,2,\dots,\lfloor\frac{d+1}2\rfloor$    and   $n+1,n,\dots,n+2-\lfloor\frac
d2\rfloor$,  thus $v_d$ is $\phi$-stable. The same follows for $w_d$, with $k=\lfloor\frac{n+1}d\rfloor$,
since $kd$ is even if $d$ is even.

For $d$ odd we have
$$w'_d.\lexp\phi w'_d=(w'_d\phi)^2= \begin{cases}(v'_{kd}\phi)^{2k}& \text{
if   }k\text{  is  odd,}\\  v_{kd}^{k/2}.\lexp\phi(v_{kd}^{k/2})&\text{  if
}k\text{ is even.}\end{cases}$$ If $k$ is odd we have
$(v'_{kd}\phi)^{2k}=(v'_{kd}\lexp\phi  v'_{kd})^k= v_{kd}^k=w_d$. If $k$
is even then $v_{kd}$ is $\phi$-stable thus
$v_{kd}^{k/2}.\lexp\phi(v_{kd}^{k/2})=v_{kd}^k=w_d$.
\end{proof}
\begin{lemma}   For $1<d\le n+1$,
\begin{itemize}
\item the element $v_d$ is $J_d$-reduced and  stabilizes $J_d$.
\item the element $w_d$ is $I_d$-reduced and  stabilizes $I_d$.
\item for $d$ odd, the element
$v'_d$ is $J_d$-reduced and  $v'_d\phi$ stabilizes $J_d$.
\item for $d$ odd, the element
$w'_d$ is $I_d$-reduced and  $w'_d\phi$ stabilizes $I_d$.
\end{itemize}
\end{lemma}
\begin{proof}
The  property for $w_d$  (resp.\ $w'_d$) follows  from that for $v_d$ (resp.
$v'_d$) and the definitions since being $I_d$-reduced and stabilizing $I_d$
are properties stable by taking a power.

It  is clear on  the expression of  $v_d$ as a  cycle that it fixes $i$ and
$i+1$  if $s_i\in  J_d$ thus  it fixes  the simple  roots corresponding to
$J_d$, whence the lemma for $v_d$.

For  $d$  odd,  $1< d\leq  n+1$,  an easy computation shows that
$v'_d=(1,2,\dots,n-\frac{d-3}2)$,  and that  $v'_d\phi$ preserves  the simple
roots  corresponding to  $J_d$. 
\end{proof}
\begin{lemma}\label{lengths add} For $1<d\le n+1$ and
for $0\leq i \le \lfloor\frac d2\rfloor$, we have
\begin{itemize}
\item  $l(v_d^i)=\frac{2i}d l(w_{J_d}\inv w_0)$ and
$l(w_d^i)=\frac{2i}d l(w_{I_d}\inv w_0)$
\item (for $d$ odd) $l((v'_d\phi)^i\phi^{-i})=\frac id l(w_{J_d}\inv w_0)$
and $l((w'_d\phi)^i\phi^{-i})=\frac id l(w_{I_d}\inv w_0)$.
\end{itemize}
\end{lemma}
\begin{proof}
It is straightforward to see that the result for $w_d$ (resp.\ $w'_d$) results
from the result for $v_d$ (resp.\ $v'_d$ or $v_d$) and the definitions.

Note  that the  group $W_{J_d}$  is of  type $A_{n-d}$, thus $l(w_{J_d}\inv
w_0)=\frac{n(n+1)}2-\frac{(n-d)(n-d+1)}2=\frac{(2n-d+1)d}2$. 

We  first prove the result for $v_d$ and  $v'_d$ when $i=1$. For odd $d$ we
have  by  definition  $l(v'_d)=n-\frac{d-1}2=  \frac{2n-d+1}2$ which is the
formula  we want for $v'_d$.  To find the length  of $v_d$ one can use that
$s_ns_{n-1}\dotsm s_{\lfloor\frac{d+1}2\rfloor}$ is
$\{s_1,s_2,\dots,s_{n-1}\}$-reduced,     thus    adds     to    $s_1s_2\dotsm
s_{n-\lfloor\frac  d2\rfloor}$, which gives $l(v_d)=2n-d+1$, the result for
$v_d$.

We now show by direct computation that when $d$ is even
$v_d^{d/2}=w_{J_d}\inv   w_0$.   Raising   the   cycle   $(1,2,\dots,\frac
d2,n+1,n,\dots,n+2-\frac d2)$ to the $d/2$-th power we get
$(1,n+1)(2,n)\dotsm(\frac  d2,n+2-\frac  d2)$  which  gives  the  result since
$w_{J_d}=(\frac d2+1,n+1-\frac d2)\dotsm(\lfloor\frac
n2\rfloor,\lfloor\frac{n+1}2\rfloor)$. The lemma follows for $v_d$ with $d$
even since its truth for  $i=1$ and $i=\frac d2$  implies its truth for all
$i$ between these values.

We  show now similarly  that for odd  $d$ we have $(v'_d\phi)^d=w_{J_d}\inv
w_0\phi^d$.  Since $\phi$  acts on  $W$ by  the inner automorphism given by
$w_0$,   this  is  the  same  as   $(v'_d  w_0)^d=w_{J_d}$.  We  find  that
$(1,2,\dots,n-\frac{d-3}2)w_0=
(1,n+1,2,n,3,n-1,\dots,n-\frac{d-5}2,\frac{d+1}2)
(\frac{d+3}2,n-\frac{d-3}2)\dotsm(\lfloor\frac{n+3}2\rfloor,
\lfloor\frac{n+4}2\rfloor)$  as a  product of  disjoint cycles, which gives
the  result since $(1,n+1,2,n,3,n-1,\dots,n-\frac{d-5}2,\frac{d+1}2)$ is a
$d$-cycle   and  $(\frac{d+3}2,n-\frac{d-3}2)\dotsm(\lfloor\frac{n+3}2\rfloor,
\lfloor\frac{n+4}2\rfloor)=w_{J_d}$. This proves the lemma for $w'_d$ by
interpolating the other values of $i$ as above.

It  remains the  case of  $v_d$ for  odd $d$.  We then have $v_d=(v'_d\phi)^2$
where  the lengths add, and we deduce the result for $v_d$ from the result for
$v'_d$.
\end{proof}
\begin{lemma} The following elements are $d$-good
\begin{itemize}
\item For $1<d\le n+1$, the elements $v_d$ and $w_d$.
\item For $d\equiv 0\pmod 4, d\le n+1$ the elements $v_d\phi$ and $w_d\phi$.
\item For $d\equiv 2\pmod 4, d\le 2(n+1)$ the elements $v'_{d/2}\phi$ and
$w'_{d/2}\phi$.
\item For $d$ odd, $d\le \frac{n+1}2$ the elements $v_{2d}^2\phi$ and
$w_{2d}^2\phi$.
\end{itemize}
\end{lemma}
\begin{proof}
In view of the previous lemmas, the only thing left to check is that
in each case, the chosen element $x$ in $W$ (resp.\ $W\phi$) satisfies 
$x^d=1$ (resp.\ $(x\phi)^d=\phi^d$). Once again, it is easy
to check that the property for $w_d$ (resp.\ $w'_d$) results from that for
$v_d$ (resp.\ $v'_d$ or $v_d$) and the definitions.

It is clear that $v_d^d=1$ since then it is a $d$-cycle, from which it follows
that when $d\equiv 2\pmod 4$ we have $(v'_{d/2}\phi)^d=v_{d/2}^{d/2}=1$.
The other cases are obvious.
\end{proof}

To prove the theorem, it remains to check that:

$\bullet$  The  possible  $d$  for  which  the $\zeta_d$-rank of $W$ (resp.\ 
$W\phi$) is non-zero are as described in the theorem. In the untwisted case
they  are the divisors of one of  the degrees, which are $2,\dots,n+1$. In
the twisted case the pairs of degrees and factors are
$(2,1),\dots,(i,(-1)^i),\dots,(n+1,(-1)^{n+1})$ and we get the given list
by the formula for the $\zeta_d$-rank recalled above Proposition \ref{reg}.

$\bullet$ The coset $W_Iw\phi$ has $\zeta_d$-rank 0 on the subspace spanned
by  the root lines of $W_I$. For this we first have to describe the type of
the  coset, which  is a  consequence of  the analysis  we did  to show that
$w\phi$ stabilizes $I$. We may assume $I$ non-empty.

Let us look first at the untwisted case. We found that $w_d$ acts trivially on
$I_d$,  so  the  coset  is  of  untwisted type $A_{n-kd}$ where $k=\lfloor
\frac{n+1}d\rfloor$. Since $1+n-kd<d$  by construction, this coset has
$\zeta_d$-rank 0.

In  the twisted case, if $d\equiv 0\pmod 4$, the coset is $W_{I_d}w_d\phi$,
which  since  $w_d$  acts  trivially  on  $I_d$  and  $\phi$  acts  by  the
non-trivial  diagram  automorphism,  is  of  type  $\lexp  2A_{n-kd}$ where
$k=\lfloor  \frac {n+1}d\rfloor$. Since $n-kd=n-\lfloor \frac {n+1}d\rfloor
d<d-1$, this coset has $\zeta_d$-rank 0.

If  $d$ is odd, the coset is $W_{I_{2d}}w_{2d}^2\phi$, which since $w_{2d}$
acts  trivially  on  $I_{2d}$  and  $\phi$  acts by the non-trivial diagram
automorphism,   is  of  type  $\lexp  2A_{n-2kd}$  where  $k=\lfloor  \frac
{n+1}{2d}\rfloor$.  Since  $n-2kd=n-\lfloor  \frac {n+1}{2d}\rfloor 2d<2d$,
this coset has $\zeta_d$-rank 0.

Finally, if $d\equiv 2$ modulo 4, the coset is $W_{I_{d/2}}w'_{d/2}\phi$.
Let $k=\lfloor\frac{2(n+1)}d\rfloor$; then $W_{I_{d/2}}$ is of type
$A_{n-kd/2}$.  If $k$ is even  then $w'_{d/2}=w_{kd/2}^{k/2}$ and the coset
is  of  type  $\lexp  2A_{n-kd/2}$.  Since  $n-kd/2<d/2-1$,  this coset has
$\zeta_d$-rank  0. Finally if  $k$ is odd $w'_{d/2}\phi=(w'_{kd/2}\phi)^k$.
Since  $kd/2$  is  odd,  we  found  that  $w'_{kd/2}\phi$ acts trivially on
$I_{d/2}$  so  the  coset  is  of  type  $A_{n-kd/2}$,  and  has  also  has
$\zeta_d$-rank 0.

$\bullet$  Determine the group $W(w\phi)$ (resp.\ $W(w)$) in each case,
We first give $V_1$ and the coset $C_W(V_1)w\phi$ or $C_W(V_1)w$. In
the  untwisted case $w_d$  acts trivially on  the roots of $W_{I_d}$, hence
$V_1$  is  spanned  by  these  roots  and  $C_W(V_1)$  is  generated by the
reflection with respect to the roots orthogonal to those, which gives that
$C_W(V_1)$  is of type  $A_{d\lfloor\frac{n+1}d\rfloor-1}$ if $d\not|n$ and
$A_n$ otherwise. In the twisted case if $d\equiv 0\pmod 4$ since $w_d$ acts
trivially  on the roots of $W_{I_d}$ the space $V_1$ is spanned by the sums
of  the  orbits  of  the  roots  under  $\phi$  which  is  the  non-trivial
automorphism   of  that   root  system.   Hence  the   type  of  the  coset
$C_W(V_1)w_d\phi$  is $\lexp  2A_{d\lfloor\frac{n+1}d\rfloor-1}$ if  $n$ is
odd  and $\lexp 2A_{d\lfloor\frac{n+1}d\rfloor}$ if $n$  is even. If $d$ is
odd a similar computation gives that the type of the coset
$C_W(V_1)w_{2d}^2\phi$ is $\lexp 2A_{2d \lfloor\frac{n+1}{2d}\rfloor-1}$ if
$n$ is odd and $\lexp 2A_{2d \lfloor\frac{n+1}{2d}\rfloor}$ if $n$ is even.
If   $d\equiv  2\pmod  4$  $w'_{d/2}\phi$  acts  also  by  the  non-trivial
automorphism on $W_{I_{d/2}}$ and we get that the coset
$C_W(V_1)w'_{d/2}\phi$ is of type $\lexp 2A_{\frac
d2\lfloor\frac{2(n+1)}d\rfloor}$  if $n$ and $\lfloor\frac{2(n+1)}d\rfloor$
have the same parity and $\lexp 2A_{\frac
d2\lfloor\frac{2(n+1)}d\rfloor-1}$ otherwise.

Knowing  the type of the coset in each case, we deduce the group $W(w\phi)$
(resp.\ $W(w)$) as in the remark at the beginning of Subsection
\ref{irreducible cosets}.
\end{proof} 
\subsection*{Type $B_n$
$\nnode{s_1}
{\rlap{\vrule width10pt height2pt depth-1pt}
\vrule width10pt height4pt depth-3pt} 
\nnode{s_2}\edge\nnode{s_3}\cdots\nnode{s_n}$}
For $d$ even, $2\leq d\leq 2n$ we define
$$v_d=s_{n+1-d/2}\dotsm s_2s_1s_2\dotsm s_n
\text{ and }J_d=\{s_i\mid 1\le i\le n-d/2\}.$$
Note  that $v_{2n}$ is the Coxeter element $s_1s_2\dotsm s_n$. Now for
$1\leq  d  \leq  2n$,  that  we  require  even if $d>n$, we define $w_d$ as
follows: let $kd$ be the largest even multiple of $d$ less than or equal to
$2n$   so   that   $k=\lfloor\frac{2n}d\rfloor$   if   $d$   is   even  and
$k=2\lfloor\frac  nd\rfloor$ is $d$ is  odd. We define $w_d=v_{kd}^{k}$ and
$I_d=J_{kd}$.

\begin{theorem}\label{good for B_n}
For $W=W(B_n)$, $d$-good maximal elements exist for odd $d$ less
than or equal to $n$ and even $d$ less than or equal to $2n$. A
representative is $w_d$, with $I=I_d$; we have
$W(w_d)=G(d,1,\lfloor\frac {2n}d\rfloor)$  if $d$ is even and
$W(w_d)=G(2d,1,\lfloor\frac nd\rfloor)$ if $d$ is odd.
\end{theorem}
\begin{proof}
We identify as usual the Weyl group of type $B_n$ with the group of signed
permutations on $\{1,\dots,n\}$ by $s_i\mapsto (i-1,i)$ for $i\geq 2$
and $s_1\mapsto (1,-1)$. 
The element $v_d$ maps to the $d$-cycle (or signed $d/2$-cycle) given by
$(n+1-d/2,n+2-d/2,\dots,n-1,n,d/2-n-1,d/2-n-2,\dots,-n)$. This element
normalizes $J_d$ and acts trivially on the corresponding roots, so is 
$J_d$-reduced. The same is thus true for $w_d$ and $I_d$, since these
properties carry to powers.
\begin{lemma}
For $0\leq i\leq \lfloor \frac d2\rfloor$ we have
$l(v_d^i)=\frac {2i}dl(w_{J_d}\inv w_0)$ and
$l(w_d^i)=\frac {2i}dl(w_{I_d}\inv w_0)$.
\end{lemma}
\begin{proof}
As in Lemma \ref{lengths add} it is sufficient to prove the lemma for
$v_d$,  which  we  do  now.  To  find  the  length  of  $v_d$  we note that
$s_1s_2\dotsm  s_n$ is $\{s_2,s_3,\dots,s_n\}$-reduced so that the lengths
of   $s_{n+1-d/2}\dotsm s_2$   and  of   $s_1s_2\dotsm s_n$  add,  whence
$l(v_d)=2n-d/2$.  Since  $l(w_0)=n^2$  and  $l(w_{J_d})=(n-d/2)^2$  we have
$l(w_{I_d}\inv  w_0)= nd-d^2/4$, which gives  the result for $i=1$. Written
as  permutations $w_0$ is the product of  all sign changes and $w_{I_d}$ is
the  product of all sign changes on the set $\{1,\dots,n-d/2 \}$; a direct
computation  shows that $v_d^{d/2}$  is the product  of all sign changes on
$\{n+1-d/2,\dots,n\}$,   hence  $v_d^{d/2}=w_{I_d}\inv   w_0$.  The  lemma
follows for the other values of $d$.
\end{proof}
Since  $v_d^{d/2}=w_{I_d}\inv   w_0$ we have $v_d^d=1$, so the same property is
true  for $w_d$, thus the above lemma shows that $v_d$ and $w_d$ are
$d$-good elements.

Note also that Theorem \ref{good  for B_n}  describes all $d$ such that $W$
has non-zero
$\zeta_d$-rank since the degrees of $W(B_n)$ are all the even integers from
2 to $2n$. We prove now the maximality property \ref{good-zeta-maximal}(iv)
for  $w_d$. If $k$ is as in the definition of $w_d$, the group $W_{I_d}$ is
a Weyl group of type $B_{n-kd/2}$ and $w_d$ acts trivially on $I_d$. Since
$n-kd/2<d$  the  $\zeta_d$-rank  of  $W_{I_d}w_d$  is  zero on the subspace
spanned by the roots corresponding to $I_d$.

It remains to get the type of $W(w_d)$. Since $w_d$ acts trivially on $I_d$
the  space $V_1$ of Lemma \ref{regular in parabolic} is spanned by the root lines
of $W_{I_d}$ and $C_W(V_1)$ is spanned by the roots orthogonal to those, so
is  of type $B_{kd/2}$. We then deduce  the group $W(w_d)$ as in the remark
at the beginning of Subsection \ref{irreducible cosets}, as the centralizer
of a $\zeta_d$-regular element in a group of type $B_{kd/2}$.
\end{proof}
\subsection*{Types $D_n$ and $\protect\lexp 2D_n$
$\nnode{s_1}\edge\vertbar{s_3}{s_2}\edge\nnode{s_4}\cdots\nnode{s_n}$}
$\lexp 2D_n$ is defined by the diagram automorphism $\phi$ which exchanges 
$s_1$ and $s_2$ and fixes $s_i$ for $i>2$.

For $d$ even, $2\le d\le 2(n-1)$ we define
$$v_d=s_{n+1-d/2}\dotsm s_3s_2s_1s_3\dotsm s_n\text{ and }
J_d=\begin{cases}
\emptyset \text{ if } d=2(n-1)\\
\{s_i\mid  1\le i\le n-d/2\}\text{ otherwise.}
\end{cases}$$
Note that $v_{2(n-1)}$ is a Coxeter element. 
Then  for $1\le d\le 2(n-1)$, that we require even if $d>n$, we let $kd$ be
the largest even multiple of $d$ less than $2n$, so that
$k=\lfloor\frac{2n-2}d\rfloor$ if $d$ is even and
$k=2\lfloor\frac{n-1}d\rfloor$ if $d$ is odd, and define $w_d=v_{kd}^k$ and
$I_d=J_{kd}$.

Note that $v_d$, and thus $w_d$, are $\phi$-stable.
\begin{theorem}\label{good for D_n}
\begin{itemize}
\item
For $W=W(D_n)$ there exist $d$-good maximal elements for odd $d$ less
than or equal to $n$ and even $d$ less than or equal to $2(n-1)$. When
$d$ does not divide $n$ a
representative is $w_d$, with $I=I_d$; in this case, if $d$ is odd
$W(w_d)=G(2d,1,\lfloor\frac{n-1}d\rfloor)$ and if $d$ is even
$W(w_d)=G(d,1,\lfloor\frac{2n-2}d\rfloor)$.

If $d|n$ a representative is $w_n^{n/d}$ where $w_n=s_1s_2s_3\dotsm
s_ns_2s_3\dotsm s_{n-1}$. In this case $I=\emptyset$ and
$W(w_n^{n/d})=G(2d,2,n/d)$.
\item
For $W\phi$ there exist $d$-good maximal elements for odd $d$ less than
$n$, for even $d$ less than $2(n-1)$ and for $d=2n$. 
Except in the case when $d$ divides $2n$ and $2n/d$ is odd
a representative is $w_d\phi$, with
$I=I_d$ and $W(w_d\phi)=G(2d,1,\lfloor\frac{n-1}d\rfloor)$ if $d$ is odd
and $W(w_d\phi)=G(d,1,\lfloor\frac{2n-2}d\rfloor)$ if $d$ is even.
In the excluded case a
representative is $(w_{2n}\phi)^{2n/d}$ where $w_{2n}=s_1s_3s_4\dotsm
s_n$. In this case $I=\emptyset$ and $W((w_{2n}\phi)^{2n/d})=G(d,2,2n/d)$.
\end{itemize}
\end{theorem}
\begin{proof}
The  cases $D_n$  with $d|n$  or $\lexp  2D_n$ with  $d|2n$ and  $2n/d$ odd
involve  regular elements, so are dealt  with in \cite{BM}. We thus consider
only the other cases.

We identify the Weyl group of type $D_n$ with the group of signed
permutations  on $\{1,\dots,n\}$  with an even  number of sign changes, by
mapping  $s_i$ to $(i-1,i)$ for $i\neq  2$ and $s_2$ to $(1,-2)(-1,2)$. For
$d$ even $v_d$ maps to $(1,-1)
(n+1-d/2,n+2-d/2,\dots,n-1,n,d/2-n-1,\dots,1-n,-n)$. This element
normalizes  $J_d$: when  $J_d\neq\emptyset$, it  exchanges the simple roots
corresponding  to $s_1$  and $s_2$  and acts  trivially on the other simple
roots  indexed by $J_d$, so it is  $J_d$-reduced. It follows that $w_d$
normalizes $I_d$ and is $I_d$-reduced.
\begin{lemma}
For $0\leq i\leq \lfloor \frac d2\rfloor$ we have
$l(v_d^i)=\frac {2i}dl(w_{J_d}\inv w_0)$ and
$l(w_d^i)=\frac {2i}dl(w_{I_d}\inv w_0)$.
\end{lemma}
\begin{proof}
As  in Lemma  \ref{lengths add}  it is  sufficient to  prove the  lemma for
$v_d$. To find the length of $v_d$ we note that $s_2s_1s_3s_4\dotsm s_n$ is
$\{s_3,\dots,s_n\}$-reduced so that the lengths of $s_{n+1-d/2}\dotsm s_3$
and   of  $s_2s_1s_3\dotsm  s_n$  add,   whence  $l(v_d)=2n-1-d/2$.  Since
$l(w_0)=n^2-n$  and $l(w_{J_d})=(n-d/2)^2-(n-d/2)$,  we have $l(w_{J_d}\inv
w_0)=d/2(2n-1-d/2)$.   which  gives  the  result   for  $i=1$.  Written  as
permutations $w_0=(1,-1)^n(2,-2)\dotsm(n,-n)$ and
$w_{J_d}=(1,-1)^{n-d/2}(2,-2)\dotsm(n-d/2,d/2-n)$;   a  direct  computation
shows   that  $v_d^{d/2}=(1,-1)^{d/2}(n+1-d/2,d/2-n-1)\dotsm(n,-n)$,  hence
$v_d^{d/2}=w_{J_d}\inv w_0$. The lemma follows for smaller $i$.
\end{proof}
Since  $v_d^{d/2}=w_{J_d}\inv w_0$ and $J_d$ is $w_0$ stable we have
$v_d^d=1$,  so the same  property follows for  $w_d$ which shows that $v_d$
and $w_d$ are $d$-good elements.

We  also note that the theorem describes  all $d$ such that the $\zeta_d$-rank
is not zero, since the degrees of $W(D_n)$ are all the even integers from 2 to
$2n-2$  and $n$, and in  the twisted case the  factor associated with the degree
$n$ is -1 and the other factors are equal to 1.

Since $w_d$ is $\phi$-stable the element $w_d\phi$ is also $d$-good.

We  now check Lemma \ref{good-zeta-maximal}(iv), that  is that the $\zeta_d$-rank
of  $W_{I_d}w_d$  in  the  untwisted  case,  resp.\ $W_{I_d}w_d\phi$ in the
twisted  case is $0$ on the subspace  spanned by the roots corresponding to
$I_d$. This property is clear if $I_d=\emptyset$. Otherwise:

$\bullet$  In the untwisted case  the type of the  coset is $D_{n-kd/2}$ if
$k$  is even and $\lexp 2 D_{n-kd/2}$ if $k$ is odd, where $k$ is as in the
definition  of $w_d$.  In both  cases the  set of  values $i$ such that the
$\zeta_i$-rank  is not $0$ consists of the  even $i$ less than $2n-kd$, the
odd  $i$ less than  $n-kd/2$ and in  the twisted case  ($k$ odd) $i=2n-kd$.
Since  if $d$  is even  we have  $2n-kd\leq d$  and if  $d$ is  odd we have
$n-kd/2\leq d$, the only case where $d$ could be in this set is $k$ odd and
$d=2n-kd$,  which means  that $\frac{k+1}2d=n$.  But $d$  is assumed not to
divide $n$, so this case does not happen.

$\bullet$  In the twisted case the type of the coset is $D_{n-kd/2}$ if $k$
is  odd and $\lexp 2 D_{n-kd/2}$  if $k$ is even. In  both cases the set of
values  $i$ such that the $\zeta_i$-rank is not $0$ consists of the even $i$
less  than $2n-kd$, the odd $i$ less  than $n-kd/2$ and in the twisted case
($k$ even) $i=2n-kd$. Since if $d$ is even we have $2n-kd\leq d$ and if $d$
is odd we have $n-kd/2\leq d$, the only case where $d$ could be in this set
is  $k$  even  and  $d=2n-kd$,  which  means  that $(k+1)d=2n$. But this is
precisely the excluded case.

We now give $C_W(V_1)$, where $V_1$ is as in Lemma \ref{regular in parabolic}, in
each case where $I$ is not empty. In the untwisted case, if $d$ is odd the
group  $C_W(V_1)$ is  of type  $D_{d\lfloor\frac{n-1}d\rfloor}$; if  $d$ is
even the group $C_W(V_1)$ is of type $D_{\frac
d2\lfloor\frac{2n-2}d\rfloor+1}$ if $\lfloor\frac{2n-2}d\rfloor$ is odd and
$D_{\frac d2\lfloor\frac{2n-2}d\rfloor}$ if $\lfloor\frac{2n-2}d\rfloor$ is
even.  In the twisted case,  if $d$ is odd  the coset $C_W(V_1)w\phi$ is of
type $\lexp 2D_{d\lfloor\frac{n-1}d\rfloor+1}$ and if $d$ is even the coset
is    of    type    $\lexp2D_{\frac   d2\lfloor\frac{2n-2}d\rfloor+1}$   if
$\lfloor\frac{2n-2}d\rfloor$ is even and $D_{\frac
d2\lfloor\frac{2n-2}d\rfloor}$  if $\lfloor\frac{2n-2}d\rfloor$  is odd. In
all  cases except if  $d$ is even  and $\lfloor\frac{2n-2}d\rfloor$ is even
(resp.\ odd) in the untwisted case (resp.\ twisted case) we then deduce the
group   $W(w\phi)$ (resp.\ $W(w)$) as  in  the  remarks  at  the  beginning 
of  Subsection \ref{irreducible   cosets}  and   after  Lemma
\ref{regular  in  parabolic},  since  in  these  cases  the centralizer of the
regular  element $w\phi$ (resp.\ $w$)  in the parabolic subgroup $W'=C_W(V_1)$
has  the  (known)  reflection  degrees  of  $W(w\phi)$ (resp.\ $W(w)$). In the
excluded  cases  the  group  $C_{W'}(w\phi)$  or  $C_{W'}(w)$ is isomorphic to
$G(d,2,\lfloor\frac{2n-2}d\rfloor)$ which does not have the reflection degrees
of $W(w\phi)$, resp.\ $W(w)$. This means that the morphism of the remark after
Lemma  \ref{regular in parabolic} is not surjective. We can prove in this case
that  $W(w\phi)$ or $W(w)$ is  $G(d,1,\lfloor\frac{2n-2}d\rfloor)$ since it is
an irreducible complex reflection group by \cite[5.6.6]{Br} and it is the only
one  which has the right reflection degrees apart from  the exceptions in low
rank  given by  $G_5,G_{10},G_{15},G_{18},G_{26}$; we  can exclude these since
they do not have $G(d,2,\lfloor\frac{2n-2}d\rfloor)$ as a reflection subgroup.
\end{proof}
\subsection*{Types $I_2(n)$ and $\protect\lexp 2I_2(n)$}
All eigenvalues $\zeta$ such that the $\zeta$-rank is non-zero are regular, so
this case can be found in \cite{BM}.
\subsection*{Exceptional types}
Below  are tables for exceptional finite  Coxeter groups giving information on
$d$-good maximal elements for each $d$. They were obtained with the {\tt
GAP} package {\tt Chevie} (see \cite{Chevie}): first, the conjugacy class of 
good $\zeta_d$-maximal  elements as described in Lemma
\ref{zeta maximal are conjugate} was determined; then we
determined  $I$ for an  element of that  class, which gave  $l(w_I)$. The next
step was to determine the elements of the right length $2(l(w_0)-l(w_I))/d$ in
that  conjugacy class; this required care in large groups like $E_8$. The best
algorithm is to start from an element of minimal length in the class (known by
\cite{geck-pfeiffer}) and conjugate by Coxeter generators until all elements of 
the right length are reached.

In  the following tables, we give for  each possible $d$ and each possible $I$
for  that $d$ a representative  good $w\phi$, and give  the number of possible
$w\phi$. We then describe the coset $W_Iw\phi$ by giving, if $I\ne\emptyset$,
in the column $I$ the
permutation  induced by $w\phi$ of the nodes of the Coxeter diagram indexed by
$I$. Then  we describe  the isomorphism  type of  the complex reflection
group $N_W(W_Iw\phi)/W_I=N_W(V)/C_W(V)$, where $V$ is the $\zeta_d$-eigenspace
of $w\phi$.
Finally, in the cases where $I\ne\emptyset$, we give the isomorphism type
of $W'=C_W(V_1)$, where $V_1$ is the $1$-eigenspace of $w\phi$ on the subspace
spanned by the root lines of $I$. We note that there are 3 cases where
$N_{W'}(V)/C_{W'}(V)\lneq N_W(V)/C_W(V)$: for 
$d=4$ or $5$ in $E_7$ and for $d=9$ in $E_8$.

$H_3$: $\nnode1\overbar5\nnode2\edge\nnode3$
The reflection degrees are $2,6,10$.
$$
\begin{array}{|l|ccc|}
\hline 
d&\text{representative }w&\#\text{good }w&C_W(w)\\
\hline 
10&w_{10}={123}&4&Z_{10}\\
6&w_6={32121}&6&Z_6\\
5&w_{10}^2&4&Z_{10}\\
3&w_6^2&6&Z_6\\
2&w_0&1&H_3\\
1&\cdot&1&H_3\\
\hline 
\end{array}
$$
$H_4$: $\nnode1\overbar5\nnode2\edge\nnode3\edge\nnode4$
The reflection degrees are $2,12,20,30$.
$$
\begin{array}{|l|ccc|}
\hline 
d&\text{representative }w&\#\text{good }w&C_W(w)\\
\hline 
30&w_{30}={1234}&8&Z_{30}\\
20&w_{20}={432121}&12&Z_{20}\\
15&w_{30}^2&8&Z_{30}\\
12&w_{12}={2121432123}&22&Z_{12}\\
10&w_{30}^3\text{ or }w_{20}^2&24&G_{16}\\
6&w_{30}^5\text{ or }w_{12}^2&40&G_{20}\\
5&w_{30}^6\text{ or }w_{20}^4&24&G_{16}\\
4&w_{20}^5\text{ or }w_{12}^3&60&G_{22}\\
3&w_{30}^{10}\text{ or }w_{12}^4&40&G_{20}\\
2&w_0&1&H_4\\
1&\cdot&1&H_4\\
\hline 
\end{array}
$$
$\lexp 3D_4$: $\nnode1\edge\vertbar32\edge\nnode4$
$\phi$ does the permutation $(1,2,4)$.
The reflection degrees are $2,4,4,6$ with corresponding factors
$1,\zeta_3,\zeta_3^2,1$.
$$
\begin{array}{|l|ccc|}
\hline 
d&\text{representative }w\phi&\#\text{good }w\phi&C_W(w\phi)\\
\hline 
12&w_{12}\phi={13}\phi&6&Z_4\\
6&w_6\phi={1243}\phi&8&G_4\\
3&w_6^2\phi&8&G_4\\
2&w_0\phi&1&G_2\\
1&\phi&1&G_2\\
\hline 
\end{array}
$$
$F_4$:
$\nnode1\edge\nnode2{\rlap{\vrule width10pt height2pt depth-1pt}
\vrule width10pt height4pt depth-3pt}\nnode3\edge\nnode4$
The reflection degrees are $2,6,8,12$.
$$
\begin{array}{|l|ccc|}
\hline 
d&\text{representative }w&\#\text{good }w&C_W(w)\\
\hline 
12&w_{12}={1234}&8&Z_{12}\\
8&w_8={214323}&14&Z_8\\
6&w_{12}^2&16&G_5\\
4&w_{12}^3\text{ or }w_8^2&12&G_8\\
3&w_{12}^4&16&G_5\\
2&w_0&1&F_4\\
1&\cdot&1&F_4\\
\hline 
\end{array}
$$
$\lexp 2F_4$:
$\phi$ does the permutation $(1,4)(2,3)$.
The factors, in increasing order of the degrees, are $1,-1,1,-1$.
$$
\begin{array}{|l|ccc|}
\hline 
d&\text{representative }w\phi&\#\text{good }w\phi&C_W(w\phi)\\
\hline 
24&w_{24}\phi={12}\phi&6&Z_{12}\\
12&w_{12}\phi={3231}\phi&10&Z_6\\
8&(w_{24}\phi)^3&12&G_8\\
4&(w_{12}\phi)^3&24&G_{12}\\
2&w_0\phi&1&I_2(8)\\
1&\phi&1&I_2(8)\\
\hline 
\end{array}
$$
$E_6$: $\nnode1\edge\nnode3\edge\vertbar42\edge\nnode5\edge\nnode6$
The reflection degrees are $2,5,6,8,9,12$.
$$
\begin{array}{|l|ccccc|}
\hline 
d&\text{representative }w&\#\text{good }w&I&N_W(W_Iw)/W_I&C_W(V_1)\\
\hline 
12&w_{12}={123654}&8&&Z_{12}&\\
9&w_9={12342654}&24&&Z_9&\\
8&w_8={123436543}&14&&Z_8&\\
6&w_{12}^2&16&&G_5&\\
5&24231454234565&8&(3)&Z_5&A_5\\
&12435423456543&8&(4)&&\\
&12314235423654&8&(5)&&\\
4&w_8^2\text{ or }w_{12}^3&12&&G_8&\\
3&w_{12}^4\text{ or }w_{9}^3&80&&G_{25}&\\
2&w_0&1&&F_4&\\
1&\cdot&1&&E_6&\\
\hline 
\end{array}
$$
$\lexp 2E_6$:
$\phi$ does the permutation $(1,6)(3,5)$.
The factors, in increasing order of the degrees, are $1,-1,1,1,-1,1$.
$$
\begin{array}{|l|ccccc|}
\hline 
d&\text{representative }w\phi&\#\text{good}w\phi&I&N_W(W_Iw\phi)/W_I&C_W(V_1)w\phi\\
\hline 
18&w_{18}\phi={1234}\phi&24&&Z_{9}&\\
12&w_{12}\phi={123654}\phi&8&&Z_{12}&\\
10&2431543\phi&8&(3)&Z_5&\lexp 2A_5\\
&5423145\phi&8&(4)&&\\
&3143542\phi&8&(5)&&\\
8&w_8\phi={123436543\phi}&14&&Z_8&\\
6&(w_{18}\phi)^3&80&&G_{25}&\\
4&(w_{12}\phi)^3&12&&G_8&\\
3&w_{12}^4\phi&16&&G_5&\\
2&w_0\phi&1&&E_6&\\
1&\phi&1&&F_4&\\
\hline 
\end{array}
$$
$E_7$: $\nnode1\edge\nnode3\edge\vertbar42\edge\nnode5\edge\nnode6\edge\nnode7$
The reflection degrees are $2,6,8,10,12,14,18$.
$$
\begin{array}{|l|ccccc|}
\hline 
d&\text{representative }w&\#\text{good }w&I&N_W(W_Iw)/W_I&C_W(V_1)\\
\hline 
18&w_{18}={1234567}&64&&Z_{18}&\\
14&w_{14}={123425467}&160&&Z_{14}&\\
12&w_{12}={1342546576}&8&(2,5,7)&Z_{12}&E_6\\
10&w_{10a}={134254234567}&8&(2,4)&Z_{10}&D_6\\
&w_{10b}={243154234567}&8&(3,4)&&\\
&w_{10c}={124354265437}&8&(4,5)&&\\
9&w_{18}^2&64&&Z_{18}&\\
8&134234542346576&14&(2)(5,7)&Z_8&D_5\\
7&w_{14}^2&160&&Z_{14}&\\
6&w_{18}^3\text{ or }w_{12}^2&800&&G_{26}&\\
5&w_{10a}^2&8&(2)(4)&Z_{10}&A_5\\
&w_{10b}^2&8&(3)(4)&&\\
&w_{10c}^2&8&(4)(5)&&\\
4&w_8^2\text{ or }w_{12}^3&12&(2)(5)(7)&G_8&D_4\\
3&w_{18}^6\text{ or }w_{12}^4&800&&G_{26}&\\
2&w_0&1&&E_7&\\
1&\cdot&1&&E_7&\\
\hline 
\end{array}
$$
$E_8$:
$\nnode1\edge\nnode3\edge\vertbar42\edge\nnode5\edge\nnode6\edge\nnode7\edge\nnode8$
The reflection degrees are $2,8,12,14,18,20,24,30$.
$$
\begin{array}{|l|ccccc|}
\hline 
d&\text{representative }w&\#\text{good }w&I&N_W(W_Iw)/W_I&C_W(V_1)\\
\hline 
30&w_{30}={12345678}&128&&Z_{30}&\\
24&w_{24}={1234254678}&320&&Z_{24}&\\
20&w_{20}={123425465478}&624&&Z_{20}&\\
18&w_{18a}={1342542345678}&16&(2,4)&Z_{18}&E_7\\
&w_{18b}={2431542345678}&16&(3,4)&&\\
&w_{18c}={1243542654378}&16&(4,5)&&\\
15&w_{30}^2&128&&Z_{30}&\\
14&w_{14a}={13423454234565768}&128&(2)&Z_{14}&E_7\\
&w_{14b}={24231454234565768}&88&(3)&&\\
&w_{14c}={12435423456543768}&108&(4)&&\\
&w_{14d}={12342543654276548}&68&(5)&&\\
12&w_{24}^2&2696&&G_{10}&\\
10&w_{30}^3\text{ or }w_{20}^2&3370&&G_{16}&\\
9&w_{18a}^2&16&(2)(4)&Z_{18}&E_6\\
&w_{18b}^2&16&(3)(4)&&\\
&w_{18c}^2&16&(4)(5)&&\\
8&w_{24}^3&7748&&G_9&\\
7&w_{14a}^2&128&(2)&Z_{14}&E_7\\
&w_{14b}^2&88&(3)&&\\
&w_{14c}^2&108&(4)&&\\
&w_{14d}^2&68&(5)&&\\
6&w_{30}^5\text{ or }w_{24}^4&4480&&G_{32}&\\
5&w_{30}^6\text{ or }w_{20}^4&3370&&G_{16}&\\
4&w_{24}^6\text{ or }w_{20}^5&15120&&G_{31}&\\
3&w_{30}^{10}\text{ or }w_{24}^8&4480&&G_{32}&\\
2&w_0&1&&E_8&\\
1&\cdot&1&&E_8&\\
\hline 
\end{array}
$$

\end{document}